\documentclass[11pt]{article}
\usepackage{smile}

\usepackage{fullpage}
\usepackage{lscape}
\usepackage{bigints}
\usepackage{upgreek}
\usepackage{framed}
\usepackage{mdframed}
\usepackage{enumerate}
\usepackage[inline]{enumitem}
\usepackage[T1]{fontenc}
\usepackage{moresize}
\usepackage{bm}
\usepackage{bbm}
\usepackage{dsfont}
\usepackage{amsmath}
\usepackage{amssymb}
\usepackage{amsthm}
\usepackage{amsfonts}
\usepackage{stmaryrd}
\usepackage{array}
\usepackage{mathrsfs}
\usepackage{mathtools} 
\usepackage{extarrows}
\usepackage{stackrel}
\usepackage{relsize,exscale}
\usepackage{scalerel}
\usepackage[nodisplayskipstretch]{setspace}
\usepackage{color}
\usepackage[usenames,dvipsnames]{xcolor}
\usepackage{cancel}
\usepackage{soul}
\usepackage{xfrac}
\usepackage{siunitx}
\usepackage{graphicx}
\usepackage{float}
\usepackage{rotating}
\usepackage{subcaption}
\usepackage{overpic}
\usepackage[all]{xy}
\usepackage{tikz}
\usetikzlibrary{arrows,matrix,positioning,calc,automata,patterns}
\usepackage{booktabs}
\usepackage{dcolumn}
\usepackage{multirow}
\usepackage{diagbox}
\usepackage{tabularx}
\usepackage{verbatim}
\usepackage{listings}
\usepackage[ruled,vlined]{algorithm2e}
\usepackage{fancyvrb}
\usepackage{hyperref}
\usepackage[round]{natbib}
\usepackage{sectsty}
\usepackage{etoolbox}
\usepackage[capitalize,noabbrev]{cleveref}
\crefname{assumption}{Assumption}{Assumptions}
\newcounter{step}
\AtBeginEnvironment{proof}{\setcounter{step}{0}}
\crefname{step}{Step}{Steps}

\hypersetup{
    bookmarks=true,         
    unicode=false,          
    pdftoolbar=true,        
    pdfmenubar=true,        
    pdffitwindow=false,     
    pdfstartview={FitH},    
    pdftitle={My title},    
    pdfauthor={Author},     
    pdfsubject={Subject},   
    pdfcreator={Creator},   
    pdfproducer={Producer}, 
    pdfkeywords={key1, key2}, 
    pdfnewwindow=true,      
    colorlinks=true,        
    linkcolor=blue,         
    citecolor=blue,         
    filecolor=blue,         
    urlcolor=cyan           
}

\def\scaleint#1{\vcenter{\hbox{\scaleto[3ex]{\displaystyle\int}{#1}}}}

\usepackage{stackengine}
\stackMath
\newcommand\tenq[2][1]{%
\def\useanchorwidth{T}%
\ifnum#1>1%
\stackunder[0pt]{\tenq[\numexpr#1-1\relax]{#2}}{\!\scriptscriptstyle\thicksim}%
\else%
\stackunder[1pt]{#2}{\!\scriptstyle\thicksim}%
\fi%
}

\makeatletter
\DeclareRobustCommand\widecheck[1]{{\mathpalette\@widecheck{#1}}}
\def\@widecheck#1#2{%
    \setbox\z@\hbox{\m@th$#1#2$}%
    \setbox\tw@\hbox{\m@th$#1%
       \widehat{%
          \vrule\@width\z@\@height\ht\z@
          \vrule\@height\z@\@width\wd\z@}$}%
    \dp\tw@-\ht\z@
    \@tempdima\ht\z@ \advance\@tempdima2\ht\tw@ \divide\@tempdima\thr@@
    \setbox\tw@\hbox{%
       \raise\@tempdima\hbox{\scalebox{1}[-1]{\lower\@tempdima\box
\tw@}}}%
    {\ooalign{\box\tw@ \cr \box\z@}}}
\makeatother

\newcommand{\Dc}{\mathcal{D}}

\newcommand{\Nc}{\mathcal{N}}

\newcommand{\var}{\mathrm{Var}}

\newcommand{\Step}[2][]{%
  \refstepcounter{step}%
  \paragraph*{Step \thestep: #2.}%
  \if\relax\detokenize{#1}\relax\else\label{#1}\fi
}

\numberwithin{equation}{section}

\newtheorem{theorem}{Theorem}[section]
\newtheorem{lemma}{Lemma}[section]
\newtheorem{proposition}{Proposition}[section]
\newtheorem{assumption}{Assumption}[section]
\newtheorem{corollary}{Corollary}[section]

\newtheorem{innercustomgeneric}{\customgenericname}
\providecommand{\customgenericname}{}
\newcommand{\newcustomtheorem}[2]{%
  \newenvironment{#1}[1]
  {%
   \renewcommand\customgenericname{#2}%
   \renewcommand\theinnercustomgeneric{##1}%
   \innercustomgeneric
  }
  {\endinnercustomgeneric}
}
\newcustomtheorem{customdefinition}{Definition}
\newcustomtheorem{customdefinitions}{Definitions}
\newcustomtheorem{customtheorem}{Theorem}
\newcustomtheorem{customassumption}{Assumption}
\newcustomtheorem{customlemma}{Lemma}
\newcustomtheorem{customexample}{Example}
\theoremstyle{definition}

\newtheorem{remark}{Remark}[section]

\usepackage{enumitem}
\makeatletter
\newcommand{\mylabel}[2]{#2\def\@currentlabel{#2}\label{#1}}
\makeatother

\graphicspath{{./fig3/}}

\def\d{{\mathrm d}}

\newcommand{\R}{\mathbb{R}}  

\newcommand{\N}{\mathbb{N}}

\renewcommand{\i}{\mathbf{i}}
\renewcommand{\j}{\mathbf{j}}

\newcommand{\E}{\mathbb{E}}
\renewcommand{\P}{\mathbb{P}}

\renewcommand{\mb}{\mathbb}

\renewcommand\mf{\mathbf}
\renewcommand\ml{\mathcal}
\newcommand\tp{\intercal}
\newcommand\wt{\widetilde}

\newcommand{\rw}{\rightarrow}

\renewcommand{\1}{^{(1)}}
\renewcommand{\2}{^{(2)}}
\renewcommand{\l}{\left}
\renewcommand{\r}{\right}

\newcommand{\be}{\begin{equation}}
\newcommand{\ee}{\end{equation}}
\newcommand{\bea}{\begin{equation}\begin{aligned}}
\newcommand{\eea}{\end{aligned}\end{equation}}
\renewcommand{\tr}{\operatorname{tr}}
\renewcommand{\var}{\operatorname{var}}
\renewcommand\ind{\text{\large 1\hskip-0.29em I}}

\newcommand{\dotheadrightarrow}{%
\tikz[baseline=-0.5ex]{
  \draw[->] (0,0) -- (0.3,0);      
  \fill (0.3,0) circle (0.8pt);    
  \fill (0,0) circle (0.8pt);
}
}

\newcommand{\dotheadrightrightarrow}{%
\tikz[baseline=-0.5ex]{
  \draw[->] (0,0) -- (0.3,0);      
  \fill (0.3,0) circle (0.8pt);    
  \draw[->] (0.3,0) -- (0.6,0);    
  \fill (0.6,0) circle (0.8pt);
  \fill (0,0) circle (0.8pt);
}
}

\newcommand{\dotheadDoubleRightarrow}{%
\tikz[baseline=-0.5ex]{
  \draw[->>] (0,0) -- (0.3,0);      
  \fill (0.3,0) circle (0.8pt);    
  \fill (0,0) circle (0.8pt);
}
}

\newcommand{\dotheadleftrightarrow}{%
\tikz[baseline=-0.5ex]{
  \draw[<-] (0,0) -- (0.3,0);      
  \fill (0.3,0) circle (0.8pt);    
  \draw[->] (0.3,0) -- (0.6,0);    
  \fill (0.6,0) circle (0.8pt);
  \fill (0,0) circle (0.8pt);
}
}

\newcommand{\dotheadrightleftarrow}{%
\tikz[baseline=-0.5ex]{
  \draw[->] (0,0) -- (0.3,0);      
  \fill (0.3,0) circle (0.8pt);    
  \draw[<-] (0.3,0) -- (0.6,0);    
  \fill (0.6,0) circle (0.8pt);
  \fill (0,0) circle (0.8pt);
}
}

\newcommand{\dotheadSingleDoublearrow}{%
\tikz[baseline=-0.5ex]{
  \draw[<->] (0,0) -- (0.3,0);      
  \fill (0.3,0) circle (0.8pt);    
  \fill (0,0) circle (0.8pt);
}
}

\allowdisplaybreaks

\begin{document}

\setlength{\abovedisplayskip}{5pt}
\setlength{\belowdisplayskip}{5pt}
\setlength{\abovedisplayshortskip}{5pt}
\setlength{\belowdisplayshortskip}{5pt}
\hypersetup{colorlinks,breaklinks,urlcolor=blue,linkcolor=blue}

\title{\LARGE Spectral analysis of large dimensional Chatterjee's rank correlation matrix}

\author{Zhaorui Dong\thanks{School of Data Science, The Chinese University of Hong Kong, Shenzhen, China; {\tt zhaoruidong@link.cuhk.edu.cn}},~~~
Fang Han\thanks{Department of Statistics, University of Washington, Seattle, WA 98195, USA; {\tt fanghan@uw.edu}},
	     ~~and~~
Jianfeng Yao\thanks{School of Data Science, The Chinese University of Hong Kong, Shenzhen, China; {\tt jeffyao@cuhk.edu.cn}}	    
}

\date{}

\maketitle

\vspace{-1em}

\begin{abstract}
This paper studies the spectral behavior of large dimensional Chatterjee's rank correlation matrix when observations are independent draws from a high-dimensional random vector with independent continuous components. 
Limits for the empirical spectral distributions of its two symmetrized versions are established in the proportional high-dimensional regime, one of them being   the semicircle law, thereby giving a first example
of a correlation matrix with a non-Marchenko--Pastur spectral limit, in contrast to the Pearson, Kendall, and Spearman cases. We further establish central limit theorems for linear spectral statistics of the symmetrized matrices. As an important application of this theory, we develop  Chatterjee's rank correlation-based tests for the complete independence among the components.
\end{abstract}

 \medskip 
 

\section{Introduction}

Let $\mX=\mX^{(p)}=(X_1,\ldots,X_p)^\top$ be a $p$-dimensional random vector with mutually independent and continuous margins, i.e., the marginal distribution functions  are continuous. Let $\mX_1,\ldots,\mX_n$ be $n$ independent copies of $\mX$, with $\mX_i=(X_{i1},\ldots,X_{ip})^\top$ for $i \in [n]:=\{1,2,\ldots,n\}$. Define the data matrix $\fX=\fX^{(n)}=[\mX_1,\ldots,\mX_n]^\top \in \mathbb{R}^{n \times p}$. Throughout the manuscript, we impose no moment conditions on $\mX$, nor do we require $X_1,\ldots,X_p$ to be identically distributed.  

In this paper, we study the spectral behavior of Chatterjee's rank correlation matrix under the following asymptotic regime:
\begin{align}\label{eq:asymptotics}
p=p(n),\qquad n\rw\infty, \qquad \frac{p}{n} \to \gamma \in (0,\infty).
\end{align}

\subsection{Chatterjee's rank correlation matrix}

The Chatterjee rank correlation is based on the notion of rankings.
For each $j \in [p]$, let $R_j = R_j^{(n)} : [n] \to [n]$ denote the permutation that assigns to each $k \in [n]$ the rank of $X_{kj}$ among the $n$ observations in $\fX_{\cdot j}=\{X_{1j}, \ldots, X_{nj}\}$. Equivalently, using the indicator function $\ind(\cdot)$,
\[
R_j(k) = \sum_{i=1}^n \ind(X_{ij} \leq X_{kj}),
\]
where, by continuity of the marginal distributions, ties occur with probability zero.  

For $i\ne j\in [p]$, Chatterjee's rank correlation \citep{chatterjee2020new} between $\fX_{\cdot i}$ and $\fX_{\cdot j}$ is then defined as
\begin{align}\nonumber
\xi_n(\fX_{\cdot i}, \fX_{\cdot j}) 
:= 1 - \frac{3}{n^2 - 1} \sum_{k=1}^{n-1} 
\Big| (R_j \circ R_i^{-1})(k+1) - (R_j \circ R_i^{-1})(k) \Big|.
\end{align}

Sourav Chatterjee showed that $\xi_n$ is a strongly consistent estimator of the dependence measure introduced by Dette, Siburg, and Stoimenov \citep{MR3024030}:
\[
\xi(X_i, X_j) := 
\frac{\scaleint{4.5ex} \Var\!\Big( \E[\ind(X_j \geq x) \mid X_i] \Big) \,\mathrm{d}F_j(x)}
{\scaleint{4.5ex} \Var\!\big( \ind(X_j \geq x) \big) \,\mathrm{d}F_j(x)},
\]
where $F_j$ denotes the distribution function of $X_j$.

Unlike the population correlation measures of Pearson, Kendall, and Spearman, $\xi(X_i,X_j)$ has the appealing property that it equals $0$ if and only if $X_j$ is independent of $X_i$, and equals $1$ if and only if $X_j$ is a measurable function of $X_i$. Consequently, the estimator $\xi_n$, being consistent for $\xi$, can be used both to construct consistent tests of pairwise independence and to detect functional dependence \citep{bickel2022measures}.

Next we define Chatterjee's rank correlation matrix as
\[
\mXi_n = \Big[\Xi_{ij}^{(n)}\Big]_{i,j \in [p]}, \qquad 
\Xi_{ij}^{(n)} := 
\begin{cases}
1, & \text{if } i = j,\\[2mm]
\xi_n(\fX_{\cdot i}, \fX_{\cdot j}), & \text{if } i \neq j.
\end{cases}
\]

Note that $\xi_n$ is generally asymmetric, so $\Xi_{ij}^{(n)} \neq \Xi_{ji}^{(n)}$ in general.  
We analyze the spectrum of $\mXi_n$ through the following two symmetrizations:
\begin{align}\label{eq:phin}
\mPhi_n := \frac{1}{2}\Big(\mXi_n + \mXi_n^\top \Big),
\end{align}
 and
\begin{align}\label{eq:psin}
\mPsi_n := (\mXi_n - \mf{I}_p)(\mXi_n - \mf{I}_p)^\top,
\end{align}
where $\mf{I}_p$ denotes the $p$-dimensional identity matrix.

\subsection{Motivation}
This work is primarily motivated by the goal of understanding the spectral behavior of large-dimensional Chatterjee's rank correlation matrices. As a natural application of the resulting theory for linear spectral statistics (LSSs), we also investigate high-dimensional independence tests constructed from these matrices.

In the spectral analysis of large-dimensional covariance and correlation matrices, \citet{marvcenko1967distribution} provided the first global spectral analysis of large-dimensional sample covariance matrices. They showed that the empirical spectral distribution (ESD) of the sample covariance matrix converges to the distribution known as the Marchenko-Pastur (M-P) law, denoted by $\operatorname{MP}(y,\sigma^2)$, with probability measure
\[
\mu_{y,\sigma^2}(\d x) = \frac{\sqrt{(b-x)(x-a)}}{2 \pi xy \sigma^2} \ind(a \le x \le b) \, \d x + (1-y^{-1}) \ind(y>1) \, \delta_0(\d x),
\]
where $a = \sigma^2(1-\sqrt{y})^2$, $b = \sigma^2(1+\sqrt{y})^2$, and $\delta_0$ is the Dirac measure at the origin. Subsequently, \citet{jiang2004limiting}, \citet{bai2008large}, and \citet{bandeira2017marvcenko} established analogous ESD convergence results for Pearson's sample correlation, Spearman's rho, and Kendall's tau matrices, respectively.

Interestingly, all of the covariance and correlation matrices mentioned above converge to the M-P law, possibly after some rescaling. In random matrix theory (RMT), another well-known universality result concerns Wigner matrices and the semicircle law. In particular, \citet{wigner1958distribution}, followed by \cite{grenander2008probabilities} and \cite{arnold1971wigner}, showed that the ESD of a normalized $p \times p$ standard Gaussian matrix converges almost surely to the semicircle law, with density
\[
\rho_{u,r}(x) = \frac{2}{\pi r^2} \sqrt{r^2 - (x-u)^2} \, \ind(x\in [u-r,\,u+r]),
\]
denoted by $\mathrm{W}(u,r)$ with $u \in \R$ as the center and $r \in \R^+$ as the radius.  
To date, \emph{no} correlation matrix has been identified whose ESD converges to a semicircle law.

As an application of the LSS theory developed below, we consider   testing the null hypothesis
\begin{align}\label{eq:H0}
H_0: X_1, \ldots, X_p \text{ are mutually independent.}
\end{align}
This is the classical problem of testing the mutual independence of the coordinates of a high-dimensional random vector, and it coincides with the null setting under which the spectral theory in this paper is developed. \citet{yin2023central}, \citet{bao2015spectral}, and \citet{li2021central} have previously constructed LSS tests for this hypothesis using Pearson sample correlation, Spearman's rho, and Kendall's tau matrices. Given the demonstrated ability of Chatterjee's rank correlation to detect certain nontrivial forms of dependence, it is therefore natural to develop a comparable LSS test based on $\Xi_n$.

Therefore, the main purpose of the present work is to establish the spectral theory of large-dimensional symmetrizations $\mPhi_n$ and $\mPsi_n$ of the Chatterjee rank correlation matrix. We show that their ESDs have a limit in the proportional asymptotic regime. In particular, the ESD of $\mPhi_n$  converges to Wigner's semicircle law, thereby giving a first concrete example  of a correlation matrix with a non-M-P spectral limit. We also develop the theory of LSSs for $\mPhi_n$ and $\mPsi_n$ and use the LSS of $\mPsi_n$ to construct Chatterjee rank correlation-based tests of $H_0$ as an application.
Here, $\mPhi_n$ and $\mPsi_n$ serve complementary purposes. The symmetrization  $\mPhi_n$ is the natural object for uncovering the new semicircle-law phenomenon. By contrast, $\mPsi_n=(\mXi_n-\mf{I}_p)(\mXi_n-\mf{I}_p)^\top$ is naturally suited to testing, since
\[
\tr(\mPsi_n)=\sum_{i\ne j}\bigl\{\Xi_{ij}^{(n)}\bigr\}^2
\]
aggregates the nonnegative pairwise dependence signals measured by Chatterjee's rank correlation. Taken together, these results contribute to the growing literature on the statistical properties of Chatterjee's rank correlation and on realizing its broader potential for dependence analysis \citep{azadkia2019simple,deb2020kernel,auddy2021exact,shi2020power,shi2021ac,bickel2022measures,lin2021boosting,lin2022limit,lin2024failure,han2024azadkia,zhang2022asymptotic,bucher2024lack,tran2024rank,ansari2025directextensionazadkia,azadkia2025bias}.

\subsection{Main results}

To state our main results, we first introduce some notation.  
For any real symmetric matrix $\fA \in \R^{p \times p}$, let its eigenvalues be ordered as
\[
\lambda_1(\fA) \geq \lambda_2(\fA) \geq \cdots \geq \lambda_p(\fA).
\]
The empirical spectral distribution (ESD) of $\fA$ is defined by
\[
F^{\fA} := \frac{1}{p} \sum_{i=1}^p \delta_{\lambda_i(\fA)},
\]
where $\delta_x$ denotes the Dirac measure at $x$. The trace of $\fA$ is denoted by  
\[
\tr(\fA) := \sum_{i=1}^p \lambda_i(\fA).
\]

We use the following notation:  
$\N = \{1,2,3,\ldots\}$; for $K \in \N$, $[K] = \{1,2,\ldots,K\}$; for $k,n \in \N$, $(n)_k = n(n-1)\cdots(n-k+1)$.  
Weak convergence of probability measures is denoted by ``$\Rightarrow$''.  
For two random variables $X$ and $Y$, we write ``$X \stackrel{d}{=} Y$'' to indicate that $X$ and $Y$ are identically distributed.

All subsequent results are established under the following data-generating assumption.

\begin{assumption}\label{assump:dgp}
Let $\mX_1,\ldots,\mX_n$ be independent copies of $\mX = (X_1,\ldots,X_p)^\top$, where $X_1,\ldots,X_p$ are mutually independent with continuous distribution functions.
\end{assumption}

Our first result establishes the weak limit of $F^{\mPhi_n}$.

\begin{theorem}[Semicircle law for $\mPhi_n$]\label{thm:sc}  
Under Assumption~\ref{assump:dgp} and the asymptotic regime \eqref{eq:asymptotics}, we have
\[
F^{\mPhi_n} \;\Rightarrow\; \mathrm{W}(1,2\sqrt{\gamma/5}) \quad \text{almost surely,}
\]
where $\mPhi_n$ was introduced in \eqref{eq:phin}.
\end{theorem}

With a minor modification of the proof of Theorem~\ref{thm:sc}, we also obtain the limiting spectral distribution of the deviation matrix $\mPsi_n$, this time following the M-P law due to the Gram matrix structure of $\mPsi_n$.

\begin{theorem}[M-P law for $\mPsi_n$]\label{thm:mp}  
Under Assumption~\ref{assump:dgp} and the asymptotic regime \eqref{eq:asymptotics}, we have
\[
F^{\mPsi_n} \;\Rightarrow\; \mathrm{MP}(1,2\gamma/5) \quad \text{almost surely,}
\]
where $\mPsi_n$ was introduced in \eqref{eq:psin}.
\end{theorem}

Our final results concern the LSSs of $\mPhi_n$ and $\mPsi_n$. In the following, let $a\wedge b$ denote $\min(a,b)$.

\begin{theorem}[CLT for LSS of $\mPhi_n$]\label{thm:clt:Phi}
    Assume Assumption~\ref{assump:dgp} and the asymptotic regime \eqref{eq:asymptotics} hold. As $n\to\infty$,
\[
\Big\{\tr((\mPhi_n-\fI_p)^k)-\E\tr((\mPhi_n-\fI_p)^k)\Big\}_{k=1}^\infty 
\;\Rightarrow\; \{H_k\}_{k=1}^{\infty},
\]
where $\{H_k\}_{k=1}^{\infty}$ is a mean-zero Gaussian process with covariance function $\Cov(H_{k_1},H_{k_2})$
given as follows.
\begin{enumerate}[itemsep=-.5ex,label=(\roman*)]
\item If $k_1,k_2$ are both odd, then

\begin{align*}
\Cov(H_{k_1},H_{k_2})=\l(\frac{\gamma}{5}\r)^{\frac{k_1+k_2}{2}}\sum_{t=1}^{\lfloor (k_1\wedge k_2)/2 \rfloor } 8t^2(2t+&1)\left( \sum_{\ell=0}^{\left\lfloor  k_1/2   \right \rfloor-t} \frac{\binom{k_{1}-2\ell-1}{\left\lfloor  k_1/2   \right \rfloor-t-\ell}\binom{2\ell+1}{\ell+1}}{k_{1}-2\ell-1} \right)\\&\left( \sum_{\ell=0}^{\left\lfloor  k_2/2   \right \rfloor-t} \frac{\binom{k_{2}-2\ell-1}{\left\lfloor  k_2/2   \right \rfloor-t-\ell}\binom{2\ell+1}{\ell+1}}{k_{2}-2\ell-1} \right).
\end{align*}

\item If $k_1,k_2$ are both even, then
\begin{align*}
\Cov(H_{k_1},H_{k_2})=&\l(\frac{\gamma}{5}\r)^{\frac{k_1+k_2}{2}}\Bigg\{4\binom{k_{1}}{k_1/2+1}\binom{k_{2}}{k_2/2+1}\\  +4\sum_{t=2}^{(k_1\wedge k_2)/2  } t&(2t-1)^2\Bigg[\Bigg( \sum_{\ell=0}^{ k_1/2 -t} \frac{\binom{k_{1}-2\ell-1}{  k_1/2  -t-\ell}\binom{2\ell+1}{\ell+1}}{k_{1}-2\ell-1}  \Bigg)\Bigg( \sum_{\ell=0}^{ k_2/2 -t} \frac{\binom{k_{2}-2\ell-1}{  k_2/2  -t-\ell}\binom{2\ell+1}{\ell+1}}{k_{2}-2\ell-1}  \Bigg)\Bigg]\Bigg\}.
\end{align*}

\item Otherwise, $\Cov(H_{k_1},H_{k_2})=0$.
\end{enumerate}
\end{theorem}

Note that when $k_1=k_2=1$, the theorem above gives $\Var(H_1)=0$. This is consistent with the fact that the diagonal entries of $\mf{\Phi}_n$ are all 1 and $\tr(\mf{\Phi}_n-\fI_p)=0$. 

\begin{theorem}[CLT for LSS of $\mPsi_n$]\label{thm:clt}
Assume Assumption~\ref{assump:dgp} and the asymptotic regime \eqref{eq:asymptotics} hold.
\begin{enumerate}[itemsep=-.5ex,label=(\roman*)]
\item As $n\to\infty$,
\[
\Big\{\tr(\mPsi_n^k)-\E\tr(\mPsi_n^k)\Big\}_{k=1}^\infty 
\;\Rightarrow\; \{G_k\}_{k=1}^{\infty},
\]
where $\{G_k\}_{k=1}^{\infty}$ is a mean-zero Gaussian process with covariance function
\begin{align*}\label{eq-covariance-function}
\Cov(G_{k_1},G_{k_2}) 
&= \left(\tfrac{2\gamma}{5}\right)^{k_1+k_2} \Bigg\{ 
  2\binom{2k_1}{k_1+1}\binom{2k_2}{k_2+1} \\
&\quad + 2\sum_{t=2}^{k_1\wedge k_2} t(2t-1)^2 
   \left[\Bigg(\sum_{\ell=0}^{k_1-t} \frac{\binom{2k_1-2\ell-1}{k_1-\ell-t}\binom{2\ell+1}{\ell+1}}{2k_1-2\ell-1}\Bigg)
          \Bigg(\sum_{\ell=0}^{k_2-t} \frac{\binom{2k_2-2\ell-1}{k_2-\ell-t}\binom{2\ell+1}{\ell+1}}{2k_2-2\ell-1}\Bigg)\right]\Bigg\}.
\end{align*}

\item Moreover, the expectations $\E\tr(\mPsi_n^k)$ are distribution-free and can be computed numerically. In the special case $k=1$, we have
\[
\E\tr(\mPsi_n)=\frac{p(p-1)(n-2)(4n-7)}{10(n-1)^2(n+1)}.
\]
\end{enumerate}
\end{theorem}
Note that the above formula of $\E\tr(\mPsi_n)$ for the special case $k=1$ is given in \citet[Lemma 2.2]{xia2025consistent}. Figures~\ref{Figure-LSD-Semicircle} and \ref{Figure-LSD-MP} provide numerical illustrations of Theorems~\ref{thm:sc} and \ref{thm:mp}. Figures~\ref{Figure-CLT-Histograms-Phi} and \ref{Figure-CLT-Histograms-Variance-function} provide numerical illustrations for the CLT results in Theorems~\ref{thm:clt:Phi} and \ref{thm:clt}.

\begin{figure}[H]
    \centering
    \begin{minipage}{0.48\linewidth}
        \centering
        \includegraphics[width=\linewidth]{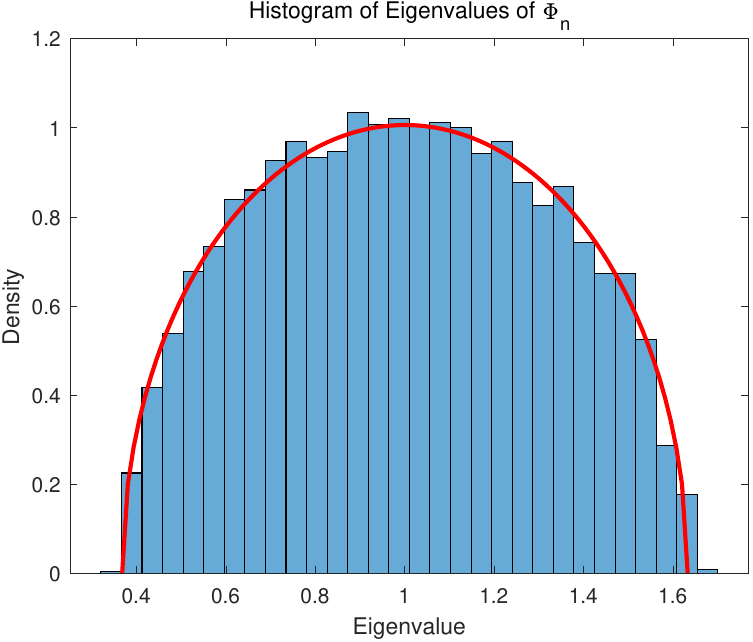}
        \caption{Semicircle law of $\mf{\Phi}_n$ with $p=100$ and $n=200$.}
        \label{Figure-LSD-Semicircle}
    \end{minipage}
    \hfill
    \begin{minipage}{0.48\linewidth}
        \centering
        \includegraphics[width=\linewidth]{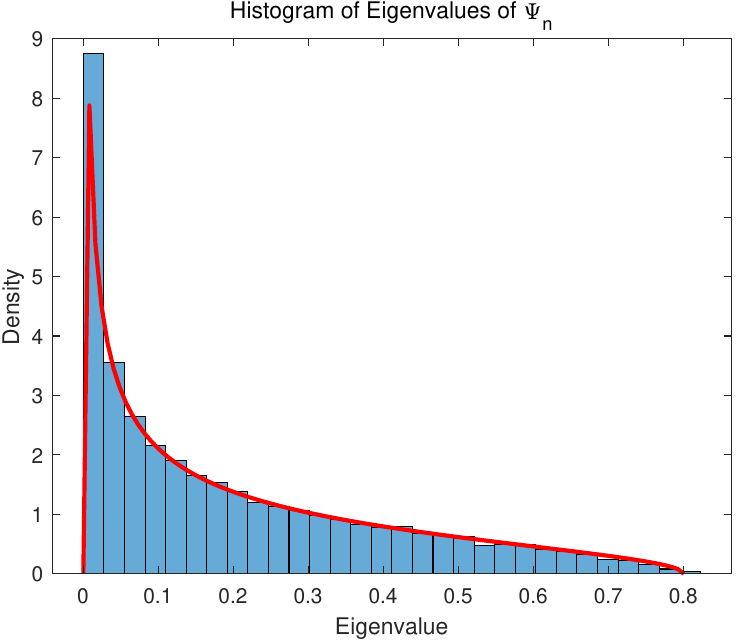}
        \caption{M-P law of $\mPsi_n$ with $p=100$ and $n=200$.}
        \label{Figure-LSD-MP}
    \end{minipage}
\end{figure}

\begin{figure}[H]
    \centering
    \includegraphics[width=1\linewidth]{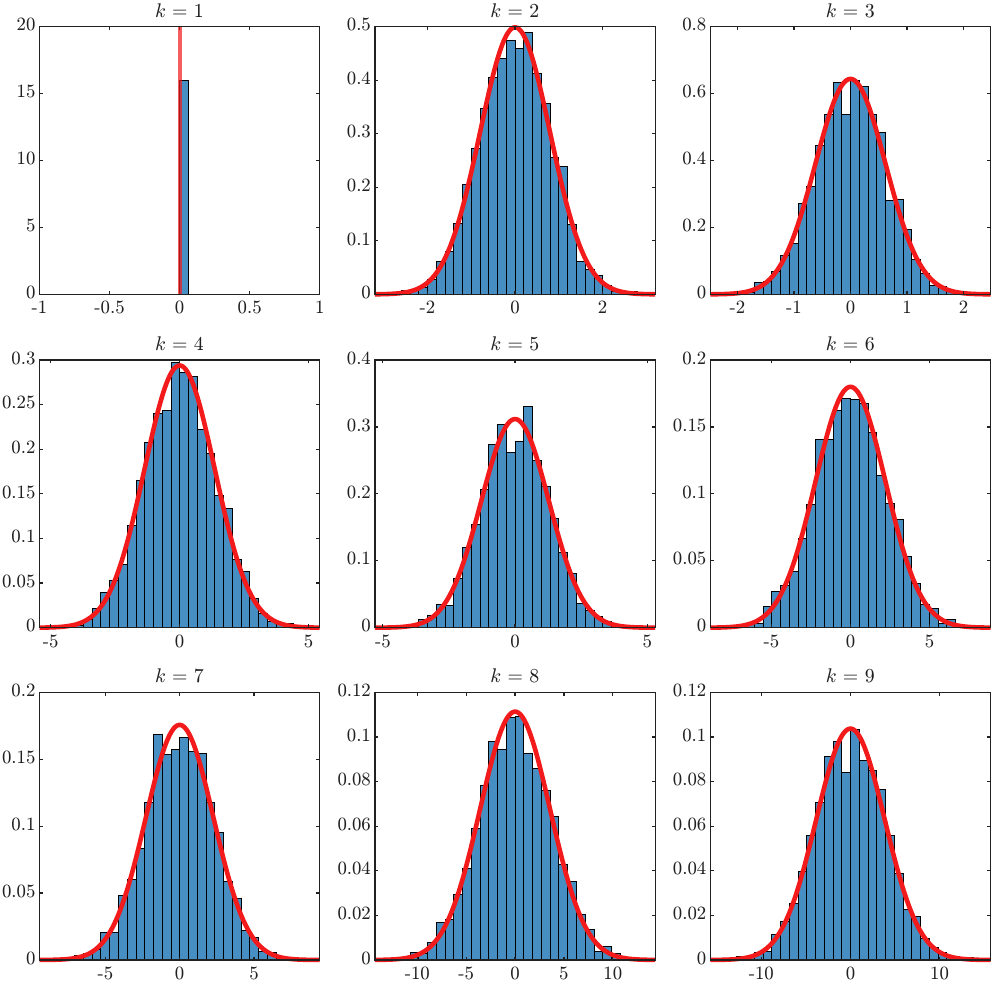}
    \caption{Histograms of $\operatorname{tr}((\mf{\Phi}_n-\mf{I}_p)^k)$ centered by sample means, overlaid with the centered Gaussian densities predicted by Theorem~\ref{thm:clt:Phi}, with $p=1000$, $n=500$, and over 2000 replications.}
    \label{Figure-CLT-Histograms-Phi}
\end{figure}

\begin{figure}[H]
    \centering
    \includegraphics[width=1\linewidth]{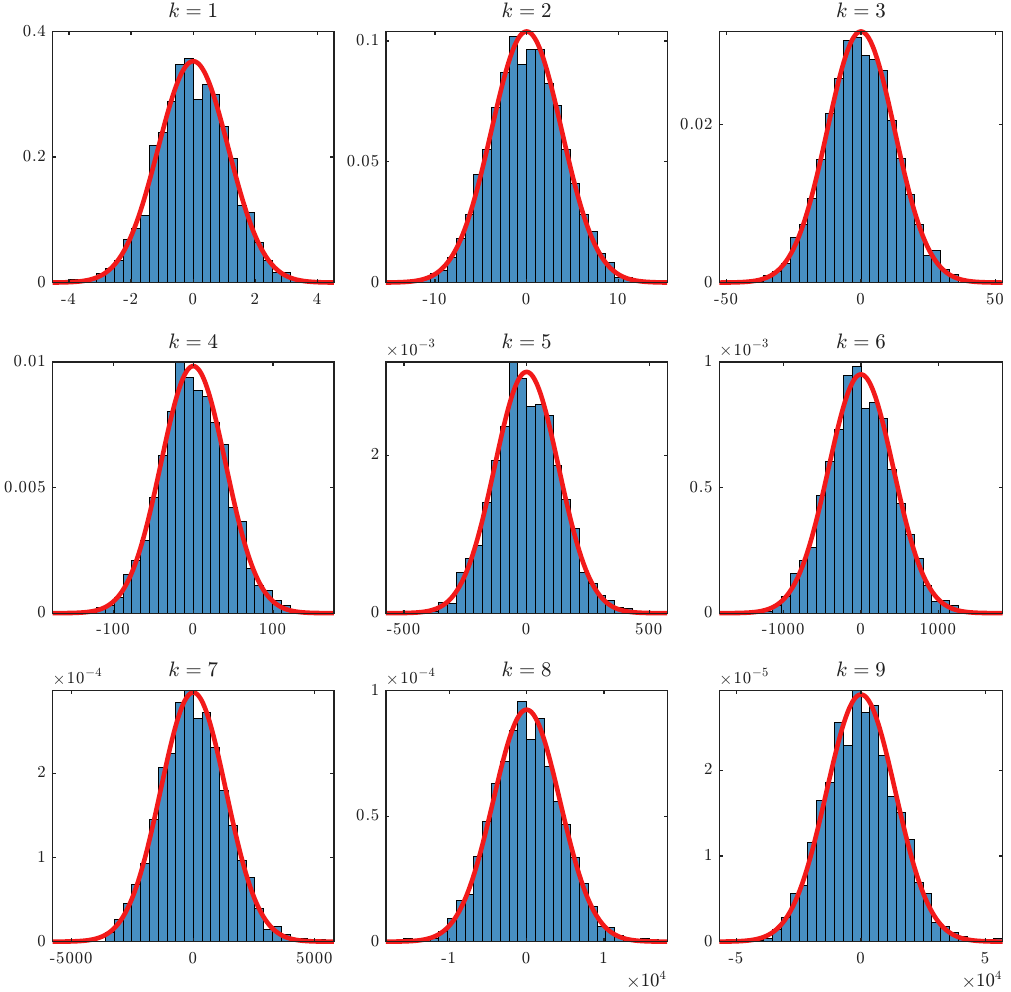}
    \caption{Histograms of $\operatorname{tr}(\mf{\Psi}_n^k)$ centered by sample means, overlaid with the centered Gaussian densities predicted by Theorem~\ref{thm:clt}, with $p=1000$, $n=500$, and over 2000 replications.}
    \label{Figure-CLT-Histograms-Variance-function}
\end{figure}

\section{Technical tools and proof strategy}\label{Section-technical-tools}

In the course of our proofs, we develop several key technical tools that are of independent interest beyond this work. We first highlight these tools in this section, and then provide an overview of the main proof strategy.

\subsection{Rankings and random permutations}\label{Section-rankings}
Chatterjee's rank correlation matrix $\mXi_n$ is independent of the data matrix $\fX$ once the rankings $R_1,\ldots,R_p$ are given. Hence, to study $\mXi_n$'s stochastic behavior under the null hypothesis $H_0$ of complete independence, it suffices to analyze the distribution of the rankings $R_j$'s, which is equivalent to studying the joint behavior of random permutations \citep{han2024introduction}. Let $\ml{S}_n$ be the group of permutations on $[n]$. We call $\sigma=\sigma^{(n)}$ a uniform random permutation if $\sigma$ has the uniform distribution on $\ml{S}_n$. Let $\sigma_1,\ldots,\sigma_p$ be $p$ independent uniform random permutations. Under Assumption \ref{assump:dgp}, it then follows immediately that
\[
\Big\{R_u : u\in[p]\Big\} \;\overset{d}{=}\; \Big\{\sigma_u : u\in[p]\Big\}.
\]
The following result, adapted from \citet[Proposition~2.1]{MR4185806}, will be the starting point for our analysis.

\begin{proposition}[\cite{MR4185806}]\label{prop:basic}
The following properties hold:
\begin{enumerate}[itemsep=-.5ex,label=(\roman*)]
\item for any $u\ne v \in [p]$, $\sigma_v \circ \sigma_u^{-1}$ is a uniform random permutation;
\item fixing an arbitrary $u \in [p]$, the collection 
\[
\{\sigma_v \circ \sigma_u^{-1} : v \in [p],\, v \neq u\}
\]
consists of mutually independent uniform random permutations;
\item fixing an arbitrary $v \in [p]$, the collection 
\[
\{\sigma_v \circ \sigma_u^{-1} : u \in [p],\, u \neq v\}
\]
consists of mutually independent uniform random permutations;
\item\label{prop:basic5} the family $\{\sigma_v \circ \sigma_u^{-1} : u\ne v \in [p]\}$ is \emph{not} mutually independent.
\end{enumerate}
\end{proposition}

Unfortunately, Proposition \ref{prop:basic} alone is insufficient to characterize the spectral behavior of some intricate rank correlation matrices like Chatterjee's. To address this, we need a more precise description of the joint dependence structure of
\[
\{\sigma_v \circ \sigma_u^{-1} : u\ne v \in [p]\}.
\]

\begin{remark}\label{remark:relative-rank}
Under Assumption \ref{assump:dgp}, the above collection is identically distributed to
\[
\{R_v \circ R_u^{-1} : u\ne v \in [p]\},
\]
commonly referred to in the literature as the \emph{relative ranks}, which are known to play a pivotal role in rank-based statistical analysis \citep{MR3737306}. In particular, define
\begin{align*}
f_{\rho}(\sigma) &= 1 - \frac{6}{n(n^2-1)}\sum_{k=1}^n \big(\sigma(k)-k\big)^2,\\
f_{\tau}(\sigma) &= \frac{2}{n(n-1)}\sum_{1\leq i<j\leq n} \operatorname{sgn}(i-j)\,\operatorname{sgn}\big(\sigma(i)-\sigma(j)\big),\\
f_{\xi}(\sigma) &= 1 - \frac{3}{n^2-1}\sum_{k=1}^{n-1} \big|\sigma(k+1)-\sigma(k)\big|.
\end{align*}
Then $f_\rho(R_v\circ R_u^{-1}), f_\tau(R_v\circ R_u^{-1}), f_\xi(R_v\circ R_u^{-1})$ correspond to Spearman, Kendall, and Chatterjee's rank correlations between $\fX_{\cdot u}$ and $\fX_{\cdot v}$, respectively.
\end{remark}

The following proposition gives a pivotal result that characterizes the joint dependence structure of $\{\sigma_v \circ \sigma_u^{-1} : u\ne v \in [p]\}$.

\begin{proposition}[Dependence structure of relative ranks]\label{lem:key} 
Let $\sigma_1,\cdots,\sigma_p$ be $p$ independent uniform random permutations on $[n]$, and consider $\Delta$ to be a directed multigraph with vertex set $V(\Delta)= [p]$, finite edge set $E(\Delta)=\{(u_k,v_k)\}_k$, and no self-loops. Then 
\[
\big\{\sigma_v \circ \sigma_u^{-1} : (u,v) \in E(\Delta)\big\}
\]
are mutually independent if and only if $\Delta$'s underlying undirected graph\footnote{See Section \ref{sec:graph-notation} ahead for a review of standard graph-theoretic terminology.} is a collection of non-overlapping trees. 
\end{proposition}

A direct consequence of Proposition \ref{lem:key} is the characterization of the joint dependence structure of rank correlations, including Spearman's rho, Kendall's tau, and Chatterjee's $\xi_n$, the last of which will be leveraged in the subsequent theoretical studies.

\begin{corollary}\label{Corollary:tree-independence-coef} 
Let $\cS_n$ denote the permutation group on $[n]$, and let $f:\cS_n\to\R$ be an arbitrary measurable function. If $\Delta$ is a tree, then the variables in
\[
\Big\{ f(\sigma_v\circ \sigma_u^{-1}) : (u,v)\in E(\Delta) \Big\}
\] 
are mutually independent. 
\end{corollary}

Corollary~\ref{Corollary:tree-independence-coef} is also important for interpreting the spectral result. Indeed, the three coefficients considered in Remark~\ref{remark:relative-rank} are all constructed from the same relative-rank variables and therefore share the tree-independence structure established in Proposition~\ref{lem:key}. Moreover, all three coefficients are asymptotically Gaussian under bivariate independence. Nevertheless, the corresponding high-dimensional matrices have fundamentally different LSDs: the Spearman and Kendall matrices converge to M-P laws, whereas the Chatterjee matrix exhibits a semicircle limit. The emergence of the semicircle law must therefore be attributed to the particular functional of the relative permutation underlying Chatterjee's statistic and the resulting graph-moment combinatorics.

A recent follow-up study by \citet{dong2026limiting} shows that the same semicircle phenomenon established in Theorem~\ref{thm:sc} also arises for other consistent rank correlation coefficients, including Hoeffding's $D$ \citep{MR0029139}, the Blum--Kiefer--Rosenblatt $R$ \citep{MR0125690}, and the Bergsma--Dassios--Yanagimoto $\tau^*$ \citep{MR3178526}. This finding supports the broader view that consistency against all forms of dependence, together with the associated rank-statistic structure, may underlie the emergence of semicircle limits.

\subsection{Preliminary results on Chatterjee's rank correlation}\label{Section-Chatterjee-correlation}

This section reviews and introduces some useful results on Chatterjee's rank correlation, while also clarifying some potential sources of confusion. Specifically, we begin by presenting several key results for the bivariate case under the assumption of independence.

\begin{proposition}\label{prop:crc}
Assume Assumption \ref{assump:dgp}. Then, for any $i \ne j \in [p]$, the following results are valid.
\begin{enumerate}[itemsep=-.5ex,label=(\roman*)]
\item \citet[Theorem 2.2]{chatterjee2020new}: 
\[
\sqrt{n}\,\Xi_{ij}^{(n)} \;\Rightarrow\; \mathcal{N}_1(0,2/5);
\]

\item \citet[Lemma 2]{zhang2022asymptotic}: 
\[
\Var\!\bigl(\sqrt{n}\,\Xi_{ij}^{(n)}\bigr) \;=\; \frac{n(n-2)(4n-7)}{10(n+1)(n-1)^2};
\]

\item \citet[Theorem 1]{zhang2022asymptotic}: 
\[
\bigl(\sqrt{n}\,\Xi_{ij}^{(n)}, \; \sqrt{n}\,\Xi_{ji}^{(n)}\bigr)^\top 
\;\Rightarrow\; \mathcal{N}_2\!\left(\mathbf{0}_2, \tfrac{2}{5}\mathbf{I}_2\right),
\]
where $\mathbf{0}_2$ denotes the two-dimensional zero vector;

\item \citet[Lemma 2.1]{xia2025consistent}: 
\begin{align*}
\E\!\left[\bigl\{\Xi_{ij}^{(n)}\bigr\}^2\right] 
&= \frac{(n-2)(4n-7)}{10(n-1)^2(n+1)}, \\[0.5ex]
\Var\!\left(\bigl\{\Xi_{ij}^{(n)}\bigr\}^2\right) 
&= \frac{224n^5 - 1792n^4 + 5051n^3 - 4969n^2 - 2458n + 18128}{700(n-1)^4(n+1)^3}.
\end{align*}
\end{enumerate}
\end{proposition}

Next, we characterize the joint behavior of multiple Chatterjee rank correlations under complete independence. In this regard, the following proposition extends \citet[Theorem~1]{zhang2022asymptotic} to the multivariate setting. Its proof builds on Chatterjee's method \citep{MR2435859,auddy2021exact,lin2022limit} and will be instrumental in deriving the closed-form expressions for the covariances in Theorem~\ref{thm:clt}.

\begin{proposition}\label{approx1}
Assume Assumption \ref{assump:dgp}.  Then, for any fixed $K\in\N$ and any distinct index pairs
\[
(i_1,j_1), \ldots, (i_K,j_K)\in[p]^2
\]
such that $i_k \ne j_k$ for all $k \in [K]$, we have
\[
\Big(\sqrt{n}\,\Xi_{i_1j_1}^{(n)}, \; \sqrt{n}\,\Xi_{i_2j_2}^{(n)}, \; \ldots, \; \sqrt{n}\,\Xi_{i_Kj_K}^{(n)}\Big)^\top 
\;\Rightarrow\; \mathcal{N}_K\!\left(\mathbf{0}_K, \tfrac{2}{5}\mathbf{I}_K\right).
\]
\end{proposition}

Lastly, we discuss the work of \citet{xia2025consistent}, which also investigates the problem of testing \eqref{eq:H0} using Chatterjee's rank correlation and LSSs. While their results are related and correspond to the special case $k=1$ in our Theorem~\ref{thm:clt} (and we have duly cited all that we used), their main result (Theorem~2.1) unfortunately suffers from serious technical flaws. Below, we point out two such issues and, in doing so, take the opportunity to clarify some common misconceptions about Chatterjee's rank correlation.

First, we emphasize that, in general, $\{\Xi_{ij}^{(n)}:\ 1\leq i< j\leq p\}$ are \emph{not} mutually independent; cf. Proposition~\ref{prop:basic}\ref{prop:basic5}. For instance, when $n=3$, we have
\bea\label{eq-counterExample-Independence-Xia-n=3-E-Prod}
\E\Big[\big((\Xi_{12}^{(3)})^2+(\Xi_{21}^{(3)})^2\big)\big((\Xi_{13}^{(3)})^2+(\Xi_{31}^{(3)})^2\big)\big((\Xi_{23}^{(3)})^2+(\Xi_{32}^{(3)})^2\big)\Big] &=& 5/16384,
\eea
while
\bea\label{eq-counterExample-Independence-Xia-n=3-Prod-E}
\E\big[(\Xi_{12}^{(3)})^2+(\Xi_{21}^{(3)})^2\big]\,
\E\big[(\Xi_{13}^{(3)})^2+(\Xi_{31}^{(3)})^2\big]\,
\E\big[(\Xi_{23}^{(3)})^2+(\Xi_{32}^{(3)})^2\big] &=& 1/4096.
\eea
The discrepancy between \eqref{eq-counterExample-Independence-Xia-n=3-E-Prod} and \eqref{eq-counterExample-Independence-Xia-n=3-Prod-E}\footnote{Equation~\eqref{eq-counterExample-Independence-Xia-n=3-E-Prod} was obtained via Matlab by enumerating all $6^3=216$ possible configurations of $(R_1,R_2,R_3)$, while \eqref{eq-counterExample-Independence-Xia-n=3-Prod-E} follows directly from \citet[Lemma~2.1]{xia2025consistent}.} shows that
\[
\Big\{(\Xi_{12}^{(3)})^2+(\Xi_{21}^{(3)})^2,\ 
(\Xi_{13}^{(3)})^2+(\Xi_{31}^{(3)})^2,\ 
(\Xi_{23}^{(3)})^2+(\Xi_{32}^{(3)})^2\Big\}
\]
are not mutually independent. Consequently, in the proof of asymptotic normality in \citet[Theorem~2.1]{xia2025consistent}, the assertion that
\[
\Big\{(\Xi_{ij}^{(n)})^2+(\Xi_{ji}^{(n)})^2:\ 1\leq i< j\leq p\Big\}
\]
are independent random variables is incorrect. Therefore, their use of the classical Lindeberg--Lévy central limit theorem is not justified. In fact, carefully addressing these dependencies is precisely the main technical challenge we resolve in this manuscript, via Proposition \ref{lem:key}.

Second, \citet[Theorem~2.1]{xia2025consistent} asserts that $\tr(\mPsi_n)$ is asymptotically normal even when $n$ is fixed, provided that $p\to\infty$. This, however, is incorrect. In detail, following the notation in \citet[Theorem~2.1]{xia2025consistent}, define
\[
J_\xi=\sigma_{np}^{-1}\sum_{1\leq k<\ell\leq p}\varphi_{k\ell}^{(n)},\quad 
\varphi_{k\ell}^{(n)}=(\Xi_{k\ell}^{(n)})^2+(\Xi_{\ell k}^{(n)})^2-2\E\big[(\Xi_{12}^{(n)})^2\big],
\]
\[
\sigma_{np}^2=p(p-1)\Big(\var\big[(\Xi_{12}^{(n)})^2\big]+\operatorname{cov}\big[(\Xi_{12}^{(n)})^2,(\Xi_{21}^{(n)})^2\big]\Big),
\]
where $\E[(\Xi_{12}^{(n)})^2]$, $\var[(\Xi_{12}^{(n)})^2]$, and $\operatorname{cov}[(\Xi_{12}^{(n)})^2,(\Xi_{21}^{(n)})^2]$ are explicit functions of $n$ given in \citet[Lemmas~2.1--2.2]{xia2025consistent}.  
\citet[Theorem~2.1]{xia2025consistent} claimed that
$
J_\xi \Rightarrow \mathcal{N}(0,1)
$
for both fixed and diverging $n$, as long as $p\to\infty$.  

However, consider the case where $n$ is fixed. As $p\to\infty$, we have
\begin{align*}
\E[J_\xi^3]
&=\sigma_{np}^{-3}\E\Bigg[\Bigg(\sum_{1\leq k<\ell\leq p}\varphi_{k\ell}^{(n)}\Bigg)^3\Bigg]\\
&=\sigma_{np}^{-3}\left(\frac{p(p-1)}{2}\E\big[(\varphi_{12}^{(n)})^3\big]+\frac{p(p-1)(p-2)}{2}\E\big[\varphi_{12}^{(n)}\varphi_{23}^{(n)}\varphi_{13}^{(n)}\big]\right)\\
&\longrightarrow \frac{1}{2}\cdot \frac{\E\big[\varphi_{12}^{(n)}\varphi_{23}^{(n)}\varphi_{13}^{(n)}\big]}{\var\big[(\Xi_{12}^{(n)})^2\big]+\operatorname{cov}\big[(\Xi_{12}^{(n)})^2,(\Xi_{21}^{(n)})^2\big]}.
\end{align*}
By \eqref{eq-counterExample-Independence-Xia-n=3-E-Prod} and \citet[Lemmas~2.1 and 2.2]{xia2025consistent}, we obtain
\[
\text{when }n=3,\quad \lim_{p\to\infty}\E[J_\xi^3]=\frac{5}{2752}\neq 0.
\]
Therefore, when $n=3$, the convergence $J_\xi \Rightarrow \mathcal{N}(0,1)$ (as $p\to \infty$) does not hold. This contradicts the claim made in \citet[Theorem~2.1]{xia2025consistent}.

\subsection{Proof strategy}
Our proof is based on the moment method. For the limiting spectral distribution (LSD) analysis (Theorems \ref{thm:sc} and \ref{thm:mp}), recall that in the classical argument for Wigner or Wishart matrices, the leading contributions arise from graphs that are trees with every edge appearing exactly twice. In our setting, when the underlying graph is a tree, we show that the corresponding matrix elements are mutually independent (Corollary~\ref{Corollary:tree-independence-coef} combined with Remark \ref{remark:relative-rank}), which yields the same classical contribution up to a scaling factor reflecting the limiting entry-wise variance of $2/5$; cf. Proposition \ref{prop:crc} above. 

The remaining task is to prove that contributions from non-tree graphs are negligible. Our approach proceeds in several steps. First, we decompose the expectation associated with a graph into a product of expectations over its vertices (Proposition~\ref{Graphical-Representation}). Next, we establish local estimates at each vertex (Proposition~\ref{Local-Estimate-Partition}), which are then aggregated into a global bound for the entire graph. To sharpen these bounds, we use an additional argument based on the idea that graphs with more single edges should be penalized and therefore contribute less (Proposition~\ref{Graphical-Representation-Reduced}). This refined estimate suffices to rule out all non-tree contributions (Proposition~\ref{Graph-Estimate-EDelta0}), leading to the semicircle and M-P laws.

For the LSS analysis (Theorem \ref{thm:clt}), the leading contribution to the covariance function arises from two types of graphs:  
\begin{enumerate}[itemsep=-.5ex,label=(\roman*)]
\item graphs containing exactly one cycle (referred to as \emph{1-cycle graphs}) with all edges doubled;
\item trees in which all edges are doubled except for a single quadruple edge.  
\end{enumerate}

The enumeration of the 1-cycle contribution is based on the following idea: to construct a 1-cycle graph, we first fix the cycle length and then attach trees along the cycle. As in the LSD case, the same technique—combined with a joint Gaussian approximation of the entries (Proposition~\ref{approx1})—is used to rule out all other graph contributions. The proof of Gaussianity itself relies on Wick's formula \citep{Wick-50}, which is inspired by the proof of \citep[Theorem 2.2]{Camille2014CentralLimitTheorems}.

\subsection{Eliminating non-tree contribution: an illustrative example}\label{section-illustrative-example}
We illustrate the proof with an example $E_4=\E[\Xi_{12}\Xi_{32}\Xi_{34}\Xi_{14}]$. We show that $E_4=O(n^{-4})$. The same strategy is applicable to general cases. Write 
\[
u_n=\frac{n^2-1}{3}, ~~a(i,j)=\frac{n+1}{3}-|i-j|,~~ \text{and }~~R_jR_{i}^{-1}=R_j\circ R_{i}^{-1} 
\]
for short. Recall that $R_1,R_2,R_3,R_4$ are independent uniform random permutations. By definition,
\bea\nonumber
E_4=u_n^{-4}\sum_{k_1,k_2,k_3,k_4=1}^{n-1}\E\bigg[&a(R_2R_1^{-1}(k_1+1), R_2R_1^{-1}(k_1)) a(R_2R_3^{-1}(k_2+1), R_2R_3^{-1}(k_2))\\ \cdot & a(R_4R_3^{-1}(k_3+1), R_4R_3^{-1}(k_3)) a(R_4R_1^{-1}(k_4+1), R_4R_1^{-1}(k_4))\bigg].
\eea
The first step is to take the expectation conditional on 
$$\l\{R_1^{-1}(k_1+1),R_1^{-1}(k_1),R_3^{-1}(k_2+1),R_3^{-1}(k_2),R_3^{-1}(k_3+1),R_3^{-1}(k_3),R_1^{-1}(k_4+1),R_1^{-1}(k_4)\r\}.$$ That is,
{\small
\bea\label{Eq-E4-Graphical-Representation}
E_4=&u_n^{-4}\sum_{\substack{k_1,k_2\\k_3,k_4}}\sum_{\substack{a_1,a_2,b_1,b_2\\c_1,c_2,d_1,d_2}}\E\Bigg[\substack{a(R_2R_1^{-1}(k_1+1), R_2R_1^{-1}(k_1))\\ \cdot a(R_2R_3^{-1}(k_2+1), R_2R_3^{-1}(k_2))\\ \cdot a(R_4R_3^{-1}(k_3+1), R_4R_3^{-1}(k_3))\\ \cdot a(R_4R_1^{-1}(k_4+1), R_4R_1^{-1}(k_4))}\Bigg|\substack{R_1^{-1}(k_1)=a_1,R_1^{-1}(k_1+1)=a_2,\\R_3^{-1}(k_2)=b_1,R_3^{-1}(k_2+1)=b_2,\\R_3^{-1}(k_3)=c_1,R_3^{-1}(k_3+1)=c_2,\\R_1^{-1}(k_4)=d_1,R_1^{-1}(k_4+1)=d_2} \Bigg]\P\Bigg(\substack{R_1^{-1}(k_1)=a_1,R_1^{-1}(k_1+1)=a_2,\\R_3^{-1}(k_2)=b_1,R_3^{-1}(k_2+1)=b_2,\\R_3^{-1}(k_3)=c_1,R_3^{-1}(k_3+1)=c_2,\\R_1^{-1}(k_4)=d_1,R_1^{-1}(k_4+1)=d_2} \Bigg)\\
=&u_n^{-4}\sum_{\substack{k_1,k_2\\k_3,k_4}}\sum_{\substack{a_1,a_2,b_1,b_2\\c_1,c_2,d_1,d_2}}\E\bigg[\substack{a(R_2(a_2), R_2(a_1))\\ \cdot a(R_2(b_2), R_2(b_1))}\bigg]\E\bigg[\substack{a(R_4(c_2), R_4(c_1))\\ \cdot a(R_4(d_2), R_4(d_1))}\bigg]\P\bigg(\substack{R_1^{-1}(k_1)=a_1,R_1^{-1}(k_1+1)=a_2,\\R_1^{-1}(k_4)=d_1,R_1^{-1}(k_4+1)=d_2}\bigg)\P\bigg(\substack{R_3^{-1}(k_2)=b_1,R_3^{-1}(k_2+1)=b_2,\\R_3^{-1}(k_3)=c_1,R_3^{-1}(k_3+1)=c_2}\bigg)\\
=&u_n^{-4}\sum_{\substack{a_1,a_2,b_1,b_2\\c_1,c_2,d_1,d_2}}\E\bigg[\substack{a(R_2(a_2), R_2(a_1))\\ \cdot a(R_2(b_2), R_2(b_1))}\bigg]\E\bigg[\substack{a(R_4(c_2), R_4(c_1))\\ \cdot a(R_4(d_2), R_4(d_1))}\bigg]\P\bigg(\substack{R_1(a_1)+1=R_1(a_2),\\R_1(d_1)+1=R_1(d_2)}\bigg)\P\bigg(\substack{R_3(b_1)+1=R_3(b_2),\\R_3(c_1)+1=R_3(c_2)}\bigg).
\eea
}
Here in the second equality, we use the independence between $\{R_2,R_4\}$ and $\{R_1,R_3\}$ to remove the conditioning in expectation, and then use the independence between $R_2$ and $R_4$ and between $R_1$ and $R_3$ to factorize the expectation or probability into product of two terms, respectively.

The second step is to further simplify $E_4$. The key observation is that, the term
$$
\P\bigg(\substack{R_1(a_1)+1=R_1(a_2),\\R_1(d_1)+1=R_1(d_2)}\bigg),
$$ denoted as $\P_\sigma(A,D)$,
only depends on how $A:=(a_1\rw a_2)$ intersects with $D:=(d_1\rw d_2)$. Denoting  $\E_a[A,D]:=\E\l[a(\sigma(a_1),\sigma(a_2))a(\sigma(d_1),\sigma(d_2))\r]$, we then have
\bea\nonumber
E_4&=u_n^{-4}\sum_{A,B,C,D}\E_a[A,B]\E_a[C,D]\P_\sigma(A,D)\P_\sigma(B,C)\\
&=u_n^{-4}\sum_{B,C,D}\E_a[C,D]\P_\sigma(B,C)\bigg(\sum_{A:A\cap D\ne\emptyset}+\sum_{A:A\cap D=\emptyset}\bigg)\E_a[A,B]\P_\sigma(A,D),
\eea
with $B:=(b_1 \rw b_2)$, $C:=(c_1 \rw c_2)$, and the corresponding $\P_\sigma$ and $\E_a$ defined similarly to those concerning $A,D$.

To proceed, we use arrows to represent all possible cases of $A\cap D$ in Table \ref{ProofSketch_Arrows}. For instance, when $a_1,a_2,d_1,d_2\text{ are distinct}$, $\P_\sigma(A,D)$ is denoted as $\P_\sigma(\dotheadrightarrow\ \dotheadrightarrow)$, with the value equal to $\P(\sigma(1)+1=\sigma(2),\sigma(3)+1=\sigma(4))$, where $\sigma$ is a uniform random permutation on $[n]$. 

\begin{table}[H]
\centering
\renewcommand{\arraystretch}{1.0} 
\caption{Arrow notations}
\begingroup
\small
\setlength{\tabcolsep}{3pt}
\renewcommand{\arraystretch}{1.05}
\begin{tabular}{@{}>{\centering\arraybackslash}m{0.24\textwidth}
                >{\centering\arraybackslash}m{0.22\textwidth}
                >{\centering\arraybackslash}m{0.43\textwidth}@{}}
\toprule
$(A,D)$ & Arrow notations & $\P_\sigma(A,D)$ \\
\midrule
$\substack{A\cap D=\emptyset\\(a_1,a_2,d_1,d_2\text{ are distinct})}$ & $\dotheadrightarrow\ \dotheadrightarrow$ & $\P(\sigma(1)+1=\sigma(2),\sigma(3)+1=\sigma(4))$\\
\midrule
$\substack{ a_2=d_1,\ a_1\ne d_2\\ \text{or}\ d_2=a_1,\ d_1\ne a_2}$ & $\dotheadrightrightarrow$ & $\P(\sigma(1)+1=\sigma(2),\sigma(2)+1=\sigma(3))$\\
\midrule
$a_1=d_1,\ a_2=d_2$ & $\dotheadDoubleRightarrow$ & $\P(\sigma(1)+1=\sigma(2))$\\
\midrule
$a_1=d_1,\ a_2\ne d_2$ & $\dotheadleftrightarrow$ & 0\\
\midrule
$a_2=d_2,\ a_1\ne d_1$ & $\dotheadrightleftarrow$ & 0\\
\midrule
$a_1=d_2,\ a_2=d_1$ & $\dotheadSingleDoublearrow$ & 0\\
\bottomrule
\end{tabular}
\label{ProofSketch_Arrows}
\endgroup
\end{table}

From $\sum_{1\leq i\ne j\leq n}a(i,j)=0$, we have
\bea\nonumber
\sum_{A:A\cap D=\emptyset}\E_a[A,B]\P_\sigma(A,D)&=\P_\sigma(\dotheadrightarrow\ \dotheadrightarrow)\sum_{A:A\cap D=\emptyset}\E_a[A,B]\\
&=\P_\sigma(\dotheadrightarrow\ \dotheadrightarrow)\E\bigg[a(\sigma(b_1),\sigma(b_2))\sum_{A:A\cap D=\emptyset}a(\sigma(a_1),\sigma(a_2))\bigg]\\
&=\P_\sigma(\dotheadrightarrow\ \dotheadrightarrow)\E\bigg[a(\sigma(b_1),\sigma(b_2))\bigg(-\sum_{A:A\cap D\ne\emptyset}a(\sigma(a_1),\sigma(a_2))\bigg)\bigg]\\
&=-\P_\sigma(\dotheadrightarrow\ \dotheadrightarrow)\sum_{A:A\cap D\ne\emptyset}\E_a[A,B],
\eea
so that
\bea\nonumber
E_4&=u_n^{-4}\sum_{\substack{A,B,C,D\\A\cap D\ne \emptyset}}\E_a[A,B]\E_a[C,D]\bigg(\P_\sigma(A,D)-\P_\sigma(\dotheadrightarrow\ \dotheadrightarrow)\bigg)\P_\sigma(B,C).
\eea
Proceeding similarly by separating the sum on $B$, we have
\bea\label{Eq-E4-Reduce-Single-Edges}
E_4&=u_n^{-4}\sum_{\substack{A,B,C,D\\A\cap D\ne \emptyset,\ B\cap C\ne \emptyset}}\E_a[A,B]\E_a[C,D]\bigg(\P_\sigma(A,D)-\P_\sigma(\dotheadrightarrow\ \dotheadrightarrow)\bigg)\bigg(\P_\sigma(B,C)-\P_\sigma(\dotheadrightarrow\ \dotheadrightarrow)\bigg).
\eea
Direct calculation gives $$\P_\sigma(\dotheadrightrightarrow)=\P_\sigma(\dotheadrightarrow\ \dotheadrightarrow)=\frac{1}{n(n-1)}.$$ Thus, we can write
$$
E_{4}=E_{41}+E_{42}+E_{43}+E_{44},
$$
where
\bea\nonumber
E_{41}&=u_n^{-4}\sum_{\substack{(A,D)=\dotheadDoubleRightarrow\\(B,C)=\dotheadDoubleRightarrow}}\E_a[A,B]\E_a[C,D]\bigg(\P_\sigma(A,D)-\P_\sigma(\dotheadrightarrow\ \dotheadrightarrow)\bigg)\bigg(\P_\sigma(B,C)-\P_\sigma(\dotheadrightarrow\ \dotheadrightarrow)\bigg),\\
E_{42}&=u_n^{-4}\sum_{\substack{(A,D)=\dotheadDoubleRightarrow\\(B,C)\in\{\dotheadleftrightarrow,\dotheadrightleftarrow,\dotheadSingleDoublearrow\}}}\E_a[A,B]\E_a[C,D]\bigg(\P_\sigma(A,D)-\P_\sigma(\dotheadrightarrow\ \dotheadrightarrow)\bigg)\bigg(\P_\sigma(B,C)-\P_\sigma(\dotheadrightarrow\ \dotheadrightarrow)\bigg),\\
E_{43}&=u_n^{-4}\sum_{\substack{(A,D)\in\{\dotheadleftrightarrow,\dotheadrightleftarrow,\dotheadSingleDoublearrow\}\\(B,C)=\dotheadDoubleRightarrow}}\E_a[A,B]\E_a[C,D]\bigg(\P_\sigma(A,D)-\P_\sigma(\dotheadrightarrow\ \dotheadrightarrow)\bigg)\bigg(\P_\sigma(B,C)-\P_\sigma(\dotheadrightarrow\ \dotheadrightarrow)\bigg),\\
E_{44}&=u_n^{-4}\sum_{\substack{(A,D)\in\{\dotheadleftrightarrow,\dotheadrightleftarrow,\dotheadSingleDoublearrow\}\\(B,C)\in\{\dotheadleftrightarrow,\dotheadrightleftarrow,\dotheadSingleDoublearrow\}}}\E_a[A,B]\E_a[C,D]\bigg(\P_\sigma(A,D)-\P_\sigma(\dotheadrightarrow\ \dotheadrightarrow)\bigg)\bigg(\P_\sigma(B,C)-\P_\sigma(\dotheadrightarrow\ \dotheadrightarrow)\bigg).
\eea
Carefully counting all possibilities gives
\bea\label{Eq-E4-Asymptotic-Order-Estimate}
E_{41}=\frac{1}{100n^3}+O(n^{-4}),\quad E_{42}=E_{43}=-\frac{1}{100n^3}+O(n^{-4}),\quad E_{44}=\frac{1}{100n^3}+O(n^{-4}),
\eea
which yields that $E_4=O(n^{-4})$.

The process of such calculation is generalized in Section \ref{Section-combinatornics} ahead as follows: 
\begin{enumerate}[itemsep=-.5ex,label=(\roman*)]
\item Proposition \ref{Graphical-Representation} generalizes the calculation of \eqref{Eq-E4-Graphical-Representation}; 
\item Proposition \ref{Graphical-Representation-Reduced} deals with the single edges as in \eqref{Eq-E4-Reduce-Single-Edges}; 
\item Propositions \ref{Local-Estimate-Partition} and \ref{Graph-Estimate-EDelta0} provide estimates analogous to \eqref{Eq-E4-Asymptotic-Order-Estimate}.
\end{enumerate}

\section{Applications to testing complete independence}\label{sec:han-app}
In this section, we apply Theorem~\ref{thm:clt} to the problem of testing complete independence in \eqref{eq:H0}. Throughout this section, let $\Phi$ denote the distribution function of the standard Gaussian distribution, and write 
\[
z_u=\Phi^{-1}(1-u)~~ \text{ for }~~ u\in(0,1). 
\]
Our basic test statistic is
\[
Q_{\xi,2}
=\frac{\operatorname{tr}(\mf{\Psi}_n)-\mu_{\xi,n,p}}{\sqrt{8p^2/(25n^2)}},
\quad~\text{with }~~
\mu_{\xi,n,p}
=\frac{p(p-1)(n-2)(4n-7)}{10(n-1)^2(n+1)}.
\]
Here the centering $\mu_{\xi,n,p}$ is the exact null mean of $\operatorname{tr}(\mf{\Psi}_n)$, and the denominator is obtained from the limiting variance $\operatorname{Var}(G_1)=8\gamma^2/25$ after replacing $\gamma$ by $p/n$. The level-$\alpha$ test based on $Q_{\xi,2}$ rejects when $Q_{\xi,2}>z_{\alpha}$.

In practice, it has been noticed that using the statistic $Q_{\xi,2}$ alone for testing  \eqref{eq:H0} may lack power in regular models \citep{shi2020power}. Thus, in order to combine ``the best of both worlds'' \cite[Section 4]{bickel2022measures}, in this section we also introduce a Bonferroni combination of $Q_{\xi,2}$ with a statistic of \citet{leung2018testing}. More specifically, for coordinates $i$ and $j$, set
\[
h_{ab}^{(ij)}=\operatorname{sign}\{(X_{ai}-X_{bi})(X_{aj}-X_{bj})\},\quad a<b,
\]
with $\operatorname{sign}(\cdot)$ denoting the sign function, and set
\[
W_{\tau,ij}
=\frac{1}{\binom{n}{2}\binom{n-2}{2}}
\sum_{a<b}
\sum_{\substack{c<d:\ \{c,d\}\cap\{a,b\}=\emptyset}}
h_{ab}^{(ij)}h_{cd}^{(ij)}.
\]
The Kendall-rank sum statistic is
\[
T_\tau=\frac{9n}{4p}\sum_{i<j}W_{\tau,ij}.
\]
Let
\[
p_\xi=1-\Phi(Q_{\xi,2}),\quad
p_\tau=1-\Phi(T_\tau),\quad
p_{\xi,\tau}^{\min}=\min(p_\xi,p_\tau).
\]
The Bonferroni-combined test, denoted by $B_{\xi,\tau}$, rejects the null hypothesis when $p_{\xi,\tau}^{\min}\leq \alpha/2$. This combination is intended to retain the sensitivity of $Q_{\xi,2}$ to nonlinear and nonmonotone dependence while inheriting the efficiency of $T_\tau$ in regular models.

\subsection{Size control of $Q_{\xi,2}$ and $B_{\xi,\tau}$}

Let $\cD_p^{\rm{ind}}$ be the class of all distributions on $\bR^p$ whose coordinates are mutually independent and have continuous marginal distribution functions. For $D\in\cD_p^{\rm{ind}}$, let $\mX_1,\ldots,\mX_n$ be independent copies of $\mX\sim D$. Denote by $\phi_\alpha(Q_{\xi,2})$ and $\phi_\alpha(B_{\xi,\tau})$ the rejection indicators of the level-$\alpha$ tests based on $Q_{\xi,2}$ and $B_{\xi,\tau}$, respectively:
$$
\phi_\alpha(Q_{\xi,2})=\ind\{Q_{\xi,2}>z_\alpha\},
$$
\begin{align}\label{eq:han-bonferoni}
\phi_\alpha(B_{\xi,\tau})=\ind\{Q_{\xi,2}>z_{\alpha/2}\ \text{or}\ T_{\tau}>z_{\alpha/2}\}.
\end{align}

The following theorem gives the asymptotically uniform size control of the level-$\alpha$ tests based on both $Q_{\xi,2}$ and $B_{\xi,\tau}$; it is analogous to \citet[Proposition 5.3]{shi2020rate}.

\begin{proposition}\label{prop:han-size}
    Assume that $p\rw\infty$, $n\rw\infty$, $p/n\rw \gamma\in (0,+\infty)$. Then
    $$
    \limsup_{n\rw\infty} \sup_{D\in \cD_p^{\rm{ind}}} \bE_D[\phi_\alpha(Q_{\xi,2})]\leq \alpha,\qquad
    \limsup_{n\rw\infty} \sup_{D\in \cD_p^{\rm{ind}}} \bE_D[\phi_\alpha(B_{\xi,\tau})]\leq \alpha.
    $$
\end{proposition}

\begin{remark}
Although the Bonferroni rule in \eqref{eq:han-bonferoni} is conservative, if $Q_{\xi,2}$ and $T_\tau$ are asymptotically independent under $H_0$, then its limiting rejection probability is $\alpha-\alpha^2/4$, differing from the nominal level by only $\alpha^2/4$. In fact, in the bivariate null setting, \citet{Zhang-relations} established the asymptotic independence of Chatterjee's $\xi$ and Spearman's $\rho$. We conjecture that an analogous high-dimensional phenomenon holds for $Q_{\xi,2}$ and $T_\tau$. Although a rigorous proof of this conjecture is currently beyond reach, the simulation results reported in Table~\ref{tab:empirical_size} are consistent with it. Specifically, in both Models~(a) and~(b), the empirical size of $B_{\xi,\tau}$ is very close to the nominal level of $0.05$.
\end{remark}

\subsection{Local power analysis of $Q_{\xi,2}$ and $B_{\xi,\tau}$}\label{sec:han-lpa}

We first characterize conditions under which the test based on $Q_{\xi,2}$ has nontrivial asymptotic power. Notably, this result does not require the high-dimensional regime in which $p\to\infty$ and $p/n\to\gamma\in(0,+\infty)$.

More precisely, let $\Dc_p$ denote the class of all distributions on $\R^p$ with continuous marginal distribution functions. For each $D\in\Dc_p$, let $\mX_1,\ldots,\mX_n$ be independent copies of $\mX\sim D$, and define
\[
\overline{\mXi}_n:=\bE_D[\mXi_n],
\]
which measures the dependence signal under $D$ through Chatterjee's rank correlation matrix. For any $t>0$, write
\[
\Dc_p\bigl(\|\overline{\mXi}_n-\mf{I}_p\|_{\mathrm F}\geq t\bigr)
\]
for the subclass of distributions in $\Dc_p$ satisfying the displayed Frobenius-norm condition.

\begin{theorem}[Power of $Q_{\xi,2}$ under a centered Frobenius signal]
\label{thm:power-Qxi2}
Assume that $n\to\infty$ and $p\geq 2$. Then, for any $0<\alpha<\beta<1$, there exists a constant $C=C(\alpha,\beta)>0$ such that
\[
\liminf_{n\to\infty}\ 
\inf_{D\in\Dc_p\bigl(\|\overline{\mXi}_n-\mf{I}_p\|_{\mathrm F}
\geq Cpn^{-1/2}\bigr)}
\bE_D\bigl\{\phi_\alpha(Q_{\xi,2})\bigr\}
>\beta.
\]
\end{theorem}

Notably, the signal strength in Theorem~\ref{thm:power-Qxi2} is expressed in terms of the population dependence measured by Chatterjee's rank correlation. In applications, it is also natural to characterize the signal strength in terms of Pearson correlation. To this end, let
\[
\Nc_p^{\mathrm{equi}}
=
\left\{
\Nc_p(\bm{0},\fR):
\fR=(1-\rho)\mf{I}_p+\rho\mf{1}_p\mf{1}_p^\top
\right\}
\]
denote the class of centered equicorrelated Gaussian distributions with population covariance matrix $\fR$. Similarly, for any $t>0$, write
\[
\Nc_p^{\mathrm{equi}}
\bigl(\|\fR-\mf{I}_p\|_{\mathrm F}\geq t\bigr)
\]
for the subclass of $\Nc_p^{\mathrm{equi}}$ satisfying the displayed Frobenius-norm condition.

The following corollary specializes Theorem~\ref{thm:power-Qxi2} to the distribution class $\Nc_p^{\mathrm{equi}}$.

\begin{corollary}[Equicorrelated Gaussian alternatives]
\label{cor:equi-gaussian-power-Qxi2}
Assume that $n\to\infty$ and $p\geq 2$. Then, for every $0<\alpha<\beta<1$, there exists a constant $C=C(\alpha,\beta)>0$ such that
\[
\liminf_{n\to\infty}
\inf_{D\in\Nc_p^{\mathrm{equi}}
\bigl(\|\fR-\mf{I}_p\|_{\mathrm F}\geq Cpn^{-1/4}\bigr)}
\bE_D\bigl\{\phi_\alpha(Q_{\xi,2})\bigr\}
>\beta.
\]
\end{corollary}

When $p=2$, the test based on $Q_{\xi,2}$ has nontrivial power at the signal scale $n^{-1/4}$, matching the detection boundary identified by \citet{auddy2021exact}. This relatively slow local rate is also consistent with the findings of \citet{shi2020power}, who showed that Chatterjee's rank correlation is rate-suboptimal, relative to Hoeffding's $D$, the Blum--Kiefer--Rosenblatt $R$, and the Bergsma--Dassios--Yanagimoto $\tau^*$, against certain bivariate local alternatives. For large $p$, our simulations likewise indicate that $Q_{\xi,2}$ alone may have limited power against local Gaussian alternatives.

The Bonferroni statistic $B_{\xi,\tau}$, by contrast, is designed to address this limitation by combining the complementary strengths of $Q_{\xi,2}$ and the Kendall-rank statistic $T_\tau$. In particular, \citet[Theorem~5.2]{leung2018testing} established the detection boundary of the test based on $T_\tau$ under the regime
\[
p\to\infty,\qquad n\to\infty,\qquad \frac{p}{n}\to\gamma\in(0,+\infty).
\]
The following theorem establishes an analogous result for $B_{\xi,\tau}$. Together with Proposition~\ref{prop:han-size} and \citet[Theorem~5.3]{leung2018testing}, it implies the rate optimality of the Bonferroni test based on $B_{\xi,\tau}$.

\begin{theorem}[Power of the Bonferroni test under equicorrelated Gaussian alternatives]
\label{thm-optimal-Bonferroni}
Assume that
\[
p\to\infty,\qquad n\to\infty,\qquad \frac{p}{n}\to\gamma\in(0,+\infty).
\]
Then, for every $0<\alpha<\beta<1$, there exists a constant
$C=C(\alpha,\beta,\gamma)>0$ such that
\[
\liminf_{n\to\infty}
\inf_{D\in\Nc_p^{\mathrm{equi}}
\bigl(\|\fR-\mf{I}_p\|_{\mathrm F}\geq C\bigr)}
\bE_D\bigl\{\phi_\alpha(B_{\xi,\tau})\bigr\}
>\beta.
\]
\end{theorem}

\subsection{Numerical simulation}

We next investigate the finite-sample performance of $Q_{\xi,2}$ and $B_{\xi,\tau}$. For comparison, we also include the existing tests proposed by \citet{Schott-05}, \citet{MR4185806}, \citet{MR3798874}, and \citet{leung2018testing}. The three tests of \citet{MR4185806} are calibrated using their limiting Gumbel null distributions, whereas all remaining procedures are calibrated by their asymptotic normal null distributions and implemented as right-tailed tests at significance level $\alpha=0.05$. The competing test statistics are defined as follows.

\begin{enumerate}
    \item \citet{Schott-05}: With $r_{ij}$ denoting the Pearson sample correlation between the $i$th and $j$th coordinates,
    \[
    Q_{r,2}=\sqrt{\frac{n^2(n+2)}{p(p-1)(n-1)}}\left(\sum_{i<j}r_{ij}^2-\frac{p(p-1)}{2(n-1)}\right).
    \]

    \item \citet{MR4185806}: Let $\widehat D_{ij}$, $\widehat R_{ij}$, and $\widehat\tau^*_{ij}$ denote Hoeffding's $D$, Blum--Kiefer--Rosenblatt's $R$, and Bergsma--Dassios--Yanagimoto's $\tau^*$, respectively, computed from the $i$th and $j$th coordinates as in \citet{MR4185806}. The three maximum-type statistics are
    \[
    \begin{aligned}
    M_D&=\frac{\pi^4(n-1)}{30}\max_{i<j}\widehat D_{ij}-4\log p+\log\log p+\frac{\pi^4}{36},\\
    M_R&=\frac{\pi^4(n-1)}{90}\max_{i<j}\widehat R_{ij}-4\log p+\log\log p+\frac{\pi^4}{36},\\
    M_{\tau^*}&=\frac{\pi^4(n-1)}{54}\max_{i<j}\widehat\tau^*_{ij}-4\log p+\log\log p+\frac{\pi^4}{36}.
    \end{aligned}
    \]
    These tests reject $H_0$ when the
    corresponding statistic exceeds
    \[
    q_D(\alpha)
    =\log\!\left(\frac{\kappa_D^2}{8\pi}\right)
    -2\log\log\!\left(\frac{1}{1-\alpha}\right),
    \quad
    \kappa_D=\left\{
    2\prod_{\ell=2}^{\infty}
    \frac{\pi/\ell}{\sin(\pi/\ell)}
    \right\}^{1/2}\approx 2.467.
    \]

    \item \citet{MR3798874}: For each coordinate $k$, let $\fA^{(k)}$ be the $n\times n$ matrix with $A^{(k)}_{aa}=0$ and, for $a\ne b$,
    \[
    \begin{aligned}
    A^{(k)}_{ab}
    =&\, |X_{ak}-X_{bk}|
    -\frac{1}{n-2}\sum_{\ell=1}^n |X_{ak}-X_{\ell k}|
    -\frac{1}{n-2}\sum_{\ell=1}^n |X_{\ell k}-X_{bk}|\\
    &+\frac{1}{(n-1)(n-2)}\sum_{\ell,m=1}^n |X_{\ell k}-X_{mk}|.
    \end{aligned}
    \]
    Define the unbiased squared sample distance covariance and distance variance by
    \[
    \widehat{V}_{ij}^2=\frac{1}{n(n-3)}\sum_{a,b=1}^n A^{(i)}_{ab}A^{(j)}_{ab}.
    \]
    The feasible distance covariance statistic is
    \[
    Z_{\mathrm{dCov}}
    =\binom{n}{2}^{1/2}
    \frac{\sum_{i<j}\widehat{V}_{ij}^2}
    {\left(\sum_{i<j}\widehat{V}_{ii}^2\widehat{V}_{jj}^2\right)^{1/2}}.
    \]

    \item \citet{leung2018testing}: We also include the Kendall-rank statistic $T_\tau$ defined above.
\end{enumerate}

To evaluate empirical size, we consider the following two null models for $\mX=(X_1,\dots,X_p)$:
\begin{enumerate}
    \item[(a)] $\mX\sim\mathcal{N}_p(\mf{0}_p,\mf{I}_p)$;
    \item[(b)] $X_1,\dots,X_p$ are independent and identically distributed (i.i.d.) as the standard Cauchy.
\end{enumerate}

To evaluate empirical power, we consider the following five alternative models:
\begin{enumerate}
    \item[(c)] \text{Linear dependence:} $\mX\sim\mathcal{N}_p(\mf{0}_p,\mf{\Sigma})$, where
    \[
    \mf{\Sigma}=(1-\rho)\mf{I}_p+\rho\mf{1}_p\mf{1}_p^\top,
    \qquad \rho=4/p;
    \]

    \item[(d)] \text{Nonlinear dependence:} Let
    \[
    \mZ\sim\mathcal{N}_p(\mf{0}_p,\mf{\Sigma}),\qquad
    \mf{\Sigma}=(1-\rho)\mf{I}_p+\rho\mf{1}_p\mf{1}_p^\top,
    \qquad \rho=5/p,
    \]
    and set
    \[
    X_i=\tan\{\pi(\Phi(Z_i)-1/2)\},\qquad i=1,\ldots,p;
    \]

    \item[(e)] \text{Oscillatory dependence:} $\mX=(\mU,\mV)$, where
    \[
    \mU\sim\mathcal{N}_{p/2}(\mf{0}_{p/2},\mf{I}_{p/2}), \quad
    \mV=\sin(2\pi\mU)+0.1\mZ, \quad
    \mZ\sim\mathcal{N}_{p/2}(\mf{0}_{p/2},\mf{I}_{p/2});
    \]

    \item[(f)] \text{W-shaped dependence:} $\mX=(\mU,\mV)$, where
    \[
    \mU\sim\mathcal{N}_{p/2}(\mf{0}_{p/2},\mf{I}_{p/2}), \quad
    \mV=|\mU+0.5|\ind(\mU<0)+|\mU-0.5|\ind(\mU\geq 0)+0.1\mZ,
    \]
    with $\mZ\sim\mathcal{N}_{p/2}(\mf{0}_{p/2},\mf{I}_{p/2})$;

    \item[(g)] \text{Shuffled copula dependence:} $\mX=(\mU,\mV)$, where $U_1,\ldots,U_{p/2}$ and $Z_1,\ldots,Z_{p/2}$ are mutually independent standard normal random variables. Fix $M=16$, and let $\varpi$ be any fixed permutation of $\{1,\ldots,M\}$. In the numerical experiments below, we take
    \[
    \varpi=(5,13,2,10,16,7,14,1,9,4,15,6,12,3,11,8).
    \]
    Define
    \[
    B_j=\lceil M\Phi(U_j)\rceil,\qquad j=1,\ldots,p/2,
    \]
    and set
    \[
    V_j=\frac{\varpi(B_j)-(M+1)/2}{\sqrt{(M^2-1)/12}}+0.6Z_j,
    \qquad j=1,\ldots,p/2.
    \]
\end{enumerate}

Model~(f) represents a typical case with zero linear correlation but nontrivial dependence. Model~(g) is a rank-local, nonmonotone alternative. Conditional on $U_j$, the mean of $V_j$ is obtained by shuffling the normal-quantile bins of $U_j$; the added Gaussian noise keeps every coordinate continuous. Thus the construction preserves strong copula-level dependence between paired coordinates while largely obscuring linear, monotone, and smooth global patterns. To our knowledge, this type of dependence has not been systematically examined in the existing literature, although it is precisely the kind of structure to which Chatterjee's $\xi$-based statistics are particularly sensitive.

Tables~\ref{tab:empirical_size} and~\ref{tab:empirical_power} report empirical rejection probabilities based on 5,000 Monte Carlo replications for each grid point $n=p\in\{50,70,100,200,300\}$. 
Under the Gaussian null Model~(a), all procedures have empirical sizes close to the nominal level, although $M_D$ is slightly oversized in lower dimensions. Under the Cauchy null Model~(b), Schott's Pearson-correlation-based statistic exhibits severe size distortion, whereas the rank-based and distance-covariance-based procedures remain substantially closer to the nominal level of $0.05$.

For the equicorrelated Gaussian linear Model~(c), Schott's test, $Z_{\mathrm{dCov}}$, $T_\tau$, and the Bonferroni test all exhibit high power, with the Bonferroni combination nearly matching Schott's test across all reported dimensions. Under the nonlinear Model~(d), the rank-based statistic $T_\tau$ is particularly effective, and the Bonferroni combination largely preserves this advantage, uniformly outperforming Schott's test, the three maximum-type tests, and $Z_{\mathrm{dCov}}$ over the reported dimensions. By contrast, $Q_{\xi,2}$ is especially powerful against the oscillatory and W-shaped dependence structures in Models~(e) and~(f). The Bonferroni combination inherits this strength while remaining competitive under Models~(c) and~(d).

Model~(g) yields an even sharper separation among the procedures. Already at $n=p=50$, $Q_{\xi,2}$ and $B_{\xi,\tau}$ attain empirical powers of $0.9000$ and $0.8408$, respectively, whereas every other procedure has power at most $0.1278$. For $n=p\geq 70$, both $Q_{\xi,2}$ and $B_{\xi,\tau}$ attain empirical power one. Although the power of the distance-covariance test increases with the dimension, reaching $0.8370$ at $n=p=300$, it remains below that of the Chatterjee-correlation-based procedures throughout. This rank-local shuffled-copula dependence model therefore illustrates a dependence regime that is not well captured by standard smooth or monotone alternatives and highlights a distinct advantage of $Q_{\xi,2}$ for high-dimensional mutual independence testing.

\begin{table*}
\centering
\caption{Empirical sizes under Models (a) and (b), based on 5,000 Monte Carlo replications}
\label{tab:empirical_size}
\begin{tabular}{llrrrrrrrr}
  \toprule
Model & $(n, p)$ & $Q_{r,2}$ & $M_D$ & $M_R$ & $M_{\tau^*}$ & $Z_{\mathrm{dCov}}$ & $T_\tau$ & $Q_{\xi,2}$ & $B_{\xi,\tau}$ \\
   \midrule
   \multirow{5}{*}{(a)}
  & $n=50, p=50$ & 0.0654 & 0.0980 & 0.0314 & 0.0420 & 0.0628 & 0.0698 & 0.0450 & 0.0634 \\
  & $n=70, p=70$ & 0.0504 & 0.0838 & 0.0400 & 0.0476 & 0.0552 & 0.0570 & 0.0422 & 0.0520 \\
  & $n=100, p=100$ & 0.0546 & 0.0734 & 0.0362 & 0.0420 & 0.0514 & 0.0618 & 0.0414 & 0.0526 \\
  & $n=200, p=200$ & 0.0486 & 0.0670 & 0.0448 & 0.0484 & 0.0514 & 0.0578 & 0.0514 & 0.0542 \\
  & $n=300, p=300$ & 0.0528 & 0.0592 & 0.0396 & 0.0448 & 0.0530 & 0.0536 & 0.0476 & 0.0494 \\
   \midrule
   \multirow{5}{*}{(b)}
    & $n=50, p=50$ & 0.2392 & 0.0894 & 0.0320 & 0.0410 & 0.0572 & 0.0706 & 0.0416 & 0.0602 \\
  & $n=70, p=70$ & 0.2674 & 0.0932 & 0.0396 & 0.0510 & 0.0518 & 0.0624 & 0.0408 & 0.0512 \\
  & $n=100, p=100$ & 0.3028 & 0.0688 & 0.0388 & 0.0436 & 0.0502 & 0.0562 & 0.0396 & 0.0494 \\
  & $n=200, p=200$ & 0.3576 & 0.0660 & 0.0452 & 0.0500 & 0.0488 & 0.0542 & 0.0464 & 0.0520 \\
  & $n=300, p=300$ & 0.3708 & 0.0542 & 0.0370 & 0.0402 & 0.0438 & 0.0554 & 0.0508 & 0.0590 \\
   \bottomrule
\end{tabular}
\end{table*}

\begin{table*}
\centering
\caption{Empirical powers under Models (c)--(g), based on 5,000 Monte Carlo replications}
\label{tab:empirical_power}
\begin{tabular}{llrrrrrrrr}
  \toprule
Model & $(n, p)$ & $Q_{r,2}$ & $M_D$ & $M_R$ & $M_{\tau^*}$ & $Z_{\mathrm{dCov}}$ & $T_\tau$ & $Q_{\xi,2}$ & $B_{\xi,\tau}$ \\
  \midrule
  \multirow{5}{*}{(c)}
	  & $n=50, p=50$ & 0.9788 & 0.2632 & 0.1178 & 0.1406 & 0.9588 & 0.9704 & 0.1130 & 0.9598 \\
	  & $n=70, p=70$ & 0.9910 & 0.2012 & 0.1016 & 0.1200 & 0.9756 & 0.9856 & 0.0850 & 0.9790 \\
	  & $n=100, p=100$ & 0.9970 & 0.1528 & 0.0840 & 0.0972 & 0.9896 & 0.9934 & 0.0758 & 0.9886 \\
	  & $n=200, p=200$ & 0.9996 & 0.1056 & 0.0712 & 0.0764 & 0.9978 & 0.9996 & 0.0626 & 0.9986 \\
	  & $n=300, p=300$ & 0.9998 & 0.0868 & 0.0590 & 0.0634 & 0.9996 & 0.9998 & 0.0594 & 0.9996 \\
   \midrule
   \multirow{5}{*}{(d)}
   & $n=50, p=50$ & 0.7830 & 0.3670 & 0.1678 & 0.2118 & 0.8886 & 0.9966 & 0.1816 & 0.9942 \\
  & $n=70, p=70$ & 0.7208 & 0.2722 & 0.1452 & 0.1680 & 0.8692 & 0.9990 & 0.1268 & 0.9984 \\
  & $n=100, p=100$ & 0.6512 & 0.2062 & 0.1174 & 0.1362 & 0.8708 & 0.9998 & 0.0920 & 0.9994 \\
  & $n=200, p=200$ & 0.5372 & 0.1302 & 0.0876 & 0.0960 & 0.8402 & 1.0000 & 0.0632 & 1.0000 \\
  & $n=300, p=300$ & 0.5086 & 0.1022 & 0.0732 & 0.0798 & 0.8118 & 1.0000 & 0.0638 & 1.0000 \\
   \midrule
   \multirow{5}{*}{(e)}
   & $n=50, p=50$ & 0.0550 & 0.1260 & 0.0376 & 0.0560 & 0.1506 & 0.0694 & 1.0000 & 1.0000 \\
  & $n=70, p=70$ & 0.0544 & 0.1184 & 0.0414 & 0.0534 & 0.2124 & 0.0634 & 1.0000 & 1.0000 \\
  & $n=100, p=100$ & 0.0532 & 0.1268 & 0.0346 & 0.0452 & 0.3234 & 0.0590 & 1.0000 & 1.0000 \\
  & $n=200, p=200$ & 0.0536 & 0.4878 & 0.0392 & 0.0646 & 0.7540 & 0.0550 & 1.0000 & 1.0000 \\
  & $n=300, p=300$ & 0.0514 & 1.0000 & 0.0536 & 0.1532 & 0.9714 & 0.0516 & 1.0000 & 1.0000 \\
   \midrule
   \multirow{5}{*}{(f)}
    & $n=50, p=50$ & 0.3806 & 0.9826 & 0.2168 & 0.5244 & 0.9836 & 0.0738 & 1.0000 & 1.0000 \\
  & $n=70, p=70$ & 0.3756 & 1.0000 & 0.4910 & 0.9294 & 1.0000 & 0.0660 & 1.0000 & 1.0000 \\
  & $n=100, p=100$ & 0.3810 & 1.0000 & 0.9836 & 1.0000 & 1.0000 & 0.0644 & 1.0000 & 1.0000 \\
  & $n=200, p=200$ & 0.3928 & 1.0000 & 1.0000 & 1.0000 & 1.0000 & 0.0574 & 1.0000 & 1.0000 \\
  & $n=300, p=300$ & 0.3876 & 1.0000 & 1.0000 & 1.0000 & 1.0000 & 0.0534 & 1.0000 & 1.0000 \\
   \midrule
   \multirow{5}{*}{(g)}
   & $n=50, p=50$ & 0.0470 & 0.0868 & 0.0298 & 0.0384 & 0.1278 & 0.0710 & 0.9000 & 0.8408 \\
  & $n=70, p=70$ & 0.0364 & 0.0810 & 0.0372 & 0.0448 & 0.1508 & 0.0556 & 1.0000 & 1.0000 \\
  & $n=100, p=100$ & 0.0428 & 0.0740 & 0.0398 & 0.0452 & 0.2304 & 0.0552 & 1.0000 & 1.0000 \\
  & $n=200, p=200$ & 0.0390 & 0.0626 & 0.0428 & 0.0462 & 0.5324 & 0.0542 & 1.0000 & 1.0000 \\
  & $n=300, p=300$ & 0.0416 & 0.0592 & 0.0434 & 0.0450 & 0.8370 & 0.0530 & 1.0000 & 1.0000 \\
   \bottomrule
\end{tabular}
\end{table*}

\section{Proof of main theorems}

\subsection{Graph notation}\label{sec:graph-notation}

The proofs of our main results, Theorems \ref{thm:sc}--\ref{thm:clt}, are based on the moment method, in particular on the moment convergence theorems of Riesz and Carleman \citep[Lemmas B.2 and B.3]{bai2010spectral}. This method involves calculating the limiting moments of the spectrum of $\mXi_n$, which requires estimating sums of expectations of products of matrix entries that typically reduces to counting the number of non-negligible terms. For this purpose, it is convenient to introduce some graph-theoretic notation.  

Throughout the remainder of the paper, let $\Delta$ denote a generic directed \emph{multigraph} with vertex set $V(\Delta)\subset[p]$, edge set $E(\Delta)=(e_k=(u_k,v_k))_k$, and no self-loops, i.e., $u\neq v$ for any $(u,v)\in E(\Delta)$. For such a $\Delta$, we define its \emph{underlying undirected graph} $\Delta^u$ by replacing each directed edge $(u_k,v_k)$ with the undirected edge $\{u_k,v_k\}$. Note that $\Delta^u$ is still allowed to be a multigraph; in particular, we do not merge multiple directed edges into a single undirected edge.  

A directed graph $\Delta$ is called a \emph{directed skeleton graph} if it contains no multiple edges. An undirected graph $\overline\Delta$  is called an \emph{undirected skeleton graph} if it contains no multiple edges. We denote by $\Delta^0$ the directed skeleton graph of $\Delta$, i.e.,
\[
V(\Delta^0)=V(\Delta) \quad \text{and} \quad 
E(\Delta^0)=\{e_k\in E(\Delta): e_j\neq e_k \text{ for all } j\neq k\}.
\]
We denote by $\overline{\Delta}^0$ the undirected skeleton graph of $\Delta$, which is defined as the skeleton graph of the underlying undirected graph $\Delta^u$, i.e.,
\[
V(\overline{\Delta}^0)=V(\Delta) \quad \text{and} \quad 
E(\overline{\Delta}^0)=\{e_k\in E(\Delta^u): e_j\neq e_k \text{ for all } j\neq k\}.
\]

For an edge $e=(u,v)\in E(\Delta^0)$, let $s_{uv}$ denote the multiplicity of $e$ in $\Delta$. For a vertex $v\in V(\Delta)$, let $N^+(v)$ and $N^-(v)$ denote its out-neighbor and in-neighbor sets in $\Delta$, respectively; that is,
\[
N^+(v)=\{u\in V(\Delta):\ (v,u)\in E(\Delta^0)\}, 
\qquad
N^-(v)=\{u\in V(\Delta):\ (u,v)\in E(\Delta^0)\}.
\]
We further define the out-degree and in-degree of $v$ by 
\[
d^+(v):=\# N^+(v), 
\qquad 
d^-(v):=\# N^-(v),
\]
where $\#$ denotes set cardinality.

\subsection{Combinatorics}\label{Section-combinatornics}
Recall the graph notation introduced in Section \ref{sec:graph-notation}. For any directed multigraph $\Delta$ considered in Section \ref{sec:graph-notation}, define
\bea\nonumber
\mf{\Xi}_n(\Delta)=\prod_{(u,v)\in E(\Delta^0)} (\mf{\Xi}_{uv}^{(n)})^{s_{uv}}.
\eea

The key step in the moment method is to estimate the asymptotic order of $\E[\mf{\Xi}_n(\Delta)]$ as $n\rw \infty$. We first give a graphical representation of $\E[\mf{\Xi}_n(\Delta)]$. Define 
\be\nonumber
a(i,j)=-|i-j|+\frac{n+1}{3},
\ee
and notice
\bea\nonumber
\sum_{1\leq i\ne j\leq n} a(i,j)=0,
\eea
\bea\nonumber
\mf{\Xi}_{ij}^{(n)}=\frac{3}{n^2-1}\sum_{k=1}^{n-1}a\Big(R_jR_i^{-1}(k),R_jR_i^{-1}(k+1)\Big).
\eea

\begin{proposition}(Graphical representation of $\E[\mf{\Xi}_n(\Delta)]$)\label{Graphical-Representation} Let $\sigma$ be the uniform random permutation on $[n]$. We then have
\bea\label{eqrep}
&\E[\mf{\Xi}_n(\Delta)]\\=&\Big(\frac{3}{n^2-1}\Big)^{\# E(\Delta)}\sum_{\cL(\Delta)}\prod_{v\in V(\Delta)}\E\Bigg[\prod_{\substack{u\in N^-(v)\\1\leq s\leq s_{uv}}}a\Big(\sigma(\ell_{uv}^{s,1}),\sigma(\ell_{uv}^{s,2})\Big)\cdot\prod_{\substack{w\in N^+(v)\\1\leq s\leq s_{vw}} }\ind\Big(\sigma(\ell_{vw}^{s,1})+1=\sigma(\ell_{vw}^{s,2})\Big) \Bigg],
\eea
where the summation $\sum_{\cL(\Delta)}$ has the summation index 
\[
\cL(\Delta)=\l\{\ell_{uv}^{s,r}:(u,v)\in E(\Delta^0), 1\leq s\leq s_{uv}, r\in \{1,2\}\r\},
\]
and each $\ell_{uv}^{s,r}$ goes from $1$ to $n$ satisfying $\ell_{uv}^{s,1}\ne \ell_{uv}^{s,2}$.
\end{proposition}

The graphical representation \eqref{eqrep} decomposes $\E[\mf{\Xi}_n(\Delta)]$ into the product of some expectation terms at each vertex of $\Delta$. It suggests that we can first study 
\be\nonumber
G_v := \E\Bigg[\prod_{\substack{u\in N^-(v)\\1\leq s\leq s_{uv}}}a\Big(\sigma(\ell_{uv}^{s,1}),\sigma(\ell_{uv}^{s,2})\Big)\cdot\prod_{\substack{w\in N^+(v)\\1\leq s\leq s_{vw}} }\ind\Big(\sigma(\ell_{vw}^{s,1})+1=\sigma(\ell_{vw}^{s,2})\Big) \Bigg],\quad v\in V(\Delta).
\ee


Notice that $G_v$ is the expectation of mixed product of some $a(\cdot,\cdot)$ terms and some $\ind(\cdot)$ terms. A natural idea is that we first study expectation of purely $a(\cdot,\cdot)$ terms, and purely $\ind(\cdot)$ terms, respectively. And then we combine them to estimate $G_v$.

To proceed, we first introduce some combinatorial notation.  
For any $m \in \mathbb{N}$, let $\mathcal{P}(m)$ denote the set of all partitions of $[m]$. That is, $\uppi = \{V_1, \ldots, V_k\} \in \mathcal{P}(m)$ means that the blocks of $\uppi$ are subsets $V_1, \ldots, V_k \subset [m]$ such that  
\[
\bigcup_{i=1}^k V_i = [m], 
\qquad 
V_i \bigcap V_j = \emptyset \quad \text{for } i \neq j.
\]  
We denote $\#\uppi = k$ for the number of blocks.  

For any $\uppi \in \mathcal{P}(m)$, by suitably relabeling the blocks we may write  
\[
\uppi = \{V_1, \ldots, V_k\} = \{V_1', \ldots, V_k'\}
\]
such that
\[
\min V_1' < \min V_2' < \cdots < \min V_k'.
\]  
The arrangement $\{V_1',\ldots,V_k'\}$ is unique. Define the map
\[
\varphi_{\uppi} : [m] \to [k], 
\qquad 
\varphi_{\uppi}(i) = j \ \text{ if and only if } \ i \in V_j'.
\]  
For brevity, we write $\uppi(i)$ in place of $\varphi_{\uppi}(i)$ for $i \in [m]$, and $\uppi(B)$ in place of $\varphi_{\uppi}(B)$ for $B \subset [m]$.  

For $W \subset [m]$ and $\uppi \in \mathcal{P}(m)$, the restriction of $\uppi$ to $W$, denoted $\uppi|_{W}$, is defined as the partition of $W$ determined by the restricted map $\varphi_{\uppi}|_W$.  

Finally, for any $L \in \mathbb{N}$, define
\[
\mathcal{RP}(L) 
= \Big\{\uppi \in \mathcal{P}(2L) : \uppi(2k-1) \neq \uppi(2k) \ \text{ for all } k \in [L] \Big\}.
\]  
Unless otherwise specified, we assume $L$ is fixed.

Lemmas \ref{lem1} and \ref{lem2} below provide estimates of pure $a(\cdot,\cdot)$ products and pure $\ind(\cdot)$ products, respectively.

\begin{lemma}\label{lem1}
    Let $\uppi\in \mathcal{RP}(L)$ and $\sigma$ be the uniform random permutation on $[n]$. Then as $n\rw \infty$,
    \be\label{11}
    \sum_{1\leq i_1\ne\cdots\ne i_{\# \uppi}\leq n}\prod_{k=1}^L a(i_{\uppi(2k-1)},i_{\uppi(2k)})=O\Big(n^{\# \uppi+L-\lceil \frac{s(\uppi)}{2} \rceil}\Big),
    \ee
    and
    \be\label{12}
    \E\Big[\prod_{k=1}^L a\Big(\sigma(\uppi(2k-1)),\sigma(\uppi(2k))\Big)\Big]=O\Big(n^{L-\lceil \frac{s(\uppi)}{2} \rceil}\Big),
    \ee    
    where 
    $$
    s(\uppi):=\# \Big\{k\in[L]:\ \uppi(2k-1)\ne \uppi(2k),\ \Big\{\uppi(2k-1),\uppi(2k)\Big\}\bigcap\uppi([m] \setminus \{2k-1,2k\})=\emptyset \Big\}.
    $$
\end{lemma}

\begin{remark}\label{remark:lem1}
For simplicity of the proof, we assume $\uppi\in \mathcal{RP}(L)$
in the lemma above. In general, for $\uppi\in \mathcal{P}(2L)$, the same estimates \eqref{11} and \eqref{12} still hold. To see this, we can apply \eqref{12} to the restriction of $\uppi$ on $$\bigcup_{\{k:\ \uppi(2k-1)\ne\uppi(2k)\}}\{2k-1,2k\}$$
and from $a(i,i)=\frac{n+1}{3}$ we obtain the same bound \eqref{12} for $\uppi$.
\end{remark}

Now for $\uppi\in \mathcal{RP}(L)$, define the directed graph $G_\uppi$ as the graph with vertices $\{1,2,\cdots,\# \uppi\}$ that connects a directed edge from $\uppi(2k-1)$ to $\uppi(2k)$ for all $k\in [L]$. Recall that $G_\uppi^0$ is the directed skeleton graph of $G_\uppi$ and $E(G_\uppi^0)$ is the edge set.

\begin{lemma}\label{lem2}
    Let $\uppi\in \mathcal{RP}(L)$ and $\sigma$ be the uniform random permutation on $[n]$. Then the following hold.
    \begin{enumerate}[itemsep=-.5ex,label=(\roman*)]
    \item The product
    \be\label{lem2prod}
    \prod_{k=1}^L \ind\Big(\sigma(\uppi(2k-1))+1=\sigma(\uppi(2k))\Big)
    \ee
    is non-zero if and only if $G_\uppi^0$ is the disjoint union of \emph{chains}. Here a chain is a directed graph  with edges $i_1\rw i_2,\ i_2\rw i_3,\cdots,i_{k-1}\rw i_{k}$.
        \item As $n\rw \infty$, \be\label{eq19}
    \E\Big[\prod_{k=1}^L \ind\Big(\sigma(\uppi(2k-1))+1=\sigma(\uppi(2k))\Big)\Big]=O\Big(n^{-\# E(G_\uppi^0)}\Big).
    \ee
    \end{enumerate}
    
\end{lemma}

We say that a partition $\uppi$ is \emph{valid} if $G_\uppi^0$ is a disjoint union of chains.  
Now let $\uppi \in \mathcal{RP}(L)$, and fix $L_1 < L$.  
Assume that the restriction $\uppi|_{[2L_1+1,2L]}$ is valid.  
We define the \emph{$*$-restriction} of $\uppi$ on $[2L_1]$, denoted $\uppi|^*_{[2L_1]}$, as the partition obtained from $\uppi|_{[2L_1]}$ by merging all $i,j \in [2L_1]$ with $i \neq j$ such that  
\[
\uppi(i), \uppi(j) \in \uppi([2L_1+1,2L]), \quad \uppi(i) \neq \uppi(j),
\]
and $\uppi(i)$ and $\uppi(j)$ belong to the same chain in $G_{\uppi|_{[2L_1+1,2L]}}^0$.

By merging $i$ and $j$ in a partition 
\[
\uppi = \{V_1, V_2, V_3, V_4, \ldots\}, \quad i \in V_1, \ j \in V_2,
\]
we mean obtaining the new partition
\[
\{V_1 \cup V_2, V_3, V_4, \ldots\}.
\]

For the special case $L_1 = L$, we simply set 
\[
\uppi|^*_{[2L_1]} = \uppi|_{[2L_1]}.
\]

By combining Lemmas \ref{lem1}-\ref{lem2} and carefully handling the interactions between $a(\cdot,\cdot)$ terms and $\ind(\cdot)$ terms, Proposition \ref{prop2} below gives estimates of the mixed products.

\begin{proposition}\label{prop2}
Let $\uppi\in \mathcal{RP}(L),\ L_1\leq L$, and $\sigma$ be the uniform random permutation on $[n]$. Then, as $n\rw \infty$,
\begin{equation}\nonumber
\begin{aligned}
&\E\Bigg[
\prod_{k=1}^{L_1} a\big(\sigma(\uppi(2k-1)), \sigma(\uppi(2k))\big)
\prod_{k=L_1+1}^L \ind\big(\sigma(\uppi(2k-1)) + 1 = \sigma(\uppi(2k))\big)
\Bigg] \\
&= 
\begin{cases}
O\Bigl(n^{-\# \uppi + L_1 - \big\lceil \frac{s(\uppi_1^*)}{2} \big\rceil + \# \uppi_1^* + h(\uppi)}\Bigr), & \text{if } \uppi|_{[2L_1+1, 2L]} \text{ is valid}, \\
0, & \text{otherwise.}
\end{cases}
\end{aligned}
\end{equation}
where $\uppi_1^*=\uppi|^*_{[2L_1]}$, and $h(\uppi)$ is the number of chains $A$ in $G_{\uppi|_{[2L_1+1,2L]}}^0$ such that $\uppi([2L_1])\cap A=\emptyset$.
\end{proposition}

Now we are ready to give an estimate of
$$
G_v=\E\Bigg[\prod_{\substack{u\in N^-(v)\\1\leq s\leq s_{uv}}}a\Big(\sigma(\ell_{uv}^{s,1}),\sigma(\ell_{uv}^{s,2})\Big)\cdot\prod_{\substack{w\in N^+(v)\\1\leq s\leq s_{vw}} }\ind\Big(\sigma(\ell_{vw}^{s,1})+1=\sigma(\ell_{vw}^{s,2})\Big) \Bigg],\quad v\in V(\Delta).
$$
For each $v \in V(\Delta)$, define
\[
\cL^-_v(\Delta) = \{\ell_{uv}^{s,j} : u \in N^-(v), \ 1 \leq s \leq s_{uv}, \ j \in \{1,2\}\},
\]
\[
\cL^+_v(\Delta) = \{\ell_{uv}^{s,j} : u \in N^+(v), \ 1 \leq s \leq s_{uv}, \ j \in \{1,2\}\},
\]
and
\[
\cL_v(\Delta) = \cL^+_v(\Delta) \cup \cL^-_v(\Delta).
\]

The definition of partitions can be extended to arbitrary finite sets. In particular, set
\[
\mathcal{RP}(\cL(\Delta)) 
:= \Big\{ \uppi \in \mathcal{P}(\cL(\Delta)) : 
\uppi(\ell_{uv}^{s,1}) \neq \uppi(\ell_{uv}^{s,2}), \ 
\text{for any } (u,v) \in E(\Delta^0), \ 1 \leq s \leq s_{uv} \Big\}.
\]
Let $\uppi \in \mathcal{RP}(\cL(\Delta))$ and denote by $\uppi^v = \uppi|_{\cL_v(\Delta)}$ its restriction to $\cL_v(\Delta)$.  
We write $\uppi^{v*}_1$ for the $*$-restriction of $\uppi^v$ in $N^-(v)$, interpreted as the $*$-restriction on $[2L]$ with
\[
L = d^-(v) + d^+(v), 
\qquad L_1 = d^-(v).
\]
The definition of $h(\uppi^v)$ is inherited accordingly.  

Finally, for $\uptau \in \mathcal{P}(L^-_v(\Delta))$, 
we define
{\small
\[
s(\uptau) := \# \Big\{ (u,s) : u \in N^-(v), \ 1 \leq s \leq s_{uv}, \ 
\uptau(\ell_{uv}^{s,1}) \neq \uptau(\ell_{uv}^{s,2}), \ 
\{\uptau(\ell_{uv}^{s,1}), \uptau(\ell_{uv}^{s,2})\} \cap 
\uptau\!\Big(L^-_v(\Delta) \setminus \{\ell_{uv}^{s,1}, \ell_{uv}^{s,2}\}\Big) 
= \emptyset \Big\}.
\]
}

By applying proposition \ref{prop2} to $\uppi^v$ with $L_1=d^-(\Delta)$, we have the following estimate of $G_v$.

\begin{proposition}\label{Local-Estimate-Partition}
    For $v\in V(\Delta),\ \uppi\in \mathcal{RP}(\cL(\Delta))$, as $n\rw\infty$,
\bea\nonumber
G_v(\uppi):=&\E\Bigg[\prod_{\substack{u\in N^-(v)\\1\leq s\leq s_{uv}}}a\Big(\sigma\uppi^v(\ell_{uv}^{s,1}),\sigma\uppi^v(\ell_{uv}^{s,2})\Big)\cdot\prod_{\substack{w\in N^+(v)\\1\leq s\leq s_{vw}} }\ind\Big(\sigma\uppi^v(\ell_{vw}^{s,1})+1=\sigma\uppi^v(\ell_{vw}^{s,2})\Big) \Bigg]\\
=&\left\{\begin{array}{cc}     O\Big(n^{\chi_v(\uppi)}\Big)\quad &\text{if }\uppi|_{\cL_v^+(\Delta)}\text{ is valid,} \\
      0 &\text{otherwise,}
\end{array}
\right.
\eea
where 
$$
\chi_v(\uppi):=-\# \uppi^v+d^-(v)-\Big\lceil\frac
{s(\uppi_1^{v*})}{2}\Big\rceil+\# \uppi_1^{v*}+h(\uppi^v).
$$
\end{proposition}

From Proposition \ref{Graphical-Representation}, we have
$$
\E[\mf{\Xi}_n(\Delta)]=\l(\frac{3}{n^2-1}\r)^{\# E(\Delta)}\sum_{\cL(\Delta)}\prod_{v\in V(\Delta)}G_v=\l(\frac{3}{n^2-1}\r)^{\# E(\Delta)}\sum_{\uppi\in\ml{RP}(\cL(\Delta))}(n)_{\# \uppi}G_v(\uppi).
$$
Together with Proposition \ref{Local-Estimate-Partition}, we now have a rough estimate for the asymptotic order of $\E[\mf{\Xi}_n(\Delta)]$. However, such estimate is not sharp enough. To explain this, recall that in the classical proof of Wigner matrices' semi-circle law with moment method, the contribution of those graphs containing {\it single edges} (i.e. edges of multiplicity 1) is 0, due to the zero-mean and independence of entries. However, in our problem, graphs with single edges can no longer be ruled out trivially, due to dependence between entries. But up to now, our estimate on $\E[\mf{\Xi}_n(\Delta)]$ has not exploited the zero-mean properties of entries yet. Hence we need an extra argument to deal with these single edges, and then the estimate of $\E[\mf{\Xi}_n(\Delta)]$ can be sharp enough for LSD/CLT results. The idea is direct: although graphs with single edges have non-zero contribution, they should contribute less if it contains more single edges. 

Letting $Q_0$ be the number of single edges in $\Delta$, we finally have the following proposition that proves sharp enough.

\begin{proposition}\label{Graph-Estimate-EDelta0} As $n\rw \infty$, 
    $$
    \E[\mf{\Xi}_n(\Delta)]=O(n^{-(\# E(\Delta^0)\land (\# E(\Delta)+Q_0)/2)}).
    $$
\end{proposition}



\subsection{Proof of Theorem \ref{thm:sc}}\label{section-LSD-SC}

The moments of W$(0,r)$
are
$$
\int_\R x^{2m}\rho_{0,r}(x)dx=\frac{r^{2m}(2m)!}{m!(m+1)!},\quad \int_\R x^{2m+1}\rho_{0,r}(x)dx=0,\quad m\in \N.
$$

For $k \in \mathbb{N}^+$, let $\mathbf{i} = (i_1, \ldots, i_k)$ with the constraint that $i_u \in [p]$,  
$i_u \neq i_{u+1}$ for all $u \in [k]$, and $i_1 \neq i_k$.  
By convention, set $i_{k+1} = i_1$. Define $\overline{T}(\mathbf{i})$ as the undirected multigraph with vertex set 
\[
\Big\{ i_u : u \in [k] \Big\},
\]
and edge set
\[
\Big\{ \{i_u, i_{u+1}\} : u \in [k] \Big\}.
\]
An undirected multigraph $\overline{\Delta}$ is called a \emph{$\overline{\Gamma}$-graph} if 
$\overline{\Delta} = \overline{T}(\mathbf{i})$ for some $\mathbf{i}$.  
It is denoted by $\overline{\Gamma}(k,t)$ if 
\[
\mathbf{i} = (i_1, \ldots, i_k), 
\qquad 
\# \{ i_u : u \in [k] \} = t.
\]  
Denote by $\overline{\Gamma}^0(k,t)$ the \emph{undirected skeleton graph} of $\overline{\Gamma}(k,t)$.  Two $\overline{\Gamma}(k,t)$-graphs are said to be \emph{isomorphic} if one can be obtained from the other by relabeling the vertices via a permutation on $[p]$.

Within each isomorphism class, there exists a unique representative, called the \emph{canonical $\overline{\Gamma}$-graph}, which satisfies
\[
i_1 = 1, 
\qquad 
i_j \leq \max\{i_1, \ldots, i_{j-1}\} + 1, 
\qquad j \geq 2.
\]
For an undirected multigraph $\overline{\Delta}$, define
\[
\mf{\Phi}_n(\overline{\Delta})
:= \mathbb{E}\!\left[ \prod_{\{u,v\} \in E(\overline{\Delta}^0)} 
\big(\Phi_{uv}^{(n)}\big)^{s_{uv}} \right].
\]
This quantity is well defined since $\mf{\Phi}_n$ is symmetric.

\begin{lemma}\label{tree-independence-semicircle}
    If an undirected skeleton graph $\overline{\Delta}^0$ is a tree, then $\big\{\Phi_{uv}^{(n)}:\{u,v\}\in E(\overline{\Delta}^0)\big\}$ is a family of independent random variables.
\end{lemma}

\begin{lemma}\label{EPhi(Delta)-semicirclelaw}
    Let $\overline{\Delta}$ be a $\overline{\Gamma}$-graph. Then, as $n\rw\infty$, we have 
    $$    \E[\mf{\Phi}_n(\overline{\Delta})]=O(n^{-\# E(\overline{\Delta}^0)}).
    $$
\end{lemma}

\begin{proposition}\label{Emoments-convergence-semicircle} For any fixed $m\in \mb{N}$, we have
    \bea\nonumber
    \lim_{n\rw\infty} \frac{1}{p}\E[\operatorname{tr}\big((\mf{\Phi}_n-\mf{I}_p)^{2m}\big)]=\left(\frac{\gamma}{5}\right)^{m}\frac{(2m)!}{m!(m+1)!},
    \eea
    and
    \bea\nonumber
    \lim_{n\rw\infty} \frac{1}{p}\E[\operatorname{tr}\big((\mf{\Phi}_n-\mf{I}_p)^{2m-1}\big)]=0.
    \eea
\end{proposition}

\begin{proof}
We classify all $\overline{\Gamma}$-graphs into the following three categories:

\begin{itemize}
    \item \textbf{Category 1 ($D_1$):} 
    $\overline{\Gamma}$-graphs in which every undirected edge has multiplicity exactly $2$, and the corresponding undirected skeleton graph is a tree;
    
    \item \textbf{Category 2 ($D_2$):}
    $\overline{\Gamma}$-graphs whose undirected skeleton graph is not a tree; 
    
    \item \textbf{Category 3 ($D_3$):}
    $\overline{\Gamma}$-graphs that do not belong to $D_1$, but whose corresponding undirected skeleton graph is a tree. 
\end{itemize}

For $k\in \N$, we obtain 
\bea\nonumber
\frac{1}{p}\E[\operatorname{tr}\big((\mf{\Phi}_n-\mf{I}_p)^{k}\big)]&=\frac{1}{p}\sum_{\i}\E[\Phi_{i_1i_2}^{(n)}\Phi_{i_2i_3}^{(n)}\cdots\Phi_{i_{k-1}i_k}^{(n)}\Phi_{i_ki_1}^{(n)}]\\
&=\frac{1}{p}\sum_{\i}\E[\mf{\Phi}_n(\overline{T}(\i))]\\
&=K_1+K_2+K_3.
\eea

where
\begin{equation}\nonumber
\begin{aligned}
       K_j&:=\frac{1}{p} \sum_{\overline{\Gamma}(k,t) \in D_j} \sum_{\overline{T}(\i) \simeq \overline{\Gamma}(k,t)} \E[\mf{\Phi}_n(\overline{T}(\i))],\quad j\in [3].
\end{aligned}
\end{equation}
Here $\sum_{\overline{\Gamma}(k,t) \in D_j}$ is taken on all different canonical $\overline{\Gamma}$-graphs in Category $j$, and  $\sum_{\overline{T}(\i) \simeq \overline{\Gamma}(k,t)}$ is taken on all isomorphic graphs of the canonical $\overline{\Gamma}$-graph $\overline{\Gamma}(k,t)$. Since $\E[\mf{\Phi}_n(\overline{T}(\i))]$ is unchanged under graph isomorphisms, we have

\begin{equation}\nonumber
\begin{aligned}
       K_j&=\frac{1}{p} \sum_{\overline{\Gamma}(k,t) \in D_j} p^{t}(1+o(1))\E[\mf{\Phi}_n(\overline{\Gamma}(k,t))],\quad j\in[3].
\end{aligned}
\end{equation}

\textbf{Category 1.} When $k=2m+1$, $K_1=0$. When $k=2m$, the number of different canonical $\overline{\Gamma}$-graphs in $D_1$ equals to
$$
\frac{(2m)!}{m!(m+1)!}.
$$
From \citet[Corollary 1]{zhang2022asymptotic}, we find
\bea\nonumber
\E[\Phi_{12}^{(n)}]=\frac{1}{2}\E[\Xi_{12}^{(n)}]+\frac{1}{2}\E[\Xi_{12}^{(n)}\Xi_{21}^{(n)}]=\frac{1}{5n}+O(n^{-2}).
\eea
Then from Lemma \ref{tree-independence-semicircle}, 
\bea\nonumber
K_1&=\frac{1}{p}\cdot p^{m+1}(1+o(1))\frac{(2m)!}{m!(m+1)!}\l(\frac{1}{5n}+O(n^{-2})\r)^m\\
&=\left(\frac{\gamma}{5}\right)^{m}\frac{(2m)!}{m!(m+1)!}+O(n^{-1}).
\eea

\textbf{Category 2.}  For any $\overline{\Gamma}(k,t)\in D_2$, from Lemma \ref{EPhi(Delta)-semicirclelaw}, 
$$
\E[\mf{\Phi}_n(\overline{\Gamma}(k,t))]=O(n^{-\# E(\overline{\Gamma}^0(k,t))}).
$$
Since $\overline{\Gamma}^0(k,t)$ is not a tree, $t\leq \# E(\overline{\Gamma}^0(k,t))$. Then
\bea\nonumber
K_2=\sum_{\overline{\Gamma}(k,t)\in D_2} O(n^{t-1-\# E(\overline{\Gamma}^0(k,t)})=O(n^{-1}).
\eea

\textbf{Category 3.}  Now let $\overline{\Gamma}(k,t)\in D_3$. If $\overline{\Gamma}(k,t)$ has an undirected edge of multiplicity 1, from Lemma \ref{tree-independence-semicircle} and $\E[\Phi_{12}^{(n)}]=0$, we have $\E[\mf{\Phi}_n(\overline{\Gamma}(k,t))]=0$. Then it suffices to consider those $\overline{\Gamma}(k,t)$ such that each edge has multiplicity at least $2$, and there is at least one edge of multiplicity at least 3. Then $t\leq (k+1)/2$. Then
\bea\nonumber
K_3=\sum_{\overline{\Gamma}(k,t)\in D_3} O(n^{t-1-\frac{k}{2}})=O(n^{-1/2})
\eea
and we finish the proof.
\end{proof}

\begin{proposition}\label{moments-concentration-semicircle}
    For any fixed $k\in \mb{N}$, we have, as $n\rw\infty$,
    \bea\nonumber
    \operatorname{Var}\l(\frac{1}{p}\operatorname{tr}((\mf{\Phi}_n-\mf{I}_p)^k)\r)=O(n^{-2}).
    \eea
\end{proposition}

\begin{proof}
    We have
    \bea\nonumber
    \operatorname{Var}\l(\frac{1}{p}\operatorname{tr}((\mf{\Phi}_n-\mf{I}_p)^k)\r)
    =\frac{1}{p^2}\sum_{\i_1,\i_2}\big(\E[\mf{\Phi}_n(\overline{T}(\i_1))\mf{\Phi}_n(\overline{T}(\i_2))]-\E[\mf{\Phi}_n(\overline{T}(\i_1))]\E[\mf{\Phi}_n(\overline{T}(\i_2))]\big). 
    \eea
    Write $\overline{T}_1=\overline{T}(\i_1)$ and $\overline{T}_2=\overline{T}(\i_2)$ for short, and denote $\overline{T}=\overline{T}_1\cup \overline{T}_2$. Note that $\overline{T}_1,\overline{T}_2$ are both $\overline{\Gamma}$-graphs. We classify all possible graph pairs $(\overline{T}_1,\overline{T}_2)$ into the following four categories:

\begin{itemize}
    \item {\bf Category 1 ($\overline{D}_1$):} $V(\overline{T}_1)\cap V(\overline{T}_2)=\emptyset$;

    \item {\bf Category 2 ($\overline{D}_2$):} $V(\overline{T}_1)\cap V(\overline{T}_2)\ne\emptyset$, and $\overline{T}^0$ is a tree;

    \item {\bf Category 3 ($\overline{D}_3$):} $V(\overline{T}_1)\cap V(\overline{T}_2)\ne\emptyset$, $\overline{T}^0$ is not a tree, and both $\overline{T}_1^0$ and $\overline{T}_2^0$ are trees;

    \item {\bf Category 4 ($\overline{D}_4$):} $V(\overline{T}_1)\cap V(\overline{T}_2)\ne\emptyset$, $\overline{T}^0$ is not a tree, and at least one of $\overline{T}_1^0$ and $\overline{T}_2^0$ is not a tree.
    \end{itemize}
    
    We can then decompose the variance to 
    \bea\nonumber
    \operatorname{Var}\l(\frac{1}{p}\operatorname{tr}((\mf{\Phi}_n-\mf{I}_p)^k)\r)=\overline{K}_1+\overline{K}_2+\overline{K}_3+\overline{K}_4,
    \eea
    where, for each $j\in[4]$,
    \bea\nonumber
    \overline{K}_j:=\frac{1}{p^2}\sum_{(\overline{T}_1,\overline{T}_2)\in \overline{D}_j}p^{\# V(T)}\big(1+o(1)\big)\big(\E[\mf{\Phi}_n(\overline{T}_1)\mf{\Phi}_n(\overline{T}_2)]-\E[\mf{\Phi}_n(\overline{T}_1)]\E[\mf{\Phi}_n(\overline{T}_2)]\big).
    \eea
    Here the summation is taken over all different graph pairs $(\overline{T}_1,\overline{T}_2)$ in Category $j$ up to graph isomorphism.

    {\bf Category 1.} For any $(\overline{T}_1,\overline{T}_2)\in \overline{D}_1$,\ $\mf{\Phi}_n(\overline{T}_1)$ and $\mf{\Phi}_n(\overline{T}_2)$ are independent. Then 
    \bea\label{moments-concentration-semicircle-H1}
    \overline{K}_1=0.
    \eea

    {\bf Category 2.}    For any $(\overline{T}_1,\overline{T}_2)\in \overline{D}_2$, if 
$E(\overline{T}_1)\cap E(\overline{T}_2)=\emptyset$, then by Lemma~\ref{tree-independence-semicircle}, 
$\mf{\Phi}_n(\overline{T}_1)$ and $\mf{\Phi}_n(\overline{T}_2)$ are independent. 
If $\overline{T}$ contains an edge with multiplicity $1$, then
\[
\E\big[\mf{\Phi}_n(\overline{T}_1)\mf{\Phi}_n(\overline{T}_2)\big]
= \E[\mf{\Phi}_n(\overline{T}_1)] \, \E[\mf{\Phi}_n(\overline{T}_2)] = 0.
\]
If $E(\overline{T}_1)\cap E(\overline{T}_2)\neq \emptyset$ and each edge in $\overline{T}$ has multiplicity at least $2$, then $\overline{T}$ must contain an edge of multiplicity at least $4$. 
In this case, $\# V(\overline{T})\leq k$. 
Since 
\[
\E[\mf{\Phi}_n(\overline{T}_1)] = O(n^{-k/2}), 
\qquad 
\E[\mf{\Phi}_n(\overline{T}_2)] = O(n^{-k/2}), 
\qquad 
\E[\mf{\Phi}_n(\overline{T}_1)\mf{\Phi}_n(\overline{T}_2)] = O(n^{-k}),
\]
we obtain
\begin{equation}\label{moments-concentration-semicircle-H2}
\overline{K}_2 
= \sum_{(\overline{T}_1,\overline{T}_2)\in \overline{D}_2} 
O\big(n^{\# V(\overline{T}) - k - 2}\big) 
= O(n^{-2}).
\end{equation}

    {\bf Category 3.}     For any $(\overline{T}_1,\overline{T}_2)\in \overline{D}_3$, there exists an edge $\{u,v\}\in \overline{T}_1^0$ such that either $\{u,v\}\in \overline{T}_2^0$ or $\{v,u\}\in \overline{T}_2^0$. 
Otherwise, since both $\overline{T}_1^0$ and $\overline{T}_2^0$ are trees, 
$\overline{T}^0$ would also be a tree. In that case, 
\[
\# V(\overline{T}) \;\leq\; \# V(\overline{T}_1) + \# V(\overline{T}_2) - 2.
\]

From Lemma~\ref{tree-independence-semicircle}, we have
\[
\E[\mf{\Phi}_n(\overline{T}_1)] = O\big(n^{-\# E(\overline{T}_1^0)}\big), 
\quad 
\E[\mf{\Phi}_n(\overline{T}_2)] = O\big(n^{-\# E(\overline{T}_2^0)}\big), 
\quad 
\E[\mf{\Phi}_n(\overline{T}_1)\mf{\Phi}_n(\overline{T}_2)] = O\big(n^{-\# E(\overline{T}^0)}\big).
\]
Since $\overline{T}^0$ is not a tree, we have $\# V(\overline{T}) \leq \# E(\overline{T}^0)$. 
Moreover, because 
\[
\# V(\overline{T}_1) = \# E(\overline{T}_1^0) + 1, 
\qquad 
\# V(\overline{T}_2) = \# E(\overline{T}_2^0) + 1,
\]
it follows that
\begin{equation}\label{moments-concentration-semicircle-H3}
\overline{K}_3
= \sum_{(\overline{T}_1,\overline{T}_2)\in \overline{D}_3} 
O\!\left(n^{\# V(\overline{T}) - \# E(\overline{T}^0) - 2}\right) 
+ O\!\left(n^{\# V(\overline{T}_1) + \# V(\overline{T}_2) 
- \# E(\overline{T}_1^0) - \# E(\overline{T}_2^0) - 4}\right)
= O(n^{-2}).
\end{equation}

    {\bf Category 4.}     For any $(\overline{T}_1,\overline{T}_2)\in \overline{D}_4$, without loss of generality assume that 
$\overline{T}_1^0$ is not a tree. Then 
\[
\# V(\overline{T}_1) \;\leq\; \# E(\overline{T}_1^0).
\]
From Lemma~\ref{tree-independence-semicircle}, together with the inequalities
\[
\# V(\overline{T}_2) \leq \# E(\overline{T}_2^0)+1, 
\qquad 
\# V(\overline{T}) \leq \# E(\overline{T}^0)+1, 
\qquad 
\# V(\overline{T}) \leq \# V(\overline{T}_1) + \# V(\overline{T}_2) - 1,
\]
we obtain
\begin{equation}\label{moments-concentration-semicircle-H4}
\overline{K}_4
= \sum_{(\overline{T}_1,\overline{T}_2)\in \overline{D}_4} 
O\!\left(n^{\# V(\overline{T}) - \# E(\overline{T}^0) - 2}\right) 
+ O\!\left(n^{\# V(\overline{T}_1) + \# V(\overline{T}_2) 
- \# E(\overline{T}_1^0) - \# E(\overline{T}_2^0) - 3}\right)
= O(n^{-2}).
\end{equation}

    Combining \eqref{moments-concentration-semicircle-H1}-\eqref{moments-concentration-semicircle-H4}, we thus finish the proof.
\end{proof}

Leveraging Propositions \ref{Emoments-convergence-semicircle} and \ref{moments-concentration-semicircle}, combined with Borel-Cantelli lemma, we thus proved Theorem \ref{thm:sc}.

\subsection{Proof of Theorem \ref{thm:mp}}\label{Section-LSD-MP}


The moments of MP$(y,\sigma^2)$ are 
$$
\int_\mb{R} x^k\mu_{y,\sigma^2}(dx)=\sigma^{2k}\sum_{r=0}^{k-1}\frac{y^r}{r+1}\binom{k}{r}\binom{k-1}{r},\quad k\in \mb{N}.
$$

We begin by introducing some notation borrowed from \cite{bai2010spectral}.  
Let \[ i_1, j_1, \ldots, i_k, j_k \in [p] \] with the conditions
\[
i_1 \neq j_k, 
\qquad 
i_u \neq j_u \quad \text{for all } u \in [k].
\]
A \emph{$\Delta$-graph} is a directed multigraph constructed as follows.  
 \begin{enumerate}[itemsep=-.5ex,label=(\roman*)]
\item Draw two parallel horizontal lines, referred to as the \emph{$I$-line} and the \emph{$J$-line}.  
On the $I$-line, place the points \( i_1, \ldots, i_k \);  
on the $J$-line, place the points \( j_1, \ldots, j_k \).  
\item Next, draw $k$ downward edges connecting \( i_u \) to \( j_u \) for each \( u \in [k] \),  
and another $k$ downward edges connecting \( i_{u+1} \) to \( j_u \) for each \( u \in [k] \),  
where we adopt the cyclic convention \( i_{k+1} = i_1 \).  
\end{enumerate}

We denote this $\Delta$-graph by
\[
G(\mathbf{i}, \mathbf{j}), 
\qquad 
\mathbf{i} = (i_1, \ldots, i_k), \quad 
\mathbf{j} = (j_1, \ldots, j_k).
\]


Two $\Delta$-graphs \( G(\mathbf{i}_1, \mathbf{j}_1) \) and \( G(\mathbf{i}_2, \mathbf{j}_2) \) are said to be \emph{isomorphic} if one can be obtained from the other by relabeling the vertices via a permutation on $[p]$.  

Within each isomorphism class, there exists a unique representative, called the \emph{canonical $\Delta$-graph}, which satisfies the following conditions:  
\[
i_1 = j_1 = 1, 
\qquad 
i_u \leq \max\{i_1, \dots, i_{u-1}\} + 1, 
\qquad 
j_u \leq \max\{j_1, \dots, j_{u-1}\} + 1.
\]

A canonical $\Delta$-graph is denoted by \(\Delta(k,r,s)\) if it contains \(r+1\) distinct $I$-vertices, \(s\) distinct $J$-vertices, and $2k$ edges. Note that \(\Delta(k,r,s)\) is also a directed multigraph. We denote by \(\Delta^0(k,r,s)\) and \(\overline{\Delta}^0(k,r,s)\) the directed and undirected skeleton graph of \(\Delta(k,r,s)\), respectively.

\begin{proposition}\label{Emoments-convergence-MPlaw} For any fixed $k\in \mb{N}$, we have
    \bea\nonumber
    \lim_{n\rw\infty} \frac{1}{p}\E\l[\operatorname{tr}\big(\mf{\Psi}_n^{k}\big)\r]=\left(\frac{2\gamma}{5}\right)^{k}\sum_{r=0}^{k-1}\frac{1}{r+1}\binom{k}{r}\binom{k-1}{r}.
    \eea
\end{proposition}


\begin{proof}
We classify all $\Delta$-graphs into the following three categories.
\begin{itemize}
    \item \textbf{Category 1 ($C_1$):} $\Delta$-graphs in which every directed edge has multiplicity exactly $2$, and the corresponding undirected skeleton graph is a tree. 
    
    \item \textbf{Category 2 ($C_2$):} $\Delta$-graphs whose 
    undirected skeleton graph is not a tree.
    
    \item \textbf{Category 3 ($C_3$):} $\Delta$-graphs that do not belong to $C_1$, but the corresponding undirected skeleton graph is a tree.
\end{itemize}

Then we have
\begin{equation}\nonumber
\begin{aligned}
\frac{1}{p}\E\left[ \operatorname{tr}\left(\mf{\Psi}_n^k\right)\right]  & =\frac{1}{p} \sum_{\mf{i},\mf{j}} \E\left[\mf{\Xi}_n\left(i_1, j_1\right) \mf{\Xi}_n^{\tp}\left(j_1, i_2\right) \mf{\Xi}_n\left(i_2, j_2\right) \mf{\Xi}_n^\tp\left(j_2, i_3\right) \cdots \mf{\Xi}_n\left(i_k, j_k\right) \mf{\Xi}_n^{\tp}\left(j_k, i_1\right)\right] \\
& =\frac{1}{p}\sum_{\mf{i},\mf{j}}\E[\mf{\Xi}_n(G(\mf{i},\mf{j}))]\\
& =H_1+H_2+H_3,
\end{aligned}
\end{equation}
where
\begin{equation}\nonumber
\begin{aligned}
       H_j&:=\frac{1}{p} \sum_{\Delta(k, r, s) \in C_j} \sum_{G(\mf{i},\mf{j}) \simeq \Delta(k, r, s)} \E[\mf{\Xi}_n(G(\mf{i},\mf{j}))],\quad j\in[3].
\end{aligned}
\end{equation}
Here $\sum_{\Delta(k, r, s) \in C_j}$ is taken over all different canonical $\Delta$-graphs in Category $j$, and  $\sum_{G(\mf{i},\mf{j}) \simeq \Delta(k, r, s)}$ is taken on all isomorphic graphs of the canonical $\Delta(k, r,s)$-graph $\Delta(k,r,s)$. Since $\E[\mf{\Xi}_n(G(\mf{i},\mf{j}))]$ is unchanged under graph isomorphisms, we have

\begin{equation}\nonumber
\begin{aligned}
       H_j&=\frac{1}{p} \sum_{\Delta(k, r, s) \in C_j} p^{r+s+1}(1+o(1))\E[\mf{\Xi}_n(\Delta(k,r,s))],\quad j\in[3].
\end{aligned}
\end{equation}

    {\bf Category 1.} From  \citet[Lemma 1.4]{bai2010spectral}, the number of different canonical graphs $\Delta(k,r,s)$ in $C_1$ equals to
$$
\frac{1}{r+1}\binom{k}{r}\binom{k-1}{r}.
$$
For any $\Delta(k,r,s)\in C_1$, from Corollary \ref{Corollary:tree-independence-coef},
\bea\nonumber
\E[\mf{\Xi}_n(\Delta(k,r,s))]
&=\E\l[\prod_{(u,v)\in E(\Delta^0(k,r,s))}({\Xi}^{(n)}_{uv})^{2}\r]\\
&=\prod_{(u,v)\in E(\Delta^0(k,r,s))}\E\l[({\Xi}^{(n)}_{uv})^{2}\r]\\&=\l(\frac{2}{5n}+O(n^{-2})\r)^k.
\eea
Then
\bea\label{Emoments-convergence-MPlaw-H1}
H_1&=\frac{1}{p}\sum_{r=0}^{k-1}\frac{1}{r+1}\binom{k}{r}\binom{k-1}{r}p^{r+s+1}(1+o(1))\l(\frac{2}{5n}+O(n^{-2})\r)^k\\
&=\left(\frac{2\gamma}{5}\right)^{k}\sum_{r=0}^{k-1}\frac{1}{r+1}\binom{k}{r}\binom{k-1}{r}+O(n^{-1}).
\eea

    {\bf Category 2.} For any $\Delta(k,r,s)\in C_2$, from Proposition \ref{Graph-Estimate-EDelta0}, $$\E[\mf{\Xi}_n(\Delta(k,r,s))]=O(n^{-\# E(\Delta^0(k,r,s))}).$$ 
As $\overline\Delta^0(k,r,s)$ is not a tree, it implies that 
$$
r+s+1=\# V(\Delta^0(k,r,s))<\# E(\Delta^0(k,r,s))+1.
$$
Then
\bea\label{Emoments-convergence-MPlaw-H2}
H_2&=\frac{1}{p}\sum_{\Delta(k, r, s) \in C_2}p^{r+s+1}(1+o(1)) O\l(n^{-\# E(\Delta(k,r,s))}\r)\\&=\sum_{\Delta(k, r, s) \in C_2}O\l(n^{r+s-\# E(\Delta^0(k,r,s))}\r)\\&=O(n^{-1}).
\eea

    {\bf Category 3.} Lastly, consider $\Delta(k,r,s)\in C_3$. If there is an edge $e_1=(u_1,v_1)\in \Delta^0(k,r,s)$ such that $s_{u_1v_1}=1$, since $\overline{\Delta}^0(k,r,s)$ is a tree, 
from Corollary \ref{Corollary:tree-independence-coef} and $\E[\Xi_{u_1v_1}^{(n)}]=0$, we have
\bea\nonumber
\E[\mf{\Xi}_n(\Delta(k,r,s))]=\E[\Xi_{u_1v_1}^{(n)}]\E\l[\prod_{(u,v)\in \big(E(\Delta^0(k,r,s))-\{e_1\}\big)}({\Xi}^{(n)}_{uv})^{s_{uv}}\r]=0.
\eea
Note that from the definition of $\Delta$-graphs,\ $(v_1,u_1)\notin E(\Delta^0(k,r,s))$. Thus it suffices to consider those $\Delta(k,r,s)$ such that 
\[
s_{uv}\geq 2, \text{ for any }(u,v)\in E(\Delta^0(k,r,s)), 
\]
and there is an edge $(u_1,v_1)$ $\in E(\Delta^0(k,r,s))$ such  that $s_{u_1v_1}\geq 3$. Then $\# E(\Delta^0(k,r,s))\leq k-1$. Since $\overline{\Delta}^0(k,r,s)$ is a tree,
$$
r+s+1=\# V(\Delta^0(k,r,s))=\# E(\Delta^0(k,r,s))+1\leq k.
$$
Since $\sqrt{n}\Xi_{uv}^{(n)}\Rightarrow\ml{N}(0,2/5)$ 
we have $\E[\mf{\Xi}_n(\Delta(k,r,s))]=O(n^{-k})$. Therefore
\bea\label{Emoments-convergence-MPlaw-H3}
H_3=\frac{1}{p}\sum_{\Delta(k, r, s) \in C_3}p^{r+s+1}(1+o(1)) O\l(n^{-k}\r)=\sum_{\Delta(k, r, s) \in C_3}O\l(n^{r+s-k}\r)=O(n^{-1}).
\eea

Combining \eqref{Emoments-convergence-MPlaw-H1}-\eqref{Emoments-convergence-MPlaw-H3} then finishes the proof.
\end{proof}

\begin{proposition}\label{moments-concentration-MPlaw}
    For any fixed $k\in \mb{N}$, we have, as $n\rw\infty$,
    \bea\nonumber
    \operatorname{Var}\l(\frac{1}{p}\operatorname{tr}(\mf{\Psi}_n^k)\r)=O(n^{-2}).
    \eea
\end{proposition}

\begin{proof}


We have
\begin{align*}
\operatorname{Var}\!\left(\frac{1}{p}\operatorname{tr}(\mf{\Psi}_n^k)\right)
&= \frac{1}{p^2} \sum_{\i_1,\i_2,\j_1,\j_2} 
\Big( \E\!\left[\mf{\Xi}_n(G(\i_1,\j_1)) \, \mf{\Xi}_n(G(\i_2,\j_2))\right] 
- \E\!\left[\mf{\Xi}_n(G(\i_1,\j_1))\right] 
  \E\!\left[\mf{\Xi}_n(G(\i_2,\j_2))\right] \Big).
\end{align*}
For brevity, write $G_1 = G(\i_1,\j_1)$ and $G_2 = G(\i_2,\j_2)$, and denote $G = G_1 \cup G_2$.  
Note that $G_1$, $G_2$, and $G$ are directed multigraphs.  
We highlight that $G^0$ and $\overline{G}^0$ denote the directed and undirected skeleton graphs of $G$, respectively.  
Analogously, the notations $G_1^0$, $\overline{G}_1^0$, $G_2^0$, and $\overline{G}_2^0$ are used for $G_1$ and $G_2$.

    Now we classify all possible pairs of directed multigraphs $(G_1,G_2)$ into the following categories.

\begin{itemize}
    \item {\bf Category 1 ($\overline{C}_1$)}:\ $V(G_1)\cap V(G_2)=\emptyset$.

    \item {\bf Category 2 ($\overline{C}_2$)}: \ $V(G_1)\cap V(G_2)\ne\emptyset$, and $\overline{G}^0$ is a tree.

    \item {\bf Category 3 ($\overline{C}_3$)}:\ $V(G_1)\cap V(G_2)\ne\emptyset$, $\overline{G}^0$ is not a tree, and both $\overline{G}_1^0,\ \overline{G}_2^0$ are trees.

    \item {\bf Category 4 ($\overline{C}_4$)}:\ $V(G_1)\cap V(G_2)\ne\emptyset$, $\overline{G}^0$ is not a tree, and at least one of $\overline{G}_1^0,\ \overline{G}_2^0$ is not a tree.
    \end{itemize}
    
    Then
    \bea\nonumber
    \operatorname{Var}\l(\frac{1}{p}\operatorname{tr}(\mf{\Psi}_n^k)\r)=\overline{H}_1+\overline{H}_2+\overline{H}_3+\overline{H}_4,
    \eea
    where, for each $j\in[4]$, 
    \bea\nonumber
    \overline{H}_j:=\frac{1}{p^2}\sum_{(G_1,G_2)\in \overline{C}_j}p^{\# V(G)}\big(1+o(1)\big)\big(\E[\mf{\Xi}_n(G_1)\mf{\Xi}_n(G_2)]-\E[\mf{\Xi}_n(G_1)]\E[\mf{\Xi}_n(G_2)]\big).
    \eea
    Here the summation is taken over all pairs of directed multigraphs $(G_1,G_2)$ in Category $j$ up to graph isomorphism.
    
    {\bf Category 1.} For any $(G_1,G_2)\in \overline{C}_1$,\ $\mf{\Xi}_n(G_1)$ and $\mf{\Xi}_n(G_2)$ are independent. Then 
    \bea\label{moments-concentration-MPlaw-H1}
    \overline{H}_1=0.
    \eea

    {\bf Category 2.}  For any $(G_1,G_2)\in \overline{C}_2$, if $E(G_1)\cap E(G_2)=\emptyset$, then by Corollary~\ref{Corollary:tree-independence-coef},  
$\mf{\Xi}_n(G_1)$ and $\mf{\Xi}_n(G_2)$ are independent.  
If $G$ contains an edge with multiplicity $1$, then
\[
\E\!\left[\mf{\Xi}_n(G_1)\mf{\Xi}_n(G_2)\right]
= \E\!\left[\mf{\Xi}_n(G_1)\right]\E\!\left[\mf{\Xi}_n(G_2)\right]
= 0.
\]
If $E(G_1)\cap E(G_2)\neq \emptyset$ and every edge in $G$ has multiplicity at least $2$, then $G$ must contain an edge with multiplicity at least $4$.  
In this case, $\# V(G)\leq 2k$.  
Moreover,
\[
\E[\mf{\Xi}_n(G_1)] = O(n^{-k}), 
\quad \E[\mf{\Xi}_n(G_2)] = O(n^{-k}), 
\quad \E[\mf{\Xi}_n(G_1)\mf{\Xi}_n(G_2)] = O(n^{-2k}).
\]
Therefore,
\begin{equation}\label{moments-concentration-MPlaw-H2}
\overline{H}_2
= \sum_{(G_1,G_2)\in \overline{C}_2} O\!\left(n^{\# V(G)-2k-2}\right)
= O(n^{-2}).
\end{equation}

    {\bf Category 3.}  For any $(G_1,G_2)\in \overline{C}_3$, we must have $\#(V(G_1)\cap V(G_2)) \geq 2$.  
Otherwise, since $\overline{G}_1^0$ and $\overline{G}_2^0$ are both trees, $\overline{G}^0$ would also be a tree.  
In that case,
\[
\# V(G)\leq \# V(G_1)+\# V(G_2)-2.
\]
From Corollary~\ref{Corollary:tree-independence-coef},
\[
\E[\mf{\Xi}_n(G_1)] = O\!\left(n^{-\# E(G_1^0)}\right), 
\quad \E[\mf{\Xi}_n(G_2)] = O\!\left(n^{-\# E(G_2^0)}\right), 
\quad \E[\mf{\Xi}_n(G_1)\mf{\Xi}_n(G_2)] = O\!\left(n^{-\# E(G^0)}\right).
\]
Since $\overline{G}^0$ is not a tree, it follows that $\# V(G)\leq \# E(G^0)$.  
Moreover, using $\# V(G_1)=\# E(G_1^0)+1$ and $\# V(G_2)=\# E(G_2^0)+1$, we obtain
\begin{equation}\label{moments-concentration-MPlaw-H3}
\overline{H}_3
= \sum_{(G_1,G_2)\in \overline{C}_3} 
O\!\left(n^{\# V(G)-\# E(G^0)-2}\right) 
+ O\!\left(n^{\# V(G_1)+\# V(G_2)-\# E(G_1^0)-\# E(G_2^0)-4}\right)
= O(n^{-2}).
\end{equation}

    {\bf Category 4.}  For any $(G_1,G_2)\in \overline{C}_4$, without loss of generality assume that $\overline{G}_1^0$ is not a tree.  
Then $\# V(G_1)\leq \# E(G_1^0)$.  
From Corollary~\ref{Corollary:tree-independence-coef}, together with the relations  
$\# V(G_2)\leq \# E(G_2^0)+1$, \ $\# V(G)\leq \# E(G^0)+1$, and \ $\# V(G)\leq \# V(G_1)+\# V(G_2)-1$, we obtain
\bea\label{moments-concentration-MPlaw-H4}
\overline{H}_4
= \sum_{(G_1,G_2)\in \overline{C}_4} 
   O\!\left(n^{\# V(G)-\# E(G^0)-2}\right)
   + O\!\left(n^{\# V(G_1)+\# V(G_2)-\# E(G_1^0)-\# E(G_2^0)-3}\right)
= O(n^{-2}).
\eea

    Combining \eqref{moments-concentration-MPlaw-H1}-\eqref{moments-concentration-MPlaw-H4} thus finishes the proof.
\end{proof}

Leveraging Propositions \ref{Emoments-convergence-MPlaw} and \ref{moments-concentration-MPlaw}, combined with Borel-Cantelli lemma, then proves Theorem \ref{thm:mp}.

\subsection{Proof of Theorem \ref{thm:clt}}\label{Section-CLT}
First, we derive the limiting variance function.  By the same proof procedure we obtain the covariance function in Theorem \ref{thm:clt}.
\begin{proposition}\label{CLT-variance}  For any fixed $k\in\N$,
\bea\nonumber
    \lim_{n\rw\infty} \operatorname{Var}\l(\tr(\mf{\Psi}_n^k)\r)=\l(\frac{2\gamma}{5}\r)^{2k}\l(2\binom{2k}{k+1}^2+2\sum_{t=2}^{k}t(2t-1)^2\l(\sum_{\ell=0}^{k-t}\frac{\binom{2k-2\ell-1}{k-\ell-t}\binom{2\ell+1}{\ell+1}}{2k-2\ell-1}\r)^2\r).
\eea
\end{proposition}

\begin{proof}
    We have
    \bea\nonumber
    \operatorname{Var}\l(\operatorname{tr}(\mf{\Psi}_n^k)\r)
=\sum_{\i_1,\i_2,\j_1,\j_2}\big(\E[\mf{\Xi}_n(G(\i_1,\j_1))\mf{\Xi}_n(G(\i_2,\j_2))]-\E[\mf{\Xi}_n(G(\i_1,\j_1))]\E[\mf{\Xi}_n(G(\i_2,\j_2))]\big). 
    \eea
    For brevity, write $G_1=G(\i_1,\j_1),\ G_2=G(\i_2,\j_2)$, and denote $G=G_1\cup G_2$. We classify all possible graph pairs $(G_1,G_2)$ into the following three categories.

\begin{itemize}
    \item {\bf Category 1 ($\overline{Q}_1$)}:\ $\# V(G)<\# E(G^0)$.
    
    \item {\bf Category 2 ($\overline{Q}_2$)}:\ $\# V(G)=\# E(G^0)$.
    
    \item {\bf Category 3 ($\overline{Q}_3$)}:\ $\# V(G)=\# E(G^0)+1$.
\end{itemize}

    Then
    \bea\nonumber
    \operatorname{Var}\l(\operatorname{tr}(\mf{\Psi}_n^k)\r)=\overline{H}_1+\overline{H}_2+\overline{H}_3,
    \eea
    where for each $j\in[3]$,
    \bea\nonumber
    \overline{H}_j:=\sum_{(G_1,G_2)\in \overline{Q}_j}p^{\# V(G)}\big(1+o(1)\big)\big(\E[\mf{\Xi}_n(G_1)\mf{\Xi}_n(G_2)]-\E[\mf{\Xi}_n(G_1)]\E[\mf{\Xi}_n(G_2)]\big).
    \eea
    Here the summation is taken over all different graph pairs $(G_1,G_2)$ in Category $j$ up to graph isomorphism. Note that the summand is 0 when $V(G_1)\cap V(G_2)=\emptyset$.

    {\bf Category 1.} For $\overline{H}_1$, from Proposition \ref{Graph-Estimate-EDelta0}, we have
    \bea\nonumber
    \overline{H}_1=\sum_{(G_1,G_2)\in \overline{Q}_1}\bigg(O(n^{\# V(G)-\# E(G^0)})+O(n^{\# V(G)-\# E(G_1^0)-\# E(G_2^0)})\bigg)=O(n^{-1}).
    \eea

    {\bf Category 2.} For $\overline{H}_2$, from Proposition \ref{Graph-Estimate-EDelta0}, we have,
    \bea\nonumber
     \overline{H}_2&=\sum_{(G_1,G_2)\in \overline{Q}_2}\bigg(O(n^{\# E(G^0)-(\# E(G)+Q_0(G))/2})+O(n^{\# E(G^0)-(\# E(G_1)+Q_0(G_1)+\# E(G_2)+Q_0(G_2))/2})\bigg)\\
     &=\sum_{(G_1,G_2)\in \overline{Q}_2}O(n^{\# E(G^0)-(\# E(G)+Q_0(G))/2}).
    \eea
Therefore, for those $(G_1,G_2)$ such that $G$ contains an edge with multiplicity at least $3$, the corresponding summand contributes $O(n^{-1})$.  
Hence, we may assume that all edges of $G$ have multiplicity at most $2$.

For any $(G_1,G_2)\in \overline{Q}_2$, the skeleton graph $G^0$ contains a unique cycle 
$C^0 = (i_1, i_2, \dots, i_s, i_1)$.  
We can then decompose $G^0$ as 
\[
G^0 = C^0 \;\bigcup\; T_1^0 \;\bigcup\; T_2^0 \;\bigcup\; \cdots \;\bigcup\; T_s^0,
\]
where $(T_k^0)_{k=1}^s$ is a sequence of disjoint trees satisfying 
$V(T_k^0) \cap V(C^0) = \{i_k\}$ and $E(T_k^0) \cap E(C^0) = \emptyset$ for each $k$.  

Correspondingly, we can write
\[
G = C \;\bigcup\; T_1 \;\bigcup\; T_2 \;\bigcup\; \cdots \;\bigcup\; T_s,
\]
where $T_k$ (or $C$) denotes the maximal subgraph of $G$ whose skeleton graph is $T_k^0$ (or $C^0$).  

For any subgraph $G'\subset G^0$, denote
\[
\ml{R}(G') = \l\{ R_v R_u^{-1} : (u,v) \in E(G') \r\}.
\]

    \begin{lemma}\label{cycle-tree independence}
        $\{\ml{R}(C^0),\ml{R}(T_1^0),\ml{R}(T_2^0),\cdots,\ml{R}(T_s^0)\}$ are independent.
    \end{lemma}
   
As a corollary, the random variables 
\[
\{\mf{\Xi}_n(C), \mf{\Xi}_n(T_1), \mf{\Xi}_n(T_2), \dots, \mf{\Xi}_n(T_s)\}
\] 
are independent.  Consequently, if $G$ contains a single edge in any of its tree components $(T_k)_{k=1}^s$, the corresponding summand in $\overline{H}_2$ is zero.

If $\overline{G}_1^0$ and $\overline{G}_2^0$ are both trees, note that in a tree $\Delta$-graph every edge has multiplicity at least $2$.  
It follows that $E(G_1^0)\cap E(G_2^0) = \emptyset$ and that all edges of $G$ have multiplicity $2$.  
Then, by Proposition~\ref{approx1}, such $(G_1,G_2)$ contributes $o(1)$.  If $G_1^0$ is a tree but $G_2^0$ is not, then $G_1^0$ must contain the entire cycle $C^0$, and $G_2$ is a sub-tree of $T_k$ for some $k \in [s]$.  
By Lemma~\ref{cycle-tree independence}, $\mf{\Xi}_n(G_1)$ is independent of $\mf{\Xi}_n(G_2)$, so such $(G_1,G_2)$ contributes $0$.

It thus suffices to consider the case where both $G_1^0$ and $G_2^0$ are not trees.  
In this case, both $G_1^0$ and $G_2^0$ contain the entire cycle $C^0$, and the edges of $C^0$ must have multiplicity $1$ in $G_1$ and $G_2$.  
The edges in the trees $(T_k)_{k=1}^s$ have multiplicity $2$.  
The length of $C^0$ takes values in $\{2t : 2 \leq t \leq k\}$.  

Suppose the length of $C^0$ is $2t$. We construct $G_1$ as follows.  
Assume that the path of $G_1$ originates in $T_1$, and let $\# E(T_1) = 2\ell$ with $0 \leq \ell \leq k-t$.
Then there are $C_\ell$ ways to choose $T_1$, where $C_\ell$ denotes the $\ell$-th Catalan number.  
For each such path, there are $(2\ell+1)$ choices for the moment at which the path enters $C^0$.  
The number of ways to choose the remaining trees is
\[
\sum_{\substack{i_2 + \cdots + i_{2t} = k-t-\ell \\ i_2, \dots, i_{2t} \geq 0}} \prod_{j=2}^{2t} C_{i_j}.
\]

The following lemma is a direct corollary of \citet[Lemma 27]{Bowman14}.

    \begin{lemma}\label{k-fold involution of Catalan number}
        For any $n, m\in\N$, 
          
        \bea\nonumber
         \sum_{\substack{i_1+\cdots+i_{m}=n\\i_1,\cdots,i_m\geq 0}}\prod_{j=1}^{m}C_{i_j}=\frac{m}{2n+m}\binom{2n+m}{n}.
        \eea
    \end{lemma}
    From Lemma \ref{k-fold involution of Catalan number}, the number of choices of $G_1$ is
    \bea\nonumber
    \sum_{\ell=0}^{k-t}C_\ell(2\ell+1)\sum_{\substack{i_2+\cdots+i_{2t}=k-t-\ell\\i_2,\cdots,i_{2t}\geq 0}}\prod_{j=2}^{2t}C_{i_j}=\sum_{\ell=0}^{k-t}\frac{2t-1}{2k-2\ell-1}\binom{2k-2\ell-1}{k-t-\ell}\binom{2\ell+1}{\ell+1}.
    \eea
    The number of choices of $G_2$ is the same. After choosing $G_1,G_2$, there are $t$ ways to merge the cycle and 2 ways to choose the cycle direction. Each $(G_1,G_2)$ contributes expectation $(\E[(\mf{\Xi}_{12}^{(n)})^2])^{2k}$. Then we have
    \bea\nonumber
    \overline{H}_2=p^{2k}(1+o(1))\l(\frac{2}{5n}\r)^{2k}\sum_{t=2}^{k}2t\l[\sum_{\ell=0}^{k-t}\frac{2t-1}{2k-2\ell-1}\binom{2k-2\ell-1}{k-t-\ell}\binom{2\ell+1}{\ell+1}\r]^2.
    \eea
    
    {\bf Category 3.}  Lastly, we consider $\overline{H}_3$.  
For any $(G_1,G_2)\in \overline{Q}_3$, the graphs $G$, $G_1$, and $G_2$ are all trees, so all edges in $G$, $G_1$, and $G_2$ have multiplicity at least $2$.  

If $E(G_1^0) \cap E(G_2^0) = \emptyset$, then by Corollary~\ref{Corollary:tree-independence-coef}, $\mf{\Xi}_n(G_1)$ and $\mf{\Xi}_n(G_2)$ are independent.  
If $\#(E(G_1^0) \cap E(G_2^0)) \geq 2$, then there are at least two edges of multiplicity $4$ in $G$, so that $\# V(G) \leq 2k-1$, and the contribution to $\overline{H}_3$ is $O(n^{-1})$.  

Therefore, it suffices to consider the case $\#(E(G_1^0) \cap E(G_2^0)) = 1$, where all edges in $G_1$ and $G_2$ have multiplicity exactly $2$, and there is exactly one edge of multiplicity $4$ in $G$.  
For such $(G_1,G_2)$, there are $C_k$ ways to choose $G_1$ or $G_2$, and {$k^2$} ways to merge $G_1$ and $G_2$.  
Each $(G_1,G_2)$ contributes an expectation of
\[
\big(\E[(\mf{\Xi}_{12}^{(n)})^4] - (\E[(\mf{\Xi}_{12}^{(n)})^2])^2\big) \, (\E[(\mf{\Xi}_{12}^{(n)})^2])^{2k-2}.
\]

Hence,
{
\bea\nonumber
\overline{H}_3 &= p^{2k} (1+o(1)) \cdot 2 \left(\frac{2}{5n}\right)^{2k} \left(\frac{\binom{2k}{k}}{k+1}\right)^2 k^2.\\
&= p^{2k} (1+o(1)) \cdot 2 \left(\frac{2}{5n}\right)^{2k} \binom{2k}{k+1}  ^2.
\eea
}

    Then we have proved Proposition \ref{CLT-variance}. 
\end{proof}

Now we prove the Gaussianity.
\begin{proposition}
    Assume Assumption~\ref{assump:dgp} and \eqref{eq:asymptotics}. Then, as $n\to\infty$,
\[
\Big\{\tr(\mPsi_n^k)-\E\tr(\mPsi_n^k)\Big\}_{k=1}^\infty 
\;\Rightarrow\; \{G_k\}_{k=1}^{\infty},
\]
where $\{G_k\}_{k=1}^{\infty}$ is a Gaussian process.
\end{proposition}
\begin{proof}
For $k_1, \dots, k_L \in \N$, we have
\bea\nonumber
\E\Bigg[\prod_{\ell=1}^L \Big(\tr(\mf{\Psi}_n^{k_\ell}) - \E[\tr(\mf{\Psi}_n^{k_\ell})]\Big)\Bigg] 
= \sum_{\i_1, \j_1, \dots, \i_L, \j_L} 
\E\Bigg[\prod_{\ell=1}^L \Big(\mf{\Xi}_n(G_\ell) - \E[\mf{\Xi}_n(G_\ell)]\Big)\Bigg],
\eea
where $G_\ell = G(\i_\ell, \j_\ell)$.  
Denote $G = \bigcup_{\ell=1}^L G_\ell$.  Note that the summand is non-zero only if each $G_\ell$ shares at least one vertex with some $G_{\ell'}$ with $\ell \neq \ell'$.  
Consequently, the number of connected components in $G$, denoted by $c(G)$, satisfies $c(G) \leq \lceil L/2 \rceil$.  

We classify all possible graph sequences $(G_1, \dots, G_L)$ into the following $\lceil L/2 \rceil$ categories:
\begin{itemize}
\item {\bf Category $t$ ($\overline{Q}_t$)}: $c(G)=t$,\quad\quad $t=1,2,\cdots,\lceil L/2\rceil$. 
\end{itemize}

Then 
\bea\nonumber
\E\l[\prod_{\ell=1}^L\l(\tr(\mf{\Psi}_n^{k_\ell})-\E[\tr(\mf{\Psi}_n^{k_\ell})]\r)\r]=\sum_{t=1}^{\lceil L/2 \rceil} \overline{H}_t,
\eea
where, for each $t\in[\lceil L/2 \rceil]$,
\bea\label{Proof_Gaussianity_Ht}
\overline{H}_t:=\sum_{(G_1,\cdots,G_L)\in \overline{Q}_t}p^{\# V(G)}\big(1+o(1)\big)\E\l[\prod_{\ell=1}^L \big(\mf{\Xi}_n(G_\ell)-\E[\mf{\Xi}_n(G_\ell)]\big)\r].
\eea

The summation is taken over all different graph sequences up to graph isomorphism. For each $(G_1,\cdots,G_L)\in \overline{Q}_t$, denote connected components of $G$ as $\{D_j\}_{j=1}^t$. Then $D_j=\cup_{i\in \uptau^{-1}(j)} G_i$, for some $\uptau$ that is a partition on $[L]$ with $t$ blocks. Then by independence, the summand of \eqref{Proof_Gaussianity_Ht} equals to
\bea\nonumber
&\big(1+o(1)\big)\prod_{j=1}^t \l(p^{\#  V(D_j)}\E\l[\prod_{i\in \uptau^{-1}(j)}\big(\mf{\Xi}_n(G_i)-\E[\mf{\Xi}_n(G_i)]\big)\r]\r)
\\ =:&\big(1+o(1)\big)\prod_{j=1}^t \omega^\uptau_j.
\eea

If $\#  \uptau^{-1}(j)=2$, from the proof covariance function, $\omega_j^{\uptau}=O(1)$. If $\# \uptau^{-1}(j)\geq 3$, we prove $\omega_j^\uptau=o(1)$ by considering the following three possible cases.

\textbf{Case 1:} $\# V(D_j)<\# E(D_j^0)$. From Proposition \ref{Graph-Estimate-EDelta0},
\bea\nonumber
\omega_j^\uptau=O(n^{\# V(D_j)-\# E(D_j^0)})+O(n^{\# V(D_j)-\sum_{i\in \uptau^{-1}(j)}\# E(G_i^0)})=O(n^{-1}).
\eea

\textbf{Case 2:} $\# V(D_j)=\# E(D_j^0)$. From Proposition \ref{Graph-Estimate-EDelta0},
\bea\nonumber
\omega_j^\uptau&=O(n^{\# E(D_j^0)-(\# E(D_j)+Q_0(D_j)/2})+O(n^{\# E(D_j^0)-\sum_{i\in \uptau^{-1}(j)}(\# E(G_i)+Q_0(G_i))/2})\\
&=O(n^{\# E(D_j^0)-(\# E(D_j)+Q_0(D_j)/2}).
\eea
Hence, $\omega_j^\uptau = O(n^{-1})$ if $D_j$ contains an edge of multiplicity at least $3$.  
Now, assume that every edge of $D_j$ has multiplicity at most $2$.  
Since $\# V(D_j) = \# E(D_j^0)$, we can still represent $D_j^0$ as a cycle $C$ with trees attached along the cycle.  
By Lemma~\ref{cycle-tree independence}, we may further assume that the tree components contain no single edges.  

If $G_i$ is a tree for some $i \in \uptau^{-1}(j)$, then $\omega_j^\uptau$ is non-zero only if $E(G_i) \cap E(C) \neq \emptyset$.  
In this case, all such $G_i$ are trees with no coincident edges, and by Proposition~\ref{approx1}, we have $\omega_j^\uptau = o(1)$.  

If every $G_i$ is a non-tree for $i \in \uptau^{-1}(j)$, then each $G_i$ contains the cycle $C$.  
Since $\# \uptau^{-1}(j) \geq 3$, the cycle $C$ would appear with multiplicity at least $3$ in $D_j$, which contradicts our assumption.  

Therefore, in Case~2, we conclude that $\omega_j^\uptau = o(1)$.

\textbf{Case 3:} $\# V(D_j)=\# E(D_j^0)+1$. Then $D_j$ and all $G_i$ for $i \in \uptau^{-1}(j)$ are trees, with all edges having multiplicity at least $2$.  
By Corollary~\ref{Corollary:tree-independence-coef}, each $G_i^0$ must share edges with $G_{i'}^0$ for $i \neq i'$.  
Since $\# \uptau^{-1}(j) \geq 3$, we have 
\[
\# E(D_j^0) \leq 2 \Big(\sum_{i \in \uptau^{-1}(j)} k_i\Big) - 2,
\] 
and hence
\[
\omega_j^\uptau = O\Big(n^{\# V(D_j) - 2 \sum_{i \in \uptau^{-1}(j)} k_i}\Big) = O(n^{-1}).
\]

Thus, we have shown that $\omega_j^\uptau = o(1)$ whenever $\# \uptau^{-1}(j) \geq 3$.  
It follows that
\[
\prod_{j=1}^t \omega_j^\uptau =
\begin{cases}
O(1), & \text{if } \# \uptau^{-1}(1) = \cdots = \# \uptau^{-1}(t) = 2,\ t = L/2,\\
o(1), & \text{otherwise}.
\end{cases}
\]

Let $\ml{P}_2(L)$ denote the set of pair partitions of $\{1, 2, \dots, L\}$. Then
\bea\nonumber
&\E\Bigg[\prod_{\ell=1}^L \Big(\tr(\mf{\Psi}_n^{k_\ell}) - \E[\tr(\mf{\Psi}_n^{k_\ell})]\Big)\Bigg] \\
= &\sum_{\uppi \in \ml{P}_2(L)} 
\prod_{\{i,j\} \in \uppi} 
\E\Big[\big(\mf{\Xi}_n(G_i)-\E[\mf{\Xi}_n(G_i)]\big) \big(\mf{\Xi}_n(G_j)-\E[\mf{\Xi}_n(G_j)]\big)\Big] + o(1) \\
= &\sum_{\uppi \in \ml{P}_2(L)} \prod_{\{i,j\} \in \uppi} M(k_i, k_j) + o(1),
\eea
where $M(\cdot,\cdot)$ is the covariance function defined in Theorem~\ref{thm:clt}.  

By Wick's formula \citep{Wick-50}, the vector $\{\tr(\mf{\Psi}_n^{k_\ell}) - \E[\tr(\mf{\Psi}_n^{k_\ell})]\}_{\ell=1}^L$ converges in distribution to a Gaussian vector with covariance $M(\cdot, \cdot)$. We then finish the proof.
\end{proof}

\subsection{Proof of Theorem \ref{thm:clt:Phi}}
{
\begin{proof}
We only prove the covariance function in detail. The proof of Gaussianity is similar to that of Theorem~\ref{thm:clt}, with the directed $\Delta$-graphs there replaced by the undirected $\overline{\Gamma}$-graphs used in the proof of Theorem~\ref{thm:sc}. For an undirected multigraph $\overline{\Delta}$, write in this proof
\[
\mf{\Phi}_n(\overline{\Delta})
:=\prod_{\{u,v\}\in E(\overline{\Delta}^0)}
\big(\Phi_{uv}^{(n)}\big)^{s_{uv}},
\]
where $\overline{\Delta}^0$ is the undirected skeleton graph of $\overline{\Delta}$.

Let $\mathbf{i}_a=(i_{a,1},\ldots,i_{a,k_a})$, $a=1,2$, with $i_{a,u}\neq i_{a,u+1}$ and $i_{a,k_a+1}=i_{a,1}$. Denote $\overline{G}_a=\overline{T}(\mathbf{i}_a)$ and $\overline{G}=\overline{G}_1\cup \overline{G}_2$. Then
\bea\nonumber
&\Cov\l(\tr((\mf{\Phi}_n-\mf{I}_p)^{k_1}),\tr((\mf{\Phi}_n-\mf{I}_p)^{k_2})\r)\\
&=\sum_{\mathbf{i}_1,\mathbf{i}_2}
\Big(\E\big[\mf{\Phi}_n(\overline{G}_1)\mf{\Phi}_n(\overline{G}_2)\big]
-\E\big[\mf{\Phi}_n(\overline{G}_1)\big]\E\big[\mf{\Phi}_n(\overline{G}_2)\big]\Big).
\eea
As in the proof of Theorem~\ref{thm:clt}, we classify all possible graph pairs into the following three categories:
\begin{itemize}
\item {\bf Category 1 ($\overline{Q}_1$)}:\ $\# V(\overline{G})<\# E(\overline{G}^0)$;
\item {\bf Category 2 ($\overline{Q}_2$)}:\ $\# V(\overline{G})=\# E(\overline{G}^0)$;
\item {\bf Category 3 ($\overline{Q}_3$)}:\ $\# V(\overline{G})=\# E(\overline{G}^0)+1$.
\end{itemize}
Thus the covariance can be written as $\overline{H}_1+\overline{H}_2+\overline{H}_3$, where, for $j\in[3]$,
\bea\nonumber
\overline{H}_j:=\sum_{(\overline{G}_1,\overline{G}_2)\in \overline{Q}_j}
p^{\# V(\overline{G})}(1+o(1))
\Big(\E\big[\mf{\Phi}_n(\overline{G}_1)\mf{\Phi}_n(\overline{G}_2)\big]
-\E\big[\mf{\Phi}_n(\overline{G}_1)\big]\E\big[\mf{\Phi}_n(\overline{G}_2)\big]\Big).
\eea
Here the summation is taken over all different graph pairs up to graph isomorphism. The summand is zero when $V(\overline{G}_1)\cap V(\overline{G}_2)=\emptyset$.

{\bf Category 1.} From Lemma~\ref{EPhi(Delta)-semicirclelaw}, applied also to $\overline{G}$, we have
\bea\nonumber
\overline{H}_1
=\sum_{(\overline{G}_1,\overline{G}_2)\in \overline{Q}_1}
\Big(O(n^{\# V(\overline{G})-\# E(\overline{G}^0)})
+O(n^{\# V(\overline{G})-\# E(\overline{G}_1^0)-\# E(\overline{G}_2^0)})\Big)
=O(n^{-1}).
\eea

{\bf Category 2.} We first remove negligible graph pairs. If $\overline{G}$ contains an edge with multiplicity at least $3$, then the expansion $\Phi_{uv}^{(n)}=(\Xi_{uv}^{(n)}+\Xi_{vu}^{(n)})/2$, together with Proposition~\ref{Graph-Estimate-EDelta0}, gives a contribution of order $O(n^{-1})$. Hence it suffices to consider graph pairs for which every edge in $\overline{G}$ has multiplicity exactly $2$.

For such a graph pair, $\overline{G}^0$ contains a unique cycle
\[
C^0=(v_1,v_2,\ldots,v_s,v_1).
\]
We decompose
\[
\overline{G}^0=C^0\bigcup T_1^0\bigcup T_2^0\bigcup\cdots\bigcup T_s^0,
\]
where $(T_j^0)_{j=1}^s$ are disjoint trees with
$V(T_j^0)\cap V(C^0)=\{v_j\}$ and $E(T_j^0)\cap E(C^0)=\emptyset$. By the same argument as Lemma~\ref{cycle-tree independence}, applied after expanding each $\Phi$-entry into the two possible directed $\Xi$-entries, the variables associated with the cycle and with the attached trees are independent. Therefore any graph pair in which one of the tree components carries a single edge in either $\overline{G}_1$ or $\overline{G}_2$ gives zero contribution.

It remains to consider the case where both $\overline{G}_1^0$ and $\overline{G}_2^0$ are not trees. Then both contain the cycle $C^0$, every edge of $C^0$ has multiplicity $1$ in each of $\overline{G}_1$ and $\overline{G}_2$, and all edges in the attached trees have multiplicity $2$. In particular, the parity of the cycle length agrees with the parity of both $k_1$ and $k_2$. Thus, if one of $k_1,k_2$ is odd and the other is even, Category 2 contributes $o(1)$.

Suppose first that $k_1,k_2$ are both odd. Then the length of $C^0$ is $2t+1$, where $1\leq t\leq \lfloor (k_1\wedge k_2)/2\rfloor$. Fix such a $t$. To construct $\overline{G}_1$, assume that its path starts from $T_1$, and let $\# E(T_1)=2\ell$, where $0\leq \ell\leq \lfloor k_1/2\rfloor-t$. There are $C_\ell$ ways to choose $T_1$, where $C_\ell$ is the $\ell$-th Catalan number, and $(2\ell+1)$ choices for the moment at which the path enters the cycle. The number of choices for the remaining attached trees is
\[
\sum_{\substack{i_2+\cdots+i_{2t+1}=\lfloor k_1/2\rfloor-t-\ell\\
i_2,\ldots,i_{2t+1}\geq 0}}
\prod_{j=2}^{2t+1}C_{i_j}.
\]
By Lemma~\ref{k-fold involution of Catalan number}, the number of choices of $\overline{G}_1$ is
\[
\sum_{\ell=0}^{\lfloor k_1/2\rfloor-t}
\frac{2t}{k_1-2\ell-1}
\binom{k_1-2\ell-1}{\lfloor k_1/2\rfloor-t-\ell}
\binom{2\ell+1}{\ell+1}.
\]
The corresponding number for $\overline{G}_2$ is
\[
\sum_{\ell=0}^{\lfloor k_2/2\rfloor-t}
\frac{2t}{k_2-2\ell-1}
\binom{k_2-2\ell-1}{\lfloor k_2/2\rfloor-t-\ell}
\binom{2\ell+1}{\ell+1}.
\]
There are $(2t+1)$ ways to merge the vertices on the cycle and $2$ choices for the direction of the cycle. Since
\[
\E\big[(\Phi_{12}^{(n)})^2\big]=\frac{1}{5n}+O(n^{-2}),
\]
the contribution from Category 2 in the odd case is
\bea\nonumber
\overline{H}_2
&=p^{(k_1+k_2)/2}(1+o(1))\left(\frac{1}{5n}\right)^{(k_1+k_2)/2}\\
&\quad\times
\sum_{t=1}^{\lfloor(k_1\wedge k_2)/2\rfloor}2(2t+1)
\left(\sum_{\ell=0}^{\lfloor k_1/2\rfloor-t}
\frac{2t}{k_1-2\ell-1}
\binom{k_1-2\ell-1}{\lfloor k_1/2\rfloor-t-\ell}
\binom{2\ell+1}{\ell+1}\right)\\
&\quad\times
\left(\sum_{\ell=0}^{\lfloor k_2/2\rfloor-t}
\frac{2t}{k_2-2\ell-1}
\binom{k_2-2\ell-1}{\lfloor k_2/2\rfloor-t-\ell}
\binom{2\ell+1}{\ell+1}\right).
\eea

Suppose next that $k_1,k_2$ are both even. Then the length of $C^0$ is $2t$, where $2\leq t\leq (k_1\wedge k_2)/2$. Repeating the same enumeration gives the number of choices of $\overline{G}_1$ as
\[
\sum_{\ell=0}^{k_1/2-t}
\frac{2t-1}{k_1-2\ell-1}
\binom{k_1-2\ell-1}{k_1/2-t-\ell}
\binom{2\ell+1}{\ell+1},
\]
and the corresponding number for $\overline{G}_2$ as
\[
\sum_{\ell=0}^{k_2/2-t}
\frac{2t-1}{k_2-2\ell-1}
\binom{k_2-2\ell-1}{k_2/2-t-\ell}
\binom{2\ell+1}{\ell+1}.
\]
There are $2t$ ways to merge the cycle vertices and $2$ choices for the direction of the cycle. Therefore, in the even case,
\bea\nonumber
\overline{H}_2
&=p^{(k_1+k_2)/2}(1+o(1))\left(\frac{1}{5n}\right)^{(k_1+k_2)/2}\\
&\quad\times
\sum_{t=2}^{(k_1\wedge k_2)/2}2(2t)
\left(\sum_{\ell=0}^{k_1/2-t}
\frac{2t-1}{k_1-2\ell-1}
\binom{k_1-2\ell-1}{k_1/2-t-\ell}
\binom{2\ell+1}{\ell+1}\right)\\
&\quad\times
\left(\sum_{\ell=0}^{k_2/2-t}
\frac{2t-1}{k_2-2\ell-1}
\binom{k_2-2\ell-1}{k_2/2-t-\ell}
\binom{2\ell+1}{\ell+1}\right).
\eea

{\bf Category 3.} In this case $\overline{G}^0$ is a tree. If $E(\overline{G}_1^0)\cap E(\overline{G}_2^0)=\emptyset$, the two products are independent by Lemma~\ref{tree-independence-semicircle}. If $\#(E(\overline{G}_1^0)\cap E(\overline{G}_2^0))\geq 2$, then $\overline{G}$ contains at least two edges of multiplicity $4$, and the contribution is $O(n^{-1})$. Hence it remains to consider the case
\[
\#(E(\overline{G}_1^0)\cap E(\overline{G}_2^0))=1.
\]
Then both $\overline{G}_1^0$ and $\overline{G}_2^0$ are trees, all edges have multiplicity $2$ except the common edge, which has multiplicity $4$, and this is possible only when $k_1,k_2$ are both even. For such graph pairs, there are $C_{k_1/2}$ and $C_{k_2/2}$ ways to choose $\overline{G}_1$ and $\overline{G}_2$, respectively. To merge the two trees, there are $(k_1/2)(k_2/2)$ choices for the common edge and $2$ choices for its direction. Moreover, Proposition~\ref{approx1} implies
\[
\E\big[(\Phi_{12}^{(n)})^4\big]
-\Big(\E\big[(\Phi_{12}^{(n)})^2\big]\Big)^2
=2\left(\frac{1}{5n}\right)^2+O(n^{-3}).
\]
Consequently,
\bea\nonumber
\overline{H}_3
&=p^{(k_1+k_2)/2}(1+o(1))\cdot
2\left(\frac{1}{5n}\right)^{(k_1+k_2)/2}
C_{k_1/2}C_{k_2/2}\left(\frac{k_1}{2}\right)\left(\frac{k_2}{2}\right)\cdot 2\\
&=p^{(k_1+k_2)/2}(1+o(1))\cdot
4\left(\frac{1}{5n}\right)^{(k_1+k_2)/2}
\binom{k_1}{k_1/2+1}\binom{k_2}{k_2/2+1}.
\eea

Combining the three categories, if $k_1$ and $k_2$ have different parity, the covariance tends to $0$. If $k_1,k_2$ are both odd, only Category 2 contributes. If $k_1,k_2$ are both even, Categories 2 and 3 contribute. Letting $n\to\infty$ and using $p/n\to\gamma$ gives the covariance function in Theorem~\ref{thm:clt:Phi}.

It remains to record why the limiting process is Gaussian. For any fixed $L$ and $k_1,\ldots,k_L$, expand
\[
\E\prod_{\ell=1}^L
\Big(\tr((\mf{\Phi}_n-\mf{I}_p)^{k_\ell})
-\E\tr((\mf{\Phi}_n-\mf{I}_p)^{k_\ell})\Big)
\]
as a sum over $L$ undirected $\overline{\Gamma}$-graphs. The same connected-component decomposition used in the proof of Theorem~\ref{thm:clt} applies. A component containing exactly two graphs contributes the covariance computed above, while a component containing at least three graphs is $o(1)$ by the three-category estimates just proved, together with Lemma~\ref{tree-independence-semicircle} and Proposition~\ref{approx1}. Hence only pair partitions survive asymptotically. Wick's formula then gives the desired Gaussian finite-dimensional limits.
\end{proof}
}

\subsection{Proof of Theorem \ref{thm:power-Qxi2} and Corollary \ref{cor:equi-gaussian-power-Qxi2}}\label{section-power-Qxi2}
We need the following lemmas. Their proofs are given in Section~\ref{Section-aux-power-Qxi2}.
\begin{lemma}[One-row stability of Chatterjee's rank correlation]\label{lem:chatterjee-one-row-stability}
Let $(X_1,Y_1),\ldots,(X_n,Y_n)$ and $(X'_1,Y'_1),\ldots,(X'_n,Y'_n)$ be two bivariate samples without ties. If the two samples differ in only one observation, then
\[
|\xi_n(\mX,\mY)-\xi_n(\mX',\mY')|\leq \frac{18}{n},
\]
where $\mX=(X_1,\ldots,X_n)$, $\mY=(Y_1,\ldots,Y_n)$, and similarly for $\mX'$ and $\mY'$.
\end{lemma}

\begin{lemma}[Gaussian local behavior of Chatterjee's coefficient]\label{lem:gaussian-local-xi}
Let $(X_1,Y_1),\ldots,(X_n,Y_n)$ be independent copies of a bivariate
standard Gaussian vector with Pearson correlation $\rho_n$, and
define
\[
\mu_n(\rho_n)
:=
\E\{\xi_n((X_1,\ldots,X_n),(Y_1,\ldots,Y_n))\}.
\]
Suppose that $\lim_{n\to \infty} \rho_n=0$ and $\liminf_{n\to\infty} n\rho_n^{4}>0$, then
\[
\mu_n(\rho_n)=\frac{\sqrt{3}}{\pi}\rho_n^2(1+o(1)),\quad \text{ as } n\to\infty.
\]
\end{lemma}

Now we prove Theorem \ref{thm:power-Qxi2}.
\begin{proof}[Proof of Theorem~\ref{thm:power-Qxi2}]
Write 
\[
T_n:=\operatorname{tr}(\mPsi_n)=\sum_{i\ne j}\bigl(\Xi_{ij}^{(n)}\bigr)^2.
\]
Under the null hypothesis $H_0$, Theorem~\ref{thm:clt} gives, with $\E_0$ denoting null expectation,
\[
\E_0T_n
=
\frac{p(p-1)(n-2)(4n-7)}{10(n-1)^2(n+1)}.
\]
Consequently the rejection event may be written as $\{T_n>a_n\}$, where
\[
a_n:=\E_0T_n+z_\alpha\sqrt{8p^2/(25n^2)}.
\]
Then there exists a constant
$A_{\alpha}>0$ such that
\[
a_n\leq A_{\alpha}\frac{p^2}{n}
\]
for all sufficiently large $n$.

Let
\[
\boldsymbol \xi_n:=\bigl(\Xi_{ij}^{(n)}\bigr)_{i\ne j}\in\R^{p(p-1)},
\qquad
\boldsymbol\mu_n:=\E_D\boldsymbol\xi_n.
\]
Then $T_n=\|\boldsymbol\xi_n\|_2^2$. Moreover, since the diagonal of $\mXi_n$ is deterministic and equal to one,
\[
\|\overline{\mXi}_n-\mf{I}_p\|_{\mathrm F}^2=\|\boldsymbol\mu_n\|_2^2.
\]

Choose $B>0$ so large that
\[
\frac{B^2}{4}>A_{\alpha},
\qquad
\frac{4\cdot 648}{B^2}<1-\beta,
\]
and choose $C>B$. If
$\|\overline{\mXi}_n-\mf{I}_p\|_{\mathrm F}\geq Cpn^{-1/2}$, then, for all sufficiently large $n$,
\[
\|\boldsymbol\mu_n\|_2^2
=
\|\overline{\mXi}_n-\mf{I}_p\|_{\mathrm F}^2
\geq C^2\frac{p^2}{n}
> B^2\frac{p^2}{n}.
\]
Thus $\|\boldsymbol\mu_n\|_2\geq Bpn^{-1/2}$.

Let
\[
u_n:=\frac{\boldsymbol\mu_n}{\|\boldsymbol\mu_n\|_2},
\qquad
W_n:=\langle u_n,\boldsymbol\xi_n\rangle .
\]
Then
\[
\E_DW_n=\|\boldsymbol\mu_n\|_2\geq Bpn^{-1/2}.
\]

We next control the fluctuations of $W_n$. Let $W_n^{(r)}$ denote the statistic obtained after replacing the $r$th row of the data matrix by an independent copy, while keeping the same vector $u_n$. Then Lemma~\ref{lem:chatterjee-one-row-stability} and the Cauchy--Schwarz inequality yield, uniformly in $r$,
\[
|W_n-W_n^{(r)}|
\leq
\frac{18}{n}\sum_{i\ne j}|u_{n,ij}|
\leq
\frac{18p}{n}
\]
for all sufficiently large $n$. Hence the Efron--Stein inequality gives
\[
\operatorname{Var}_D(W_n)
\leq \frac12\sum_{r=1}^n\E_D\{(W_n-W_n^{(r)})^2\}
\leq \frac{648p^2}{n}.
\]

By Chebyshev's inequality,
\[
\P_D\left(W_n<\frac{Bp}{2\sqrt n}\right)
\leq
\frac{4\operatorname{Var}_D(W_n)}{\|\boldsymbol\mu_n\|_2^2}
\leq
\frac{4\cdot 648}{B^2}
<1-\beta.
\]
On the complementary event,
\[
T_n=\|\boldsymbol\xi_n\|_2^2
\geq \langle u_n,\boldsymbol\xi_n\rangle^2
=W_n^2
\geq \frac{B^2p^2}{4n}.
\]
By the choice of $B$, we have $B^2p^2/(4n)>a_n$ for all sufficiently large $n$. Therefore, uniformly over all $D\in\Dc_p$ satisfying $\|\overline{\mXi}_n-\mf{I}_p\|_{\mathrm F}\geq Cpn^{-1/2}$,
\[
\E_D\{\phi_\alpha(Q_{\xi,2})\}
=
\P_D(T_n>a_n)
\geq
1-\frac{4\cdot 648}{B^2}
>
\beta.
\]
Taking the infimum over the indicated class and then the limit inferior proves the theorem.
\end{proof}

The proof of Corollary~\ref{cor:equi-gaussian-power-Qxi2} is given below.
\begin{proof}[Proof of Corollary~\ref{cor:equi-gaussian-power-Qxi2}]
Let $C_0=C_0(\alpha,\beta)$ be the constant in
Theorem~\ref{thm:power-Qxi2}. We first show that, for $C$ sufficiently
large, the implication
\[
\|\fR-\fI_p\|_{\mathrm F}\geq C pn^{-1/4}
\quad\Longrightarrow\quad
\|\overline{\mXi}_n-\mf{I}_p\|_{\mathrm F}\geq C_0pn^{-1/2}
\]
holds for all sufficiently large $n$. Suppose otherwise. Then, along a
violating subsequence, there are equicorrelated Gaussian alternatives satisfying the
left-hand side but not the right-hand side. Under the equicorrelated Gaussian model, all off-diagonal pairs have the
same bivariate Gaussian distribution with Pearson correlation $\rho=\rho_n$.
Thus all off-diagonal entries of $\overline{\mXi}_n=\E_D\mXi_n$ are equal to
\[
\mu_n(\rho)
=
\E_{\rho}\{\xi_n(\fX_{\cdot i},\fX_{\cdot j})\},
\qquad i\ne j,
\]
and hence
\[
\|\overline{\mXi}_n-\mf{I}_p\|_{\mathrm F}^2
=
p(p-1)\mu_n(\rho)^2.
\]
Moreover,
\[
\|\fR-\fI_p\|_{\mathrm F}
=
\sqrt{p(p-1)}\,|\rho|.
\]
Therefore the assumed lower bound on $\|\fR-\fI_p\|_{\mathrm F}$
implies
\be\label{eq-lower-bound-rho}
|\rho|
\geq
C\,n^{-1/4}\sqrt{\frac{p}{p-1}}.
\ee

If $\rho=\rho_n$ does not converge to 0, then it has a further subsequence on which it is
bounded away from zero. Along this further subsequence,
consistency of Chatterjee's $\xi$ implies that $\mu_n(\rho)$ is also bounded away from zero. Consequently
$\|\overline{\mXi}_n-\mf{I}_p\|_{\mathrm F}\asymp p$, contradicting the
failure of the right-hand side. Thus any violating sequence must satisfy
$\rho\to0$. The lower bound \eqref{eq-lower-bound-rho} gives $\liminf n\rho^4>0$.
By Lemma~\ref{lem:gaussian-local-xi}, there exists a constant $c_\xi>0$
such that
\[
\mu_n(\rho)\geq c_\xi\rho^2
\]
for all sufficiently large $n$. Hence
\[
\|\overline{\mXi}_n-\mf{I}_p\|_{\mathrm F}
\geq
\sqrt{p(p-1)}\,\mu_n(\rho)
\geq
c_\xi\sqrt{p(p-1)}\,\rho^2.
\]
Again using \eqref{eq-lower-bound-rho}, we obtain
\[
\|\overline{\mXi}_n-\mf{I}_p\|_{\mathrm F}
\geq
c_\xi C^2\sqrt{p(p-1)}
\,\frac{p}{p-1}n^{-1/2}
\geq
c_\xi C^2pn^{-1/2}.
\]
Choose $C$ large enough so that $c_\xi C^2\geq C_0$. Then
\[
\|\fR-\fI_p\|_{\mathrm F}\geq C pn^{-1/4}
\quad\Longrightarrow\quad
\|\overline{\mXi}_n-\mf{I}_p\|_{\mathrm F}\geq C_0pn^{-1/2}
\]
eventually, contradicting the chosen violating sequence. Therefore, for all
sufficiently large $n$,
\[
\Nc_p^{\mathrm{equi}}
(\|\fR-\fI_p\|_{\mathrm F}\geq C pn^{-1/4})
\subseteq
\Dc_p(\|\overline{\mXi}_n-\mf{I}_p\|_{\mathrm F}\geq C_0pn^{-1/2}).
\]
Applying Theorem~\ref{thm:power-Qxi2} gives
\bea\nonumber
&\liminf_{n\to\infty}
\inf_{D\in\Nc_p^{\mathrm{equi}}
(\|\fR-\fI_p\|_{\mathrm F}\geq C pn^{-1/4})}
\E_D\{\phi_\alpha(Q_{\xi,2})\}\\
\geq
&\liminf_{n\to\infty}
\inf_{D\in \Dc_p(\|\overline{\mXi}_n-\mf{I}_p\|_{\mathrm F}\geq C_0pn^{-1/2})}
\E_D \{\phi_\alpha(Q_{\xi,2})\}\\
>&\beta,
\eea
which proves the claim.
\end{proof}

\subsection{Proof of Theorem \ref{thm-optimal-Bonferroni}}
\begin{proof}[Proof of Theorem \ref{thm-optimal-Bonferroni}]
By definition of the Bonferroni test $B_{\xi,\tau}$, 
$$
\E_D[\phi_\alpha(B_{\xi,\tau})]=\P_D(\{Q_{\xi,2}>z_{\alpha/2}\}\cup \{T_{\tau}>z_{\alpha/2}\})\geq  \P_D(T_{\tau}>z_{\alpha/2})=\E_D[\phi_{\alpha/2}(T_\tau)].
$$

By \citep[Theorem 5.2]{leung2018testing}, there exists contsant $C=C(\alpha,\beta,\gamma)>0$ such that
$$
\liminf_{n\to\infty}\inf_{D\in\Nc_p^{\mathrm{equi}}
(\|\fR-\fI_p\|_{\mathrm F}\ge C)}
\E_D\{\phi_{\alpha/2}(T_{\tau})\}
>\beta.
$$
This implies
    $$
    \liminf_{n\to\infty}\inf_{D\in\Nc_p^{\mathrm{equi}}
(\|\fR-\fI_p\|_{\mathrm F}\ge C)}
\E_D\{\phi_\alpha(B_{\xi,\tau})\}
>\beta.
    $$
\end{proof}

\section{Proofs of auxiliary results}



\subsection{Proofs in Section \ref{Section-technical-tools}}

\subsubsection{Proof of Proposition \ref{lem:key}}
Recall from Section \ref{sec:graph-notation} that $\Delta^0,\ \overline{\Delta}^0$ are the directed, undirected skeleton graph of $\Delta$, respectively.
 ~\\   
{\bf One direction.} If $\overline{\Delta}^0$ is not a tree, then it contains a cycle $L$ of length at least $2$.  

First, suppose that the edges in $\Delta^0$ along $L$ all have the same direction.  
That is, $L = (i_1, i_2, \dots, i_k, i_1)$ with $k \geq 2$ is a directed path in $\Delta^0$.  
Then we have
\[
R_{i_1} R_{i_k}^{-1} = (R_{i_1} R_{i_2}^{-1}) (R_{i_2} R_{i_3}^{-1}) \cdots (R_{i_{k-1}} R_{i_k}^{-1}).
\]

If, instead, the edges along $L$ are not all in the same direction,  
note that $R_u R_v^{-1} = (R_v R_u^{-1})^{-1}$.  
Hence, $R_{i_1} R_{i_k}^{-1}$ can always be expressed as a product of permutations in $\ml{R} \setminus \{R_{i_1} R_{i_k}^{-1}\}$ or their inverses along $L$.  

Therefore, the random permutations in $\ml{R}$ are dependent.

~\\
{\bf The other direction.}   If $\overline{\Delta}^0$ is a tree, we proceed by induction on $\# V(\Delta^0)$ to show that $\ml{R}$ is a set of i.i.d.\ uniform random permutations on $[n]$.  
The case $\# V(\Delta^0)=2$ is trivial.  

Assume that the random permutations in $\ml{R}$ are jointly independent when $\# V(\Delta^0) = k$, $k \geq 2$.  
Now consider the case $\# V(\Delta^0) = k+1$ and write $V(\Delta^0) = \{i_1, i_2, \dots, i_{k+1}\}$.  
Without loss of generality, by relabeling vertices and possibly reversing edge directions, assume that $i_1$ is a leaf node of $\overline{\Delta}^0$ and $(i_1, i_2) \in E(\Delta^0)$.  

Let $\{\sigma_{uv} : (u,v) \in E(\Delta^0)\}$ be a set of deterministic permutations. Then
\begin{align*}
&\P\Big(\bigcap_{(u,v)\in E(\Delta^0)} \{ R_v R_u^{-1} = \sigma_{uv} \}\Big) \\
=& \P\Big(\bigcap_{(u,v)\in E(\Delta^0)} \{ R_v = \sigma_{uv} R_u \}\Big) \\
=& \sum_{\uptau_1} \P\Big( \{ R_{i_2} = \sigma_{12} \uptau_1 \} \cap 
\bigcap_{\substack{(u,v)\in E(\Delta^0)\\(u,v) \neq (i_1,i_2)}} \{ R_v = \sigma_{uv} R_u \} \,\Big|\, R_{i_1} = \uptau_1 \Big) \, \P(R_{i_1} = \uptau_1) \\
=& \frac{1}{n!} \sum_{\uptau_1} \P\Big( \{ R_{i_2} = \sigma_{12} \uptau_1 \} \cap 
\bigcap_{\substack{(u,v)\in E(\Delta^0)\\(u,v) \neq (i_1,i_2)}} \{ R_v = \sigma_{uv} R_u \} \Big) \\
=& \frac{1}{n!} \sum_{\uptau_1} \P(R_{i_2} = \sigma_{12} \uptau_1) \, 
\P\Big( \bigcap_{\substack{(u,v)\in E(\Delta^0)\\(u,v) \neq (i_1,i_2)}} \{ R_v = \sigma_{uv} R_u \} \Big) \\
=& \frac{1}{n!} \P \Big( \bigcap_{\substack{(u,v)\in E(\Delta^0)\\(u,v) \neq (i_1,i_2)}} \{ R_v R_u^{-1} = \sigma_{uv} \} \Big) \\
=& \left( \frac{1}{n!} \right)^k.
\end{align*}
Here, the third equality holds because $R_{i_1}$ is independent of $\{ R_{i_r} : r \ge 2 \}$, and $(u,v) \neq (i_1,i_2)$ implies $u \neq i_1$ and $v \neq i_1$.  
The fourth equality follows from Proposition~\ref{prop:basic}, and the fifth equality uses the uniformity of $R_{i_2}$.  The last equality follows from the induction hypothesis applied to the sub-tree with nodes $\{i_2, \dots, i_{k+1}\}$ and $(k-1)$ edges.  Then by induction we finish the proof.

\subsubsection{Proof of Proposition \ref{approx1}}

 We first prove the joint distribution is asymptoticly Gaussian. Define
    \bea\nonumber
    H_n^k=\frac{1}{\sqrt{n}(n-1)}\l[\sum_{h=1}^{n-1}\l|(R_{j_k}\circ R_{i_k}^{-1})(h+1)-(R_{j_k}\circ R_{i_k}^{-1})(h)\r|-\frac{n(n-1)}{3}\r],\quad 1\leq k\leq K.
    \eea
Let $U_k(h)=F_{j_k}\l(X_{j_k,R_{i_k}^{-1}(h)}\r)$. From  \citet[Equations (6)-(9)]{MR1378827}, we obtain
    \bea\nonumber
    H_n^k=\frac{1}{\sqrt{n}}\sum_{h=1}^{n-1}\l[\l|U_k(h+1)-U_k(h)\r|+2U_k(h)(1-U_k(h))-\frac{2}{3}\r]+R_n^k,\quad 1\leq k\leq K,
    \eea
    with $R_n^k\stackrel{\P}{\rw}0$ for all $k$. Then 
    \bea\nonumber
    \sqrt{n}\Xi_{i_kj_k}^{(n)}=H_n^k-\frac{\sum_{h=1}^{n-1}\l|(R_{j_k}\circ R_{i_k}^{-1})(h+1)-(R_{j_k}\circ R_{i_k}^{-1})(h)\r|}{\sqrt{n}(n-1)(n+1)}=H_n^k+O\l(\frac{1}{\sqrt{n}}\r).
    \eea
    By using the Cramér-Wold device, it suffices to show that, for any $K$ fixed constants $\{a_k\}_{k=1}^K$, 
    $\sum_{k=1}^K a_k\sqrt{n}\Xi_{i_kj_k}^{(n)}$ is asymptoticly Gaussian. We have
    \bea\nonumber
    \sum_{k=1}^K a_k\sqrt{n}\Xi_{i_kj_k}^{(n)}&=\frac{1}{\sqrt{n}}\sum_{k=1}^K\sum_{h=1}^{n-1}a_k\l[\l|U_k(h+1)-U_k(h)\r|+2U_k(h)(1-U_k(h))-\frac{2}{3}\r]+O\l(\frac{1}{\sqrt{n}}\r).
    \eea
     Then it suffces to show that 
     \bea\nonumber
     W_n=\frac{1}{\sqrt{n}}\sum_{k=1}^K\sum_{h=1}^{n}a_k\l[\l|U_k(h+1)-U_k(h)\r|+2U_k(h)(1-U_k(h))-\frac{2}{3}\r]
     \eea
     is asymptoticly Gaussian, since the extra term $h=n$ contributes also $O\l(n^{-1/2}\r)$. We will proceed using Chatterjee's normal approximation technique for graphical statistics \citep{MR2435859, auddy2021exact,lin2022limit}. Let $i_{(1)}<\cdots<i_{(S)}$ be all distinct elements in row indices $\{i_k:1\leq k\leq K\}$. Let

     \bea\nonumber
     \mathcal{M}=
\left(
\begin{bmatrix}
    X_{i_{(1)},1} \\
    X_{i_{(2)},1} \\
    \vdots \\
    X_{i_{(s)},1}
\end{bmatrix},
\cdots,
\begin{bmatrix}
    X_{i_{(1)},n} \\
    X_{i_{(2)},n} \\
    \vdots \\
    X_{i_{(s)},n}
\end{bmatrix}
\right),\quad
\mathcal{M}'=
\left(
\begin{bmatrix}
    X_{i_{(1)},1}' \\
    X_{i_{(2)},1}' \\
    \vdots \\
    X_{i_{(s)},1}'
\end{bmatrix},
\cdots,
\begin{bmatrix}
    X_{i_{(1)},n}' \\
    X_{i_{(2)},n}' \\
    \vdots \\
    X_{i_{(s)},n}'
\end{bmatrix}
\right)
     \eea
be two i.i.d. samples at the rows involved. Let 
\bea\nonumber
\mathcal{M}^\ell=
\left(
\begin{bmatrix}
    X_{i_{(1)},1} \\
    X_{i_{(2)},1} \\
    \vdots \\
    X_{i_{(s)},1}
\end{bmatrix},
\cdots,
\begin{bmatrix}
    X_{i_{(1)},\ell}' \\
    X_{i_{(2)},\ell}' \\
    \vdots \\
    X_{i_{(s)},\ell}'
\end{bmatrix},
\cdots,
\begin{bmatrix}
    X_{i_{(1)},n} \\
    X_{i_{(2)},n} \\
    \vdots \\
    X_{i_{(s)},n}
\end{bmatrix}
\right),
\eea

\bea\nonumber
\mathcal{M}^r=
\left(
\begin{bmatrix}
    X_{i_{(1)},1} \\
    X_{i_{(2)},1} \\
    \vdots \\
    X_{i_{(s)},1}
\end{bmatrix},
\cdots,
\begin{bmatrix}
    X_{i_{(1)},r}' \\
    X_{i_{(2)},r}' \\
    \vdots \\
    X_{i_{(s)},r}'
\end{bmatrix},
\cdots,
\begin{bmatrix}
    X_{i_{(1)},n} \\
    X_{i_{(2)},n} \\
    \vdots \\
    X_{i_{(s)},n}
\end{bmatrix}
\right),
\eea

\bea\nonumber
\mathcal{M}^{\ell r}=
\left(
\begin{bmatrix}
    X_{i_{(1)},1} \\
    X_{i_{(2)},1} \\
    \vdots \\
    X_{i_{(s)},1}
\end{bmatrix},
\cdots,
\begin{bmatrix}
    X_{i_{(1)},\ell}' \\
    X_{i_{(2)},\ell}' \\
    \vdots \\
    X_{i_{(s)},\ell}'
\end{bmatrix},
\begin{bmatrix}
    X_{i_{(1)},r}' \\
    X_{i_{(2)},r}' \\
    \vdots \\
    X_{i_{(s)},r}'
\end{bmatrix},
\cdots,
\begin{bmatrix}
    X_{i_{(1)},n} \\
    X_{i_{(2)},n} \\
    \vdots \\
    X_{i_{(s)},n}
\end{bmatrix}
\right)
\eea
be the replacement with i.i.d. copies at column $\ell,r$, and at both column $\ell$ and $r$, respectively. For $s\in [S]$ and $j',j''\in [n]$, define 
\bea\nonumber
\ml{D}_{\ml{M}}^s(j',j'')=\begin{cases}
    \infty,\quad \text{if}\  X_{i(s),j'}>X_{i(s),j''},\\
    \# \{t:\ X_{i(s),j'}<X_{i(s),t}<X_{i(s),j''}\},\quad \text{if}\ X_{i(s),j'}<X_{i(s),j''}.
\end{cases}
\eea

Now, the graphical rule $\ml{G}(\ml{M})$ on $[n]$ is defined as follows. For any two indices $j',j''\in [n]$ an edge occurs between them, if there exists $s\in [S]$ and $j'''\in [n]$ such that, $$\max\{\ml{D}_{\ml{M}}^s(j''',j'),\ml{D}_{\ml{M}}^s(j''',j'')\}\leq 2.$$ 

The existence of an edge is only related to ranks of the entries in their corresponding rows. Hence such graphical rule $\ml{G}(\ml{M})$ is a symmetric rule.

Now we show that, $\ml{G}(\ml{M})$ is also an interaction rule. First, we can write
\bea\nonumber
W_n=W_n(\ml{M})=\sum_{k=1}^{K}W_{n,k}(\ml{M}), 
\eea
\bea\nonumber
W_{n,k}(\ml{M})=\frac{1}{\sqrt{n}}\sum_{h=1}^{n}a_k\l[\l|U_k(h+1)-U_k(h)\r|+2U_k(h)(1-U_k(h))-\frac{2}{3}\r].
\eea

For any pair of indices ${j',j''}$, if no edge occurs between them, then for every $s\in [S]$, there's no $j'''\in[n]$ such that $\max\{\ml{D}_{\ml{M}}^s(j''',j'),\ml{D}_{\ml{M}}^s(j''',j'')\}\leq 2.$ Then from the proof of  \citet[Lemma 2.4 of supplement]{auddy2021exact}, we have for every $k\in [K]$
\bea\nonumber
W_{n,k}(\ml{M})-W_{n,k}(\ml{M}^{j'})-W_{n,k}(\ml{M}^{j''})+W_{n,k}(\ml{M}^{j'j''})=0.
\eea
Then 
\bea\nonumber
W_{n}(\ml{M})-W_{n}(\ml{M}^{j'})-W_{n}(\ml{M}^{j''})+W_{n}(\ml{M}^{j'j''})=0.
\eea
Hence $\ml{G}(\ml{M})$ is a symmetric interaction rule based on $W_n(\ml{M})$.
To apply \citet[Theorem 2.5]{MR2435859},
let $$\Delta_j=W_n(\ml{M})-W_n(\ml{M}^j), M=\max_j |\Delta_j|.$$ There then exists a constant $C_1>0$ such that $\max(M,|\Delta_j|)<\frac{C_1K}{\sqrt{n}}$.
The extended graph $\ml{G}'$ on $[n+4]$ can be defined as follows. For any two indices $j',j''$ an edge occurs between them, if there exists $s\in [S]$ and $j'''\in [n]$ such that, $$\max\{\ml{D}_{\ml{M}}^s(j''',j'),\ml{D}_{\ml{M}}^s(j''',j'')\}\leq 6.$$ The degree of any vertices in $\mathcal{G'}$ is bounded by $C_2K$, $C_2>0$ is constant. Then from  \citet[Theorem 2.5]{MR2435859}, there is a universal constant $C>0$ such that
\bea\nonumber
\delta(W_n)\leq \frac{CK^3}{\sqrt{n}\sigma^2}+\frac{CK^3}{2\sqrt{n}\sigma^3},
\eea
where $\sigma^2=\operatorname{Var}(W_n)$, 
 and $\delta(W_n)$ is the Wasserstein distance between $(W_n-\E W_n)/\sigma$ and standard Gaussian distribution. Analogous with the calculation in the proof of  \citet[Lemma 3]{zhang2022asymptotic}, we have 
 $$
 \sigma^2=\operatorname{Var}(W_n)=C_{a_1,\cdots,a_K}+O\l(\frac{1}{n}\r),
 $$
 where constant $C_{a_1,\cdots,a_K}>0$ depends only on $a_1,\cdots,a_K$. Then $\delta(W_n)\rw 0$, and hence $W_n$ converges in distribution to Gaussian distribution. And then by the Cramér-Wold device, 
 \[
 (\sqrt{n}\Xi_{i_1j_1}^{(n)},\cdots,\sqrt{n}\Xi_{i_Kj_K}^{(n)}) 
 \]
 converges in distribution to a $K$ dimensional Gaussian distribution.

Now it remains to specify the covariance.  
It was shown in \cite{chatterjee2020new} that 
\[
\operatorname{Var}\big(\sqrt{n}\,\Xi^{(n)}_{12}\big) \longrightarrow \frac{2}{5}.
\]  
Consider two distinct pairs $(i_1,j_1)$ and $(i_2,j_2)$.  

\textbf{Case 1:} $\#\big(\{i_1,j_1\} \cap \{i_2,j_2\}\big) \le 1$.  
In this case, by Proposition~\ref{prop:basic}, $\Xi^{(n)}_{i_1 j_1}$ and $\Xi^{(n)}_{i_2 j_2}$ are independent.  

\textbf{Case 2:} $\#\big(\{i_1,j_1\} \cap \{i_2,j_2\}\big) = 2$, i.e., $i_1 = j_2$ and $i_2 = j_1$.  
Then, by \citet[Corollary 1]{zhang2022asymptotic}, 
\[
\operatorname{Cov}\big(\sqrt{n}\,\Xi^{(n)}_{i_1 j_1}, \sqrt{n}\,\Xi^{(n)}_{j_1 i_1}\big) = O\Big(\frac{1}{n}\Big).
\]  

Therefore, the covariance of the limiting Gaussian distribution is $\frac{2}{5}\,\mf{I}_K$.

\subsection{Proofs in Section \ref{Section-combinatornics}}

\subsubsection{Proof of Proposition \ref{Graphical-Representation}}

By definition,
\bea\label{EXiD}
&\E[\mf{\Xi}_n(\Delta)]=\Big(\frac{3}{n^2-1}\Big)^{\# E(\Delta)}\E\Bigg[\prod_{(u,v)\in E(\Delta^0)} \Bigg(\sum_{k=1}^{n-1} a\Big(R_vR_u^{-1}(k),R_vR_u^{-1}(k+1)\Big)\Bigg)^{s_{uv}}\Bigg]\\
&=\Big(\frac{3}{n^2-1}\Big)^{\# E(\Delta)}\sum_{\cK}
\E\Bigg[\prod_{\substack{(u,v)\in E(\Delta)\\1\leq s\leq s_{uv}}}a\Big(R_vR_u^{-1}(k_{uv}^s),R_vR_u^{-1}(k_{uv}^s+1)\Big)\Bigg]\\
&=\Big(\frac{3}{n^2-1}\Big)^{\# E(\Delta)}\sum_{\cK}\sum_{\cL(\Delta)}
\E\Bigg[\prod_{\substack{(u,v)\in E(\Delta)\\1\leq s\leq s_{uv}}}a\Big(R_vR_u^{-1}(k_{uv}^s),R_vR_u^{-1}(k_{uv}^s+1)\Big)\ \Bigg|\ A_{\cK,\cL(\Delta)}\Bigg] \P[A_{\cK,\cL(\Delta)}].
\eea
Here the summation $\sum_{\mathcal{K}}$ is over the index set 
\[
\mathcal{K} = \{ k_{uv}^s : (u,v) \in E(\Delta^0),\ 1 \le s \le s_{uv} \},
\] 
where $s$ appears as a superscript, and each $k_{uv}^s$ ranges from $1$ to $n-1$.  

Recall that the summation $\sum_{\cL(\Delta)}$ is over the index set
\[
\cL(\Delta) = \{ \ell_{uv}^{s,r} : (u,v) \in E(\Delta^0),\ 1 \le s \le s_{uv},\ r \in \{1,2\} \},
\] 
where each $\ell_{uv}^{s,r}$ ranges from $1$ to $n$, subject to the constraint $\ell_{uv}^{s,1} \ne \ell_{uv}^{s,2}$.  

For each $\mathcal{K}$ and $\cL(\Delta)$, define the event
\[
A_{\mathcal{K}, \cL(\Delta)} = \Big\{ R_u(\ell_{uv}^{s,1}) = k_{uv}^s, \ R_u(\ell_{uv}^{s,2}) = k_{uv}^s + 1,\ \text{for all } (u,v) \in E(\Delta^0),\ 1 \le s \le s_{uv} \Big\}.
\]
Since $\{R_u : u \in V(\Delta)\}$ are i.i.d. uniform random permutations on $[n]$, we have
\begin{align*}
\P[A_{\mathcal{K},\cL(\Delta)}] 
&= \P\Bigg[ \bigcap_{u \in V(\Delta)} \Big\{ R_u(\ell_{uv}^{s,1}) = k_{uv}^s, \ R_u(\ell_{uv}^{s,2}) = k_{uv}^s + 1, \ \text{for all } v \in N^+(u),\ 1 \le s \le s_{uv} \Big\} \Bigg] \\
&= \prod_{u \in V(\Delta)} \P\Bigg[ R_u(\ell_{uv}^{s,1}) = k_{uv}^s, \ R_u(\ell_{uv}^{s,2}) = k_{uv}^s + 1, \ \text{for all } v \in N^+(u),\ 1 \le s \le s_{uv} \Bigg].
\end{align*}
Furthermore,
\begin{align*}
&\E\Bigg[ \prod_{\substack{(u,v)\in E(\Delta)\\1 \le s \le s_{uv}}} a\Big(R_v R_u^{-1}(k_{uv}^s), R_v R_u^{-1}(k_{uv}^s + 1)\Big) \ \Big| \ A_{\mathcal{K},\cL(\Delta)} \Bigg] \\
=& \E\Bigg[ \prod_{\substack{(u,v)\in E(\Delta)\\1 \le s \le s_{uv}}} a\Big(R_v(\ell_{uv}^{s,1}), R_v(\ell_{uv}^{s,2})\Big) \ \Big| \ A_{\mathcal{K},\cL(\Delta)} \Bigg] \\
=& \prod_{v \in V(\Delta)} \E\Bigg[ \prod_{\substack{u \in N^-(v)\\1 \le s \le s_{uv}}} a\Big(R_v(\ell_{uv}^{s,1}), R_v(\ell_{uv}^{s,2})\Big) \ \Big| \ 
A_{\mathcal{K},\cL(\Delta)}^v
\Bigg],
\end{align*}
where for $v \in V(\Delta)$, we have defined the event
\[
A_{\mathcal{K},\cL(\Delta)}^v = \Big\{ R_v(\ell_{vw}^{s,1}) = k_{vw}^s, \ R_v(\ell_{vw}^{s,2}) = k_{vw}^s + 1, \ \text{for all } w \in N^+(v),\ 1 \le s \le s_{vw} \Big\}.
\]
It follows that
\[
\sum_{\mathcal{K}_v^+} A_{\mathcal{K},\cL(\Delta)}^v
= \bigcap_{\substack{w \in N^+(v) \\ 1 \le s \le s_{vw}}} \Big\{ R_v(\ell_{vw}^{s,1}) + 1 = R_v(\ell_{vw}^{s,2}) \Big\},
\]
where the summation $\sum_{\mathcal{K}_v^+}$ runs over the index set
\[
\mathcal{K}_v^+ = \Big\{ k_{vw}^s : w \in N^+(v),\ 1 \le s \le s_{vw} \Big\},
\]
with each $k_{vw}^s$ taking values from $1$ to $n-1$.

Then \eqref{EXiD} becomes
{\small
\bea\nonumber
&\E[\mf{\Xi}_n(\Delta)]\\=&\Big(\frac{3}{n^2-1}\Big)^{\# E(\Delta)}\sum_{\cK}\sum_{\cL(\Delta)}\prod_{v\in V(\Delta)}\E\Bigg[\prod_{\substack{u\in N^-(v)\\1\leq s\leq s_{uv}}}a\Big(R_v(\ell_{uv}^{s,1}),R_v(\ell_{uv}^{s,2})\Big)\ \Bigg|\ A_{\cK,\cL(\Delta)}^v\Bigg]\P[A_{\cK,\cL(\Delta)}^v]\\
=&\Big(\frac{3}{n^2-1}\Big)^{\# E(\Delta)}\sum_{\cL(\Delta)}\prod_{v\in V(\Delta)}\sum_{\cK_v^+}\E\Bigg[\prod_{\substack{u\in N^-(v)\\1\leq s\leq s_{uv}}}a\Big(R_v(\ell_{uv}^{s,1}),R_v(\ell_{uv}^{s,2})\Big)\ \Bigg|\ A_{\cK,\cL(\Delta)}^v\Bigg]\P[A_{\cK,\cL(\Delta)}^v]\\
=&\Big(\frac{3}{n^2-1}\Big)^{\# E(\Delta)}\sum_{\cL(\Delta)}\prod_{v\in V(\Delta)}\E\Bigg[\prod_{\substack{u\in N^-(v)\\1\leq s\leq s_{uv}}}a\Big(R_v(\ell_{uv}^{s,1}),R_v(\ell_{uv}^{s,2})\Big)\cdot \ind\Big(\sum_{\cK_v^+}A_{\cK,\cL(\Delta)}^v\Big) \Bigg]\\
=&\Big(\frac{3}{n^2-1}\Big)^{\# E(\Delta)}\sum_{\cL(\Delta)}\prod_{v\in V(\Delta)}\E\Bigg[\prod_{\substack{u\in N^-(v)\\1\leq s\leq s_{uv}}}a\Big(R_v(\ell_{uv}^{s,1}),R_v(\ell_{uv}^{s,2})\Big)\cdot\prod_{\substack{w\in N^+(v)\\1\leq s\leq s_{vw}} }\ind\Big(R_v(\ell_{vw}^{s,1})+1=R_v(\ell_{vw}^{s,2})\Big) \Bigg]\\
=&\Big(\frac{3}{n^2-1}\Big)^{\# E(\Delta)}\sum_{\cL(\Delta)}\prod_{v\in V(\Delta)}\E\Bigg[\prod_{\substack{u\in N^-(v)\\1\leq s\leq s_{uv}}}a\Big(\sigma(\ell_{uv}^{s,1}),\sigma(\ell_{uv}^{s,2})\Big)\cdot\prod_{\substack{w\in N^+(v)\\1\leq s\leq s_{vw}} }\ind\Big(\sigma(\ell_{vw}^{s,1})+1=\sigma(\ell_{vw}^{s,2})\Big) \Bigg].
\eea
}
Here, the second equality follows from the fact that the term
\[
\E\Bigg[\prod_{\substack{u\in N^-(v)\\1\le s\le s_{uv}}} 
a\Big(R_v(\ell_{uv}^{s,1}),R_v(\ell_{uv}^{s,2})\Big) \ \Bigg|\ A_{\mathcal{K},\cL(\Delta)}^v\Bigg] 
\P[A_{\mathcal{K},\cL(\Delta)}^v]
\]
depends only on the summation indices in $\mathcal{K}_v^+$ among all indices in $\mathcal{K}$.  

In the last equality, we recall that $\sigma$ denotes a uniform random permutation on $[n]$. This completes the proof.

\subsubsection{Proof of Lemma \ref{lem1}}
Clearly \eqref{11} and \eqref{12} are equivalent. We will prove \eqref{11} by induction on $s(\uppi)$. The $s(\uppi)=0$ case is trivial from the bound $|a(i,j)|< 2n$. For the case $s(\uppi)=1$, without loss of generality suppose that $\uppi(1)=1,\ \uppi(2)=2,\ \uppi(k)\geq 3,\ 3\le k\leq 2L$. Denote
    $$
    S_n(\uppi)=\sum_{1\leq i_1\ne\cdots\ne i_{\# \uppi}\leq n}\prod_{k=1}^L a(i_{\uppi(2k-1)},i_{\uppi(2k)}).
    $$
    Then since $\sum_{1\leq i\ne j\leq n}a(i,j)=0$,
    \bea\nonumber
    S_n(\uppi)&=
    \sum_{1\leq i_3\ne\cdots\ne i_{\# \uppi}\leq n}\Big(\sum_{\substack{1\leq i_1\ne i_2\leq n\\ i_1,i_2\notin \{i_3,\cdots,i_{\# \uppi}\}}}a(i_1,i_2)\Big)\prod_{k=3}^La(i_{\uppi(2k-1)},i_{\uppi(2k)})\\
    &=-\sum_{1\leq i_3\ne\cdots\ne i_{\# \uppi}\leq n}\sum_{\substack{1\leq i_1\ne i_2\leq n\\ \{i_1,i_2\}\cap \{i_3,\cdots,i_{\# \uppi}\}\ne \emptyset}}a(i_1,i_2)\prod_{k=3}^La(i_{\uppi(2k-1)},i_{\uppi(2k)})\\
    &=O(n^{\# \uppi+L-1}).
    \eea
    Here in the last equality, the summation has $O(n^{\# \uppi-1})$ number of terms and the summand is $O(n^L)$.
    
    Now suppose that \eqref{11} holds for all $s(\uppi)<r,\ r\geq 2$ and consider the case $s(\uppi)=r$. Without loss of generality suppose that $\uppi(k)=k,\ \text{for any }1\leq k\leq2r$ and $\uppi(k)>2r,\ \text{for any }2r<k\leq 2L$. Denote $[i,j]=\{i,i+1,\cdots,j\}$ for $i<j$. We define two types of adjoint partitions of $\uppi$ as
    \bea\nonumber
    \operatorname{adj}_1(\uppi)=\Big\{\uptau\in \mathcal{RP}(L):\ \uptau\Big|_{[1,2r]}=\uppi\Big|_{[1,2r]},\ &\uptau\Big|_{[2r+1,2L]}=\uppi\Big|_{[2r+1,2L]},\\ &\,\uptau([1,2r])\cap\uptau([2r+1,2L])\ne \emptyset\Big\},
    \eea
    \bea\nonumber
    \operatorname{adj}_2(\uppi)=\Big\{\uptau\in \mathcal{RP}(L):\ &\uptau\Big|_{[2r+1,2L]}=\uppi\Big|_{[2r+1,2L]},\\ &\exists(i,j\in [r],i\ne j)\  \uptau([2i-1,2i])\cap\uptau([2j-1,2j])\ne \emptyset\Big\}.
    \eea
    Noticing that
    \begin{align*}
&\Bigg(
\sum_{\substack{
    i_1\ne\cdots\ne i_{2r}\\
    i_1,\dots,i_{2r} \notin \{ i_{2r+1},\dots,i_{\# \uppi} \}
}}
+
\sum_{\substack{
    i_1\ne\cdots\ne i_{2r}\\
    \{i_1,\dots,i_{2r}\} \cap \{ i_{2r+1},\dots,i_{\# \uppi} \} \neq \emptyset
}}
+
\sum_{\substack{
    i_1\ne i_2,\dots,i_{2r-1}\ne i_{2r}\\
    \exists(w,w'\in [r],\, w\ne w')\\
    \{ i_{2w-1},i_{2w} \} \cap \{ i_{2w'-1},i_{2w'} \} \neq \emptyset
}}
\Bigg)
\prod_{k=1}^r a(i_{2k-1},i_{2k}) \notag\\
&= \sum_{i_1\ne i_2,\dots,i_{2r-1}\ne i_{2r}}
\prod_{k=1}^r a(i_{2k-1},i_{2k})
= \prod_{k=1}^r \sum_{i_{2k-1}\ne i_{2k}} a(i_{2k-1},i_{2k}) = 0,
    \end{align*}
    we have 
    \bea\nonumber
    S_n(\uppi)=&\sum_{i_{2r+1}\ne\cdots\ne i_{\# \uppi}}\Big(\sum_{\substack{ i_1\ne\cdots\ne i_{2r}\\ i_1,\cdots,i_{2r}\notin \{i_{2r+1},\cdots,i_{\# \uppi}\}}}\prod_{k=1}^r a(i_{2k-1},i_{2k})\Big)\prod_{k=r+1}^L a(i_{\uppi(2k-1)},i_{\uppi(2k)})\\
    =&-\sum_{i_{2r+1}\ne\cdots\ne i_{\# \uppi}}\Big(\sum_{\substack{ i_1\ne\cdots\ne i_{2r}\\ \{i_1,\cdots,i_{2r}\}\cap \{i_{2r+1},\cdots,i_{\# \uppi}\}\ne\emptyset}}\prod_{k=1}^r a(i_{2k-1},i_{2k})\Big)\prod_{k=r+1}^L a(i_{\uppi(2k-1)},i_{\uppi(2k)})\\
     &-\sum_{i_{2r+1}\ne\cdots\ne i_{\# \uppi}}\Big(\sum_{\substack{ i_1\ne i_2,\cdots,i_{2r-1}\ne i_{2r}\\ \exists(w,w'\in [r],w\ne w')\\\{i_{2w-1},i_{2w}\}\cap\{i_{2w'-1},i_{2w'}\}\ne\emptyset}}\prod_{k=1}^r a(i_{2k-1},i_{2k})\Big)\prod_{k=r+1}^L a(i_{\uppi(2k-1)},i_{\uppi(2k)})\\
     =&-\sum_{\uptau\in \operatorname{adj}_1(\uppi)}S_n(\uptau)-\sum_{\uptau\in \operatorname{adj}_2(\uppi)}S_n(\uptau).
    \eea
    
    Now let $\uptau\in \operatorname{adj}_1(\uppi)$. We then have $$\# \uptau\leq \# \uppi-1,\ s(\uptau)\leq s(\uppi)-1.$$ Since each $j\in [r]$ such that $\uptau([2j-1,2j+1])\cap\uptau([2r+1,2L])$ reduces $\# \uptau$ by at least 1 compared to $\# \uppi$, we have
    $$
    \# \uppi-\# \uptau\geq s(\uppi)-s(\uptau).
    $$
    Then
    \bea\nonumber
    L+\# \uptau-\Big\lceil\frac
    {s(\uptau)}{2}\Big\rceil&\leq L+\# \uppi-s(\uppi)+s(\uptau)-\Big\lceil\frac{s(\uptau)}{2}\Big\rceil\\
    &=L+\# \uppi-s(\uppi)+\Big\lfloor\frac{s(\uptau)}{2}\Big\rfloor\\
    &\leq L+\# \uppi-s(\uppi)+\Big\lfloor\frac{s(\uppi)}{2}\Big\rfloor\\
    &=L+\# \uppi-\Big\lceil\frac
    {s(\uppi)}{2}\Big\rceil.
    \eea
    From the induction assumption,
    $$
    S_n(\uptau)=O\Big(n^{\# \uptau+L-\lceil \frac{s(\uptau)}{2} \rceil}\Big)= O\Big(n^{\# \uppi+L-\lceil \frac{s(\uppi)}{2} \rceil}\Big),\quad \text{for any }\uptau\in \operatorname{adj}_1(\uppi).
    $$

    Now consider $\uptau\in\operatorname{adj}_2(\uppi)$, then $\# \uptau\leq \# \uppi-1,\ s(\uptau)\leq s(\uppi)-1$ still holds. Since each pair of $i\ne j (i,j\in [r])$ such that $\uptau([2i-1,2i])\cap\uptau([2j-1,2j])\ne\emptyset$ reduces $\# \uptau$ by at least $1$ from $\# \uppi$, we have
    $$
    \# \uppi-\# \uptau\geq\Big\lceil\frac
    {s(\uppi)-s(\uptau)}{2}\Big\rceil.
    $$
    Then
    \bea\nonumber
    L+\# \uptau-\Big\lceil\frac
    {s(\uptau)}{2}\Big\rceil&\leq L+\# \uppi-\Big\lceil\frac
    {s(\uppi)-s(\uptau)}{2}\Big\rceil-\Big\lceil\frac{s(\uptau)}{2}\Big\rceil\\
    &\leq L+\# \uppi-\Big\lceil\frac
    {s(\uppi)}{2}\Big\rceil.
    \eea
    From the induction assumption,
    $$
    S_n(\uptau)=O\Big(n^{\# \uptau+L-\lceil \frac{s(\uptau)}{2} \rceil}\Big)= O\Big(n^{\# \uppi+L-\lceil \frac{s(\uppi)}{2} \rceil}\Big),\quad \text{for any }\uptau\in \operatorname{adj}_2(\uppi).
    $$
    Since $\# \operatorname{adj}_1(\uppi)=O(1),\ \# \operatorname{adj}_2(\uppi)=O(1)$, we have $S_n(\uppi)=O\Big(n^{\# \uppi+L-\lceil \frac{s(\uppi)}{2} \rceil}\Big)$ and from induction we finish the proof.

\subsubsection{Proof of Lemma \ref{lem2}}
 \begin{enumerate}[itemsep=-.5ex,label=(\roman*)]
        \item Since a permutation is a one-to-one map from $[n]$ to itself, \eqref{lem2prod} does not degenerate to 0 if and only if $G_\uppi^0$ has no cycles and all its vertices with both in-degrees and out-degrees at most 1. The only possibility of such directed graphs is the disjoint union of chains.
        \item When \eqref{lem2prod} is non-zero, write
        \bea\nonumber
        &\E\Big[\prod_{k=1}^L \ind\Big(\sigma(\uppi(2k-1))+1=\sigma(\uppi(2k))\Big)\Big]
        =\frac{1}{(n)_{\# \uppi}}\sum_{i_1\ne \cdots\ne i_{\# \uppi}}\prod_{k=1}^L \ind\Big(i_{\uppi(2k-1)}+1=i_{\uppi(2k)}\Big).
        \eea
        Let $A=\{\uppi(j_1)\rw\uppi(j_2)\rw\cdots\rw\uppi(j_l)\}$ be one of the {chains} in $G_\uppi^0$, then 
        $$
        \prod_{k=1}^{\ell-1}\ind\Big(i_{\uppi(j_k)}+1=i_{\uppi(j_k+1)}\Big)=\prod_{k=2}^{\ell}\ind\Big(i_{\uppi(j_k)}=i_{\uppi(j_1)}+k-1\Big).
        $$
        Therefore the summation indices $\{i_1,\cdots,i_{\# \uppi}\}$ reduces to those $i_{\uppi(j_1)}$ such that $\uppi(j_1)$ is the starting point of a {chain} in $G_{\uppi}^0$. 
        Since $G_{\uppi}^0$ is the union of {chains}, we have 
        $$\# \{\text{chains in }G_{\uppi}^0\}=\# V(G_\uppi^0)-\# E(G_\uppi^0).$$
        Then we have \eqref{eq19} since $\# \uppi=\# V(G_\uppi^0)$.
    \end{enumerate}
    
\subsubsection{Proof of Proposition \ref{prop2}}
  Denote $\uppi_2=\uppi|_{[2L_1+1,2L]}$. If $\uppi_2$ is not valid, then from Lemma \ref{lem2} we obtain
    $$ \prod_{k=L_1+1}^L\ind\Big(\sigma(\uppi(2k-1))+1=\sigma(\uppi(2k))\Big)
    $$
    is always 0. If $\uppi_2$ is valid, we have
    \bea\nonumber
S^*(\uppi):=&\E\Big[\prod_{k=1}^{L_1}a\Big(\sigma(\uppi(2k-1)),\sigma(\uppi(2k))\Big)\prod_{k=L_1+1}^L\ind\Big(\sigma(\uppi(2k-1))+1=\sigma(\uppi(2k))\Big)\Big]
\\=&\frac{1}{(n)_{\# \uppi}}\sum_{i_1\ne\cdots\ne i_{\# \uppi}}\prod_{k=1}^{L_1}a\Big(i_{\uppi(2k-1)},i_{\uppi(2k)}\Big)\prod_{k=L_1+1}^L\ind\Big(i_{\uppi(2k-1)}+1=i_{\uppi(2k)}\Big).
    \eea
Now let $A \subset G_{\uppi_2}^0$ be a chain. Consider those $A$ for which $\#(A \cap \uppi([2L_1])) \ge 2$.  
Suppose that $\uppi(j_1), \uppi(j_2)$ with $1 \le j_1 \ne j_2 \le 2L_1$ lie in the same segment of $G_{\uppi_2}^0$. Then the summand being nonzero implies
\[
i_{\uppi(j_1)} + d(j_1,j_2) = i_{\uppi(j_2)},
\]
where $d(j_1,j_2)$ is the signed distance between $\uppi(j_1)$ and $\uppi(j_2)$ in $G_{\uppi_2}^0$, taken as positive if the path goes from $\uppi(j_1) \to \cdots \to \uppi(j_2)$ and negative if it goes from $\uppi(j_2) \to \cdots \to \uppi(j_1)$.

Since $d(j_1,j_2)$ is fixed, as $n \to \infty$ we can replace the summation indices $i_{\uppi(j_1)}$ and $i_{\uppi(j_2)}$ with a single index without changing the asymptotic order of $S^*(\uppi)$. Repeating this procedure for all such pairs $j_1 \ne j_2$, we effectively merge $j_1$ and $j_2$ in the partition $\uppi|_{[2L_1]}$ whenever their indices are merged.

By the definition of $\uppi_1^*$, the first product in the summation then reduces to
\[
\prod_{k=1}^{L_1} a\Big(i_{\uppi_1^*(2k-1)}, i_{\uppi_1^*(2k)}\Big),
\]
and the summation indices
\[
\{ i_{\uppi(j)} : j \le 2L_1 \text{ or } \uppi(j) \in A,\ \#(A \cap \uppi([2L_1])) \ge 2 \}
\]
reduce to $\{ i_1^*, \dots, i_{\# \uppi_1^*}^* \}$ after relabeling.

Now we consider those $A$ such that $\#(A \cap \uppi([2L_1])) = 1$.  
For each such $A$, among the indices $\{i_{\uppi(j)} : \uppi(j) \in A\}$, there is exactly one $i_{\uppi(j_0)}$ with $j_0 \le 2L_1$ that already occurs in $\{i_1^*, \dots, i_{\# \uppi_1^*}^*\}$.  
Then, for those $\uppi(j)$ with $j > 2L_1$ that lie in the same segment as $\uppi(j_0)$ in $G_{\uppi_2}^0$, the summand being nonzero implies the constraint
\[
i_{\uppi(j)} + d(j,j_0) = i_{\uppi(j_0)},
\] 
so that the index $i_{\uppi(j)}$ is effectively determined and can be eliminated from the summation.

For those $A$ such that $A \cap \uppi([2L_1]) = \emptyset$, the same constraint for nonzero summands implies that the summation indices
\[
I_2 = \{ i_{\uppi(j)} : \uppi(j) \in A, \ A \cap \uppi([2L_1]) = \emptyset \}
\] 
reduce to choosing exactly one representative from each such segment $A$, giving a total of $h(\uppi)$ free indices.

Combining all the arguments above, we obtain
\[
S^*(\uppi) = \big(1+o(1)\big)\ {n^{-\# \uppi}} \sum_{i_1^* \ne \cdots \ne i_{\# \uppi_1^*}^*} 
\prod_{k=1}^{L_1} a\Big(i_{\uppi_1^*(2k-1)}, i_{\uppi_1^*(2k)}\Big) \cdot n^{h(\uppi)}
= O\Bigg(n^{-\# \uppi + L_1 - \Big\lceil \frac{s(\uppi_1^*)}{2} \Big\rceil + \# \uppi_1^* + h(\uppi)}\Bigg),
\]
where the last equality follows from Lemma \ref{lem1} and Remark \ref{remark:lem1}.

\subsubsection{Proof of Proposition \ref{Graph-Estimate-EDelta0}}
Recall that an edge $e=(u,v)\in E(\Delta^0)$ is said to be a single edge if the multiplicity $s_{uv}=1$. If further more $(v,u)\notin E(\Delta^0)$, then $e$ is said to be a reducible single edge. If $s_{uv}\geq 2$, $e$ is said to be a multiple edge. We denote $E_\text{RSE}(\Delta^0)$ as the set of reducible single edges.

Now let $e_1=(u_1,v_1)\in E_{\text{RSE}}(\Delta^0)$. Denote $\cL_{e_1}(\Delta)=\{\ell_{u_1v_1}^{1,1},\ell_{u_1v_1}^{1,2}\}$. From Proposition \ref{Graphical-Representation}, $\E[\mf{\Xi}_n(\Delta)]=S_{11}+S_{12}$, where 
\bea\nonumber
S_{11}=\Big(\frac{3}{n^2-1}\Big)^{\# E(\Delta)}\sum_{\cL\1(\Delta)}\prod_{v\in V(\Delta)}G_v,\quad 
S_{12}=\Big(\frac{3}{n^2-1}\Big)^{\# E(\Delta)}\sum_{\wt{\cL}\1(\Delta)}\prod_{v\in V(\Delta)}G_v.
\eea
Here
\bea\nonumber
\cL\1(\Delta)&=\{\ell_{uv}^{s,r}:(u,v)\in E(\Delta^0), 1\leq s\leq s_{uv}, r\in \{1,2\}, \cL_{e_1}(\Delta)\cap (\cL_{u_1}(\Delta)\setminus \cL_{e_1}(\Delta))\ne\emptyset\},\\
\wt{\cL}\1(\Delta)&=\{\ell_{uv}^{s,r}:(u,v)\in E(\Delta^0), 1\leq s\leq s_{uv}, r\in \{1,2\}, \cL_{e_1}(\Delta)\cap (\cL_{u_1}(\Delta)\setminus \cL_{e_1}(\Delta))=\emptyset\},\\
G_v&=\E\Bigg[\prod_{\substack{u\in N^-(v)\\1\leq s\leq s_{uv}}}a\Big(\sigma(\ell_{uv}^{s,1}),\sigma(\ell_{uv}^{s,2})\Big)\cdot\prod_{\substack{w\in N^+(v)\\1\leq s\leq s_{vw}} }\ind\Big(\sigma(\ell_{vw}^{s,1})+1=\sigma(\ell_{vw}^{s,2})\Big) \Bigg],\quad v\in V(\Delta),
\eea
and the summation $\sum_{\cL^{(t)}(\Delta)}$ runs each index $\ell_{uv}^{s,r}$ in $\cL^{(t)}(\Delta)$ from $1$ to $n$ satisfying $\ell_{uv}^{s,1}\ne \ell_{uv}^{s,2}$. 

We have
\bea\nonumber
S_{12}&=\Big(\frac{3}{n^2-1}\Big)^{\# E(\Delta)}\sum_{\cL(\Delta)\setminus \cL_{e_1}(\Delta)}\sum_{\substack{\cL_{e_1}(\Delta)\\\cL_{e_1}(\Delta)\cap (\cL_{u_1}(\Delta)\setminus \cL_{e_1}(\Delta))=\emptyset}}\prod_{v\in V(\Delta)}G_v\\
&=\Big(\frac{3}{n^2-1}\Big)^{\# E(\Delta)}\sum_{\cL(\Delta)\setminus \cL_{e_1}(\Delta)} \l(\prod_{\substack{v\in V(\Delta)\\v\ne u_1,v_1}}G_v\r)\sum_{\substack{\cL_{e_1}(\Delta)\\\cL_{e_1}(\Delta)\cap (\cL_{u_1}(\Delta)\setminus \cL_{e_1}(\Delta))=\emptyset}}G_{u_1}G_{v_1}\\
&=\Big(\frac{3}{n^2-1}\Big)^{\# E(\Delta)}\sum_{\cL(\Delta)\setminus \cL_{e_1}(\Delta)} \l(\prod_{\substack{v\in V(\Delta)\\v\ne u_1,v_1}}G_v\r)\wt{G}_{u_1}^{v_1}\sum_{\substack{\cL_{e_1}(\Delta)\\\cL_{e_1}(\Delta)\cap (\cL_{u_1}(\Delta)\setminus \cL_{e_1}(\Delta))=\emptyset}}G_{v_1}.
\eea
Here
\bea\nonumber
\wt{G}_{u_1}^{v_1}=\E\Bigg[\prod_{\substack{u\in N^-(u_1)\\1\leq s\leq s_{uu_1}}}a\Big(\sigma(\ell_{uu_1}^{s,1}),\sigma(\ell_{uu_1}^{s,2})\Big)\cdot\prod_{\substack{v\in N^+(u_1)\\v\ne v_1\\1\leq s\leq s_{u_1v}} }\ind\Big(\sigma(\ell_{u_1v}^{s,1})+1=\sigma(\ell_{u_1v}^{s,2})\Big)\cdot\ind\Big(\sigma(b_1)+1=\sigma(b_2)\Big) \Bigg],
\eea
where $b_1,b_2\in [n]-(\cL_{u_1}(\Delta)\setminus \cL_{e_1}(\Delta)),\ b_1\ne b_2$. 

We have
\bea\nonumber
&\sum_{\substack{\cL_{e_1}(\Delta)\\\cL_{e_1}(\Delta)\cap (\cL(\Delta)\setminus \cL_{e_1}(\Delta))=\emptyset}}G_{v_1}\\\
=&\sum_{\substack{\cL_{e_1}(\Delta)\\\cL_{e_1}(\Delta)\cap (\cL(\Delta)\setminus \cL_{e_1}(\Delta))=\emptyset}}\E\Bigg[\prod_{\substack{u\in N^-(v_1)\\1\leq s\leq s_{uv_1}}}a\Big(\sigma(\ell_{uv_1}^{s,1}),\sigma(\ell_{uv_1}^{s,2})\Big)\cdot\prod_{\substack{w\in N^+(v_1)\\1\leq s\leq s_{v_1w}} }\ind\Big(\sigma(\ell_{v_1w}^{s,1})+1=\sigma(\ell_{v_1w}^{s,2})\Big) \Bigg]\\
=&\E\Bigg[\prod_{\substack{u\in N^-(v_1)\\u\ne u_1\\1\leq s\leq s_{uv_1}}}a\Big(\sigma(\ell_{uv_1}^{s,1}),\sigma(\ell_{uv_1}^{s,2})\Big)\cdot\l(\sum_{\substack{\cL_{e_1}(\Delta)\\\cL_{e_1}(\Delta)\cap (\cL(\Delta)\setminus \cL_{e_1}(\Delta))=\emptyset}}a\Big(\sigma(\ell_{u_1v_1}^{s,1}),\sigma(\ell_{u_1v_1}^{s,2})\Big)\r)\cdot\\ &\quad \prod_{\substack{w\in N^+(v_1)\\1\leq s\leq s_{v_1w}} }\ind\Big(\sigma(\ell_{v_1w}^{s,1})+1=\sigma(\ell_{v_1w}^{s,2})\Big) \Bigg].
\eea
Since 
$$
\l(\sum_{\substack{\cL_{e_1}(\Delta)\\\cL_{e_1}(\Delta)\cap (\cL(\Delta)\setminus \cL_{e_1}(\Delta))=\emptyset}}+\sum_{\substack{\cL_{e_1}(\Delta)\\\cL_{e_1}(\Delta)\cap (\cL(\Delta)\setminus \cL_{e_1}(\Delta))\ne\emptyset}}\r)a\Big(\sigma(\ell_{u_1v_1}^{s,1}),\sigma(\ell_{u_1v_1}^{s,2})\Big)=\sum_{1\leq i\ne j\leq n}a(i,j)=0,
$$
we have
\bea\nonumber
\sum_{\substack{\cL_{e_1}(\Delta)\\\cL_{e_1}(\Delta)\cap (\cL(\Delta)\setminus \cL_{e_1}(\Delta))=\emptyset}}G_{v_1}=\sum_{\substack{\cL_{e_1}(\Delta)\\\cL_{e_1}(\Delta)\cap (\cL(\Delta)\setminus \cL_{e_1}(\Delta))\ne\emptyset}}(-G_{v_1}).
\eea

Then
\bea\nonumber
S_{12}&=\Big(\frac{3}{n^2-1}\Big)^{\# E(\Delta)}\sum_{\cL(\Delta)\setminus \cL_{e_1}(\Delta)} \l(\prod_{\substack{v\in V(\Delta)\\v\ne u_1,v_1}}G_v\r)\wt{G}_{u_1}^{v_1}\sum_{\substack{\cL_{e_1}(\Delta)\\\cL_{e_1}(\Delta)\cap (\cL_{u_1}(\Delta)\setminus \cL_{e_1}(\Delta))\ne\emptyset}}(-G_{v_1})\\
&=\Big(\frac{3}{n^2-1}\Big)^{\# E(\Delta)}\sum_{\cL\1(\Delta)}\l((-\wt{G}_{u_1}^{v_1})\prod_{\substack{v\in V(\Delta)\\v\ne u_1}}G_v\r).
\eea
Therefore
\bea\nonumber
\E[\mf{\Xi}_n(\Delta)]=\Big(\frac{3}{n^2-1}\Big)^{\# E(\Delta)}\sum_{\cL\1(\Delta)}\l((G_{u_1}-\wt{G}_{u_1}^{v_1})\prod_{\substack{v\in V(\Delta)\\v\ne u_1}}G_v\r).
\eea

Now suppose that there is $e_2=(u_2,v_2)\in E_{\text{RSE}}(\Delta),\ e_2\ne e_1$. Then $\E[\mf{\Xi}_n(\Delta)]=S_{21}+S_{22}$, where
\bea\nonumber
S_{21}&=\Big(\frac{3}{n^2-1}\Big)^{\# E(\Delta)}\sum_{\cL\2(\Delta)}\l((G_{u_1}-\wt{G}_{u_1}^{v_1})\prod_{\substack{v\in V(\Delta)\\v\ne u_1}}G_v\r),\\
S_{22}&=\Big(\frac{3}{n^2-1}\Big)^{\# E(\Delta)}\sum_{\wt{\cL}\2(\Delta)}\l((G_{u_1}-\wt{G}_{u_1}^{v_1})\prod_{\substack{v\in V(\Delta)\\v\ne u_1}}G_v\r).
\eea
Here
\bea\nonumber
\cL\2(\Delta)&=\{\ell_{uv}^{s,r}:\cL_{e_1}(\Delta)\cap (\cL_{u_1}(\Delta)\setminus \cL_{e_1}(\Delta))\ne\emptyset,\cL_{e_2}(\Delta)\cap (\cL_{u_2}(\Delta)\setminus \cL_{e_2}(\Delta))\ne\emptyset\},\\
\wt{\cL}\2(\Delta)&=\{\ell_{uv}^{s,r}:\cL_{e_1}(\Delta)\cap (\cL_{u_1}(\Delta)\setminus \cL_{e_1}(\Delta))\ne\emptyset,\cL_{e_2}(\Delta)\cap (\cL_{u_2}(\Delta)\setminus \cL_{e_2}(\Delta))=\emptyset\}.
\eea
Then,
\bea\nonumber
S_{22}=&\Big(\frac{3}{n^2-1}\Big)^{\# E(\Delta)}\\ &\quad \cdot \sum_{\substack{\cL(\Delta)\setminus \cL_{e_2}(\Delta)\\\cL_{e_1}(\Delta)\cap (\cL_{u_1}(\Delta)\setminus \cL_{e_1}(\Delta))\ne\emptyset}}~~\sum_{\substack{\cL_{e_2}(\Delta)\\\cL_{e_2}(\Delta)\cap (\cL_{u_2}(\Delta)\setminus \cL_{e_2}(\Delta))=\emptyset}}\l((G_{u_1}-\wt{G}_{u_1}^{v_1})\prod_{\substack{v\in V(\Delta)\\v\ne u_1}}G_v\r).
\eea

\medskip
\noindent\underline{Case I:}\ $u_1\ne u_2$. We have
    \begin{eqnarray*} 
    & &\Big(\frac{3}{n^2-1}\Big)^{-\# E(\Delta)} S_{22}\\
    &=&\sum_{\substack{\cL(\Delta)\setminus \cL_{e_2}(\Delta)\\\cL_{e_1}(\Delta)\cap (\cL_{u_1}(\Delta)\setminus \cL_{e_1}(\Delta))\ne\emptyset}} \l(\prod_{\substack{v\in V(\Delta)\\v\ne u_1,u_2,v_2}}G_v\r)(G_{u_1}-\wt{G}_{u_1}^{v_1})\sum_{\substack{\cL_{e_2}(\Delta)\\\cL_{e_2}(\Delta)\cap (\cL_{u_2}(\Delta)\setminus \cL_{e_2}(\Delta))=\emptyset}}F_{u_2}F_{v_2}\\
    &=& \sum_{\substack{\cL(\Delta)\setminus \cL_{e_2}(\Delta)\\\cL_{e_1}(\Delta)\cap (\cL_{u_1}(\Delta)\setminus \cL_{e_1}(\Delta))\ne\emptyset}} \l(\prod_{\substack{v\in V(\Delta)\\v\ne u_1,u_2,v_2}}G_v\r)(G_{u_1}-\wt{G}_{u_1}^{v_1})(-\wt{G}_{u_2}^{v_2})\sum_{\substack{\cL_{e_2}(\Delta)\\\cL_{e_2}(\Delta)\cap (\cL_{u_2}(\Delta)\setminus \cL_{e_2}(\Delta))\ne\emptyset}}F_{v_2}\\
    &=& \sum_{\cL\2(\Delta)}\l((G_{u_1}-\wt{G}_{u_1}^{v_1})(-\wt{G}_{u_2}^{v_2})\prod_{\substack{v\in V(\Delta)\\v\ne u_1,u_2}}G_v\r).
    \end{eqnarray*}
Then
\bea\nonumber
\E[\mf{\Xi}_n(\Delta)]=\Big(\frac{3}{n^2-1}\Big)^{\# E(\Delta)}\sum_{\cL\2(\Delta)}\l((G_{u_1}-\wt{G}_{u_1}^{v_1})(F_{u_2}-\wt{G}_{u_2}^{v_2})\prod_{\substack{v\in V(\Delta)\\v\ne u_1,u_2}}G_v\r).
\eea

\medskip
\noindent\underline{Case II}:\ $u_1=u_2,\ v_1\ne v_2$. We have
 \begin{eqnarray*} 
    & &\Big(\frac{3}{n^2-1}\Big)^{-\# E(\Delta)} S_{22}\\
    &= &\sum_{\substack{\cL(\Delta)\setminus \cL_{e_2}(\Delta)\\\cL_{e_1}(\Delta)\cap (\cL_{u_1}(\Delta)\setminus (\cL_{e_1}(\Delta)\cup \cL_{e_2}(\Delta)))\ne\emptyset}} \l(\prod_{\substack{v\in V(\Delta)\\v\ne u_1,v_2}}G_v\r)\sum_{\substack{\cL_{e_2}(\Delta)\\\cL_{e_2}(\Delta)\cap (\cL_{u_2}(\Delta)\setminus \cL_{e_2}(\Delta))=\emptyset}}(G_{u_1}-\wt{G}_{u_1}^{v_1})F_{v_2}\\
    &=& \sum_{\substack{\cL(\Delta)\setminus \cL_{e_2}(\Delta)\\\cL_{e_1}(\Delta)\cap (\cL_{u_1}(\Delta)\setminus (\cL_{e_1}(\Delta)\cup \cL_{e_2}(\Delta)))\ne\emptyset}} \l(\prod_{\substack{v\in V(\Delta)\\v\ne u_1,v_2}}G_v\r)(\wt{G}_{u_1}^{v_1v_2}-\wt{G}_{u_1}^{v_2})\sum_{\substack{\cL_{e_2}(\Delta)\\\cL_{e_2}(\Delta)\cap (\cL_{u_2}(\Delta)\setminus \cL_{e_2}(\Delta))\ne\emptyset}}F_{v_2}\\
    &=& \sum_{\cL\2(\Delta)}\l((\wt{G}_{u_1}^{v_1v_2}-\wt{G}_{u_1}^{v_2})\prod_{\substack{v\in V(\Delta)\\v\ne u_1,u_2}}G_v\r).
\end{eqnarray*}
Here
\bea\nonumber
\wt{G}_{u_1}^{v_1v_2}=&\E\Bigg[\prod_{\substack{u\in N^-(u_1)\\1\leq s\leq s_{uu_1}}}a\Big(\sigma(\ell_{uu_1}^{s,1}),\sigma(\ell_{uu_1}^{s,2})\Big)\cdot\prod_{\substack{v\in N^+(u_1)\\v\ne v_1,v_2\\1\leq s\leq s_{u_1v}} }\ind\Big(\sigma(\ell_{u_1v}^{s,1})+1=\sigma(\ell_{u_1v}^{s,2})\Big)\cdot\\&\quad\quad\ind\Big(\sigma(b_1)+1=\sigma(b_2)\Big)\cdot\ind\Big(\sigma(b_3)+1=\sigma(b_4)\Big) \Bigg],
\eea
where $b_1,b_2,b_3,b_4$ are distinct elements in $[n]-(\cL_{u_1}(\Delta)\setminus (\cL_{e_1}(\Delta)\cup \cL_{e_2}(\Delta)))$. 

Then
\bea\nonumber
\E[\mf{\Xi}_n(\Delta)]=\Big(\frac{3}{n^2-1}\Big)^{\# E(\Delta)}\sum_{\cL\2(\Delta)}\l(
(G_{u_1}-\wt{G}_{u_1}^{v_1}-\wt{G}_{u_1}^{v_2}+\wt{G}_{u_1}^{v_1v_2})\prod_{\substack{v\in V(\Delta)\\v\ne u_1}}G_v\r).
\eea
Now suppose that $E_{\text{RSE}}(\Delta^0)=\{e_q=(u_q,v_q):\ q\in [Q_0]\}$. Define
\bea\nonumber
\cL^{(q)}(\Delta)=\{\ell_{uv}^{s,r}:\ \cL_{e_q'}(\Delta)\cap (\cL_{u_q'}(\Delta)\setminus \cL_{e_q'}(\Delta))=\emptyset,\ \text{for any }q'\in [q]\},\quad q\in [Q_0].
\eea
For $u\in V(\Delta), \cA\subset N^+(u)$,
define
\bea\nonumber
\wt{G}^{\cA}_u=\E\Bigg[\prod_{\substack{w\in N^-(u)\\1\leq s\leq s_{wu}}}a\big(\sigma(\ell_{wu}^{s,1}),\sigma(\ell_{wu}^{s,2})\big)\cdot& \prod_{\substack{v\in (N^+(u)-\cA)\\1\leq s\leq s_{uv}}}\ind\big(\sigma(\ell_{uv}^{s,1})+1=\sigma(\ell_{uv}^{s,2})\big)\\ \cdot& \prod_{v\in \cA}\ind\big(\sigma(b_v^1)+1=\sigma(b_v^2)\big)\Bigg],
\eea
here $\{b_{v}^1,b_v^2:\ v\in \cA\}$ are distinct elements in $[n]\setminus \big(\cL_u(\Delta)\setminus \cup_{v\in \cA}\{\ell_{uv}^{1,1},\ell_{uv}^{1,2}\}\big)$.
Denote $\cU_0=\{u_1,\cdots,u_{Q_0}\},\ \cV_0=\{v_1,\cdots,v_{Q_0}\}$. For $u\in \cU_0$, let $V_{\text{RSE}}(u)=\{v_q\in \cV_0:\ (u,v_q)\in E_{\text{RSE}}(\Delta)\}$. By convention $\wt{G}_u^\emptyset=G_u$.
Then, by induction, we have the following.

\begin{proposition}\label{Graphical-Representation-Reduced} We have
\bea\nonumber
\E[\mf{\Xi}_n(\Delta)]=\Big(\frac{3}{n^2-1}\Big)^{\# E(\Delta)}\sum_{\cL^{(Q_0)}(\Delta)}\l(
\prod_{u\in \cU_0}\bigg(\sum_{\cA\subset V_{\text{RSE}}(u)}(-1)^{\# \cA}\wt{G}_u^\cA\bigg)\cdot\prod_{\substack{u\in (V(\Delta)\setminus \cU_0)}}G_u\r).
\eea
\end{proposition}

Now define
\bea\nonumber
&\ml{RP}_\text{RSE}(\cL(\Delta))\\=&\l\{\uppi\in \ml{RP}(\cL(\Delta)):\ \uppi\big(\cL_{e_q}(\Delta)\big)\cap\bigg(\uppi\big(\cL_{u_q}(\Delta)\big)\setminus \uppi\big(\cL_{e_q}(\Delta)\big)\bigg)\ne\emptyset,\ \text{for any }q\in [Q_0]\r\}.
\eea
For $\rA=(\cA_u)_{u\in \cU_0},\ \cA_u\subset V_\text{RSE}(u)$, define
\bea\nonumber
\wt{G}(\uppi,\rA)=(n)_{\# \uppi}\prod_{u\in \cU_0}\wt{G}_u^{\cA_u}(\uppi)\prod_{u\in V(\Delta)\setminus \cU_0}G_u(\uppi),\quad \uppi\in \ml{RP}_\text{RSE}(\cL(\Delta)).
\eea
Then
\bea\label{Final-graphical-representation}
\E[\mf{\Xi}_n(\Delta)]=\Big(\frac{3}{n^2-1}\Big)^{\# E(\Delta)}\sum_{\uppi\in \ml{RP}_\text{RSE}(\cL(\Delta))}\sum_{\rA=(\cA_u)_{u\in \cU_0}}(-1)^{\sum_{u\in \cU_0}\# \cA_u}\wt{G}(\uppi,\rA).
\eea
To give an estimate of the order of $\E[\mf{\Xi}_n(\Delta)]$, it suffices to consider the order of $\wt{G}(\uppi,\rA)$. In the sum on $\uppi$, those $\uppi\in \ml{RP}_\text{RSE}(\cL(\Delta))$ such that $\uppi(\cL_{e_q}(\Delta))\cap\uppi(\cL_{e_q'}(\Delta))=\emptyset$ for all $e_q,e_q'$ such that $u_q\neq u_q'$ contributes the leading order term of $\E[\mf{\Xi}_n(\Delta)]$. The set of $\uppi$'s with this property is denoted as $\overline{\ml{RP}}_\text{RSE}(\cL(\Delta))$.

Define $\uppi_0$ as the partition on $\cL(\Delta)$ such that every block is a singleton. In general, any $\uppi\in \mathcal{RP}_\text{RSE}(\cL(\Delta))$ such that $\uppi^v|_{\cL_v^+(\Delta)}$ is valid for all $v\in V(\Delta)$ can be obtained from $\uppi_0$ in finite steps
following the two procedures below. For $\uptau\in\ml{RP}(\cL(\Delta))$, we define $\operatorname{merge}(\uptau,\ell_{uv}^{s,r},\ell_{u'v'}^{s',r'})$ as the partition obtained from $\uptau$ by merging the blocks that contain $\ell_{uv}^{s,r},\ell_{u'v'}^{s',r'}$. If $\ell_{uv}^{s,r},\ell_{u'v'}^{s',r'}$ are already in the same block of $\uptau$, then $\operatorname{merge}(\uptau,\ell_{uv}^{s,r},\ell_{u'v'}^{s',r'})=\uptau$.

\textbf{Procedure 1. }For $q\in[Q_0]$, $\uppi_q$ is obtained from $\uppi_{q-1}$ by follows. Since $\uppi\in \overline{\ml{RP}}_\text{RSE}(\cL(\Delta))$, at least one of $\ell_{u_qv_q}^{1,1},\ell_{u_qv_q}^{1,2}$ is not singleton in $\uppi^{u_q}$. For such $\ell_{u_qv_q}^{1,t}$, pick a $\ell_{u_qv'}^{s',t'}$ that is in the same block as $\ell_{u_qv_q}^{1,t}$ in $\uppi$, let $\uppi_{q}=\operatorname{merge}(\uppi_{q-1},\ell_{u_qv_q}^{1,t},\ell_{u_qv'}^{s',t'})$.

\textbf{Procedure 2. }For $q\geq Q_0+1$, if $\uppi_{q-1}=\uppi$, terminate the procedure. If $\uppi_{q-1}\ne \uppi$, 
there is a pair $(\ell_{uv}^{s,r},\ell_{u'v'}^{s',r'})$ such that $\ell_{uv}^{s,r},\ell_{u'v'}^{s',r'}$ are in the same block of $\uppi$ but are in distinct blocks of $\uppi_{q-1}$, and $\ell_{uv}^{s,r}$ is a singleton in $\uppi_{q-1}$. Then let $\uppi_q=\operatorname{merge}(\uppi_{q-1},\ell_{uv}^{s,r},\ell_{u'v'}^{s',r'})$.

By definition, $\uppi_{q}\in\overline{\ml{RP}}_\text{RSE}(\cL(\Delta)),\ \text{for any } q\geq Q_0$. In both procedures we have $\# \uppi_{q}=\# \uppi_{q-1}-1$, therefore the procedure 2 must terminate in finite steps. If $\uppi_{q+1}^v|_{\cL_v^+(\Delta)}$ is valid for all $v\in V(\Delta)$, then $\uppi_{q}^v|_{\cL_v^+(\Delta)}$ is also valid for all $v\in V(\Delta)$. Therefore, $\uppi_{q}^v|_{\cL_v^+(\Delta)}$ is valid for all $v\in V(\Delta)$ and all $q\in [Q_0]$.
From Proposition \ref{Local-Estimate-Partition}, for any $\cA_u\subset V_{\text{RSE}}(u)$,
$$\wt{G}_u^{\cA_u}=G_u(\uppi_0)=O(n^{\chi_u(\uppi_0)}),\quad \chi_u(\uppi_0)=d^-(u)-d^+(u)-\bigg\lceil\frac{d^-(u)}{2}\bigg\rceil.$$
Then for any $\rA=(\cA_u)_{u\in \cU_0}$, 
$
\wt{G}(\uppi_0,\rA)=O(n^{\wt{\chi}(\uppi_0,\rA)})$, where
\bea\nonumber
\wt{\chi}(\uppi_0,\rA)=2\# E(\Delta)-\sum_{u\in V(\Delta)}\bigg\lceil\frac{d^-(u)}{2}\bigg\rceil.
\eea
\begin{lemma}\label{Lemma-Reduce-Procedure-Before-Q0}
    For $q\in [Q_0]$, $\wt{G}(\uppi_q,\rA)=O(n^{\wt{\chi}(\uppi_0,\rA)-\lceil q/2\rceil})$.
\end{lemma}
\begin{proof}[Proof of Lemma \ref{Lemma-Reduce-Procedure-Before-Q0}]
    First, consider $q=1$. Write $\uppi_1=\operatorname{merge}(\uppi_0,\ell_{u_1v_1}^{1,1},\ell_{u_1v'}^{s',t'})$.
    
    \textbf{Case I: }$(u_1,v_1)\notin \cA_{u_1}$. We have
    $$
    \# \uppi_1=\# \uppi_0-1,\quad \# \uppi_1^{u_1}=\# \uppi_0^{u_1}-1,\quad h(\uppi_1^v)=h(\uppi_0^v)-1,\quad \uppi_{11}^{u_1*}=\uppi_{01}^{u_1*}.
    $$
    Then from Proposition \ref{Local-Estimate-Partition}, $\wt{G}_{u_1}(\uppi_1,\rA)=O(n^{{\chi}_{u_1}(\uppi_0)-1})$. For $v\ne u_1$, $\uppi_1^v=\uppi_0^v$, hence $\wt{G}_{v}(\uppi_1,\rA)=O(n^{{\chi}_{v}(\uppi_0)})$. Therefore $\wt{G}(\uppi_1,\rA)=O(n^{\wt{\chi}(\uppi_0,\rA)-1})$.
    
    \textbf{Case II: }$(u_1,v_1)\in \cA_{u_1}$. Then $\uppi_1^u=\uppi_0^u,\ \text{for any } u\in V(\Delta)$, and $\# \uppi_1=\# \uppi_0-1$. Again we have
    $\wt{G}(\uppi_1,\rA)=O(n^{\wt{\chi}(\uppi_0,\rA)-1})$.

    Now consider $q=2$. Write $\uppi_2=\operatorname{merge}(\uppi_1,\ell_{u_2v_2}^{1,1},\ell_{u_2v''}^{s'',t''})$.

    \textbf{Case I: }$u_1=u_2,\ \ell_{u_1v_1}^{1,1}=\ell_{u_2v''}^{s'',t''},\ell_{u_2v_2}^{1,1}=\ell_{u_1v'}^{s',t'}$. Then $\uppi_2=\uppi_1,\ \wt{G}(\uppi_2,\rA)=O(n^{\wt{\chi}(\uppi_0,\rA)-1})$. 

    \textbf{Case II: }$u_1=u_2,\ \# \{\ell_{u_1v_1}^{1,1},\ell_{u_2v_2}^{1,1},\ell_{u_1v'}^{s',r'},\ell_{u_2v''}^{s'',r''}\}\geq 3$. Then
    $$
    \# \uppi_2=\# \uppi_1-1,\quad \# \uppi_2^{u_1}-h(\uppi_2^{u_1})=\# \uppi_1^{u_1}-h(\uppi_1^{u_1}),\quad \uppi_{21}^{u_1*}=\uppi_{11}^{u_1*}.
    $$
    For $v\ne u_1, \uppi_2^v=\uppi_1^v$. Then from Proposition \ref{Local-Estimate-Partition} and the case $q=1$, $\wt{G}(\uppi_2,\rA)=O(n^{\wt{\chi}(\uppi_0,\rA)-2})$.

    \textbf{Case III: }$u_1\ne u_2$. Then
    $$
    \# \uppi_2=\# \uppi_0-2,\quad \# \uppi_2^{u_t}-h(\uppi_2^{u_t})=\# \uppi_0^{u_t}-h(\uppi_0^{u_t}),\ \uppi_{21}^{u_t*}=\uppi_{01}^{u_t*},\ t\in \{1,2\}.
    $$
    For $v\ne u_1,u_2$, $\uppi_2^v=\uppi_0^v$. Then from Proposition \ref{Local-Estimate-Partition}, $\wt{G}(\uppi_2,\rA)=O(n^{\wt{\chi}(\uppi_0,\rA)-2})$.

    The cases $q\geq 3$ can be done by induction on $q$. $\wt{G}(\uppi_q,\rA)$ reaches its maximal order when there is a sequence $k_1<k_2<\cdots<k_b\leq q\ ,\ b=\lceil q/2\rceil$ such that $\uppi_{q_{k_{b'}}}=\uppi_{q_{{k_{b'}}-1}},\ \text{for any }b'\in[b]$.
\end{proof}

\begin{lemma}\label{Lemma-Reduce-Procedure-After-Q0}
    For $q\geq Q_0+1$,
    $\wt{G}(\uppi_q,\rA)=O(n^{\wt{\chi}(\uppi_0,\rA)-\lceil 
Q_0/2 \rceil}).$
\end{lemma}

\begin{proof}[Proof of Lemma \ref{Lemma-Reduce-Procedure-After-Q0}]
    From Proposition \ref{Local-Estimate-Partition}, we have $\wt{G}(\uppi_q,\rA)=O(n^{\wt{\chi}(\uppi_q,\rA)})$, where
    $$
    \wt{\chi}(\uppi_q,\rA)=\# \uppi_q+\sum_{v\in V(\Delta)} \l(-\# \uppi_q^v+d^-(v)-\bigg\lceil\frac{s(\uppi_{q1}^{v*})}{2}\bigg\rceil+\# \uppi_{q1}^{v*}+h(\uppi^v_{q})\r).
    $$
    By induction on $q$, it suffices to prove that $\wt{\chi}(\uppi_q,\rA)\leq \wt{\chi}(\uppi_{q-1},\rA),\ \text{for any }q\geq Q_0+1$. Recall that $\uppi_q=\operatorname{merge}(\uppi_{q-1},\ell_{uv}^{s,r},\ell_{u'v'}^{s',r'})$. Then $\# \uppi_{q}=\# \uppi_{q-1}-1$. For every $v\in V(\Delta)$, $\uppi_{q1}^{w*}$ is either $\uppi_{q-1,1}^{w*}$ or obtained from $\uppi_{q-1,1}^{w*}$ by merging blocks. Therefore
    \bea\label{Lemma-Reduce-Procedure-Before-Q0-CaseI-eq1}
    -\bigg\lceil\frac{s(\uppi_{q1}^{w*})}{2}\bigg\rceil+\# \uppi_{q1}^{w*}\leq -\bigg\lceil\frac{s(\uppi_{q-1,1}^{w*})}{2}\bigg\rceil+\# \uppi_{q-1,1}^{w*}.
    \eea
    Since merging blocks does not increase connected components in $G_{\uppi^w}^0$ and at least eliminate one connected component, we have $h(\uppi^w_{q-1})-1\leq h(\uppi^w_{q})\leq h(\uppi^w_{q-1})$. Now let $f=\# \{u,v,u',v'\}\in\{2,3,4\}$.
    
    \textbf{Case I: }$f\geq 3.$ If $f=4$, then $\# \uppi_{q}^w=\# \uppi_{q-1}^w,\ \text{for any }w\in V(\Delta)$. If $f=3$, suppose $v\ne u=u'\ne v'$, then $\# \uppi_{q}^u=\# \uppi_{q-1}^u-1,\ \# \uppi_q^w=\# \uppi_{q-1}^w,\ \text{for any }w\ne u$. Hence
    \bea\label{Lemma-Reduce-Procedure-Before-Q0-CaseI-eq2}
    \# \uppi_q-\sum_{w\in V(\Delta)}\# \uppi_q^w\leq \# \uppi_{q-1}-\sum_{w\in V(\Delta)}\# \uppi_{q-1}^w.
    \eea
     Then from \eqref{Lemma-Reduce-Procedure-Before-Q0-CaseI-eq1} and \eqref{Lemma-Reduce-Procedure-Before-Q0-CaseI-eq2}, we have $\wt{\chi}(\uppi_q,\rA)\leq \wt{\chi}(\uppi_{q-1},\rA)$.

    \textbf{Case II: } $f=2$. Then $u=u',\ v=v'$ since $\uppi_q^u$ is valid. Then $\# \uppi_{q}^u=\# \uppi_{q-1}^u-1,\ \# \uppi_{q}^v=\# \uppi_{q-1}^v-1,\ $and $\# \uppi_{q}^w=\# \uppi_{q-1}^w, \text{for any }w\ne u,v$. If $h(\uppi_q^u)=h(\uppi_{q-1}^u)$, then ${s(\uppi_{q1}^{v*})}={s(\uppi_{q-1,1}^{v*})}$. Then together with $\uppi_{q1}^{v*}\leq\uppi_{q-1,1}^{v*}-1$ and \eqref{Lemma-Reduce-Procedure-Before-Q0-CaseI-eq2} we have
    $$
      h(\uppi_q^u)-\bigg\lceil\frac{s(\uppi_{q1}^{v*})}{2}\bigg\rceil+\# \uppi_{q1}^{v*}\leq h(\uppi_{q-1}^u)-\bigg\lceil\frac{s(\uppi_{q-1,1}^{v*})}{2}\bigg\rceil+\# \uppi_{q-1,1}^{v*}-1.
    $$
    Then again from \eqref{Lemma-Reduce-Procedure-Before-Q0-CaseI-eq2} for all $w\ne v$, we have $\wt{\chi}(\uppi_q,\rA)\leq \wt{\chi}(\uppi_{q-1},\rA)$. 
\end{proof}
    From Lemmas \ref{Lemma-Reduce-Procedure-Before-Q0} and \ref{Lemma-Reduce-Procedure-After-Q0}, we have
    \bea\nonumber
    \wt{G}(\uppi,\rA)=&O\big(n^{2\# E(\Delta)-\sum_{u\in V(\Delta)}\lceil d^-(u)/2\rceil-\lceil Q_0/2\rceil}\big)\\=&O\big(n^{2\# E(\Delta)-(\# E(\Delta^0)\land (\# E(\Delta)+Q_0)/2)}\big).
    \eea
    Then from \eqref{Final-graphical-representation} we finish the proof.

\subsection{Proofs in Section \ref{section-LSD-SC}}

\subsubsection{Proof of Lemma \ref{tree-independence-semicircle}}

    It is a direct corollary of Corollary \ref{Corollary:tree-independence-coef}, choosing the $f$ as
    $$
    f_\phi(\sigma)=\frac{1}{2}f_\xi(\sigma)+\frac{1}{2}f_\xi(\sigma^{-1}).
    $$

\subsubsection{Proof of Lemma \ref{EPhi(Delta)-semicirclelaw}}

    When $u\ne v$, $\Phi_{uv}^{(n)}=(\Xi_{uv}^{(n)}+\Xi_{vu}^{(n)})/2$. Then
    $$
\E[\mf{\Phi}_n(\overline{\Delta})]=2^{-k}\sum_{\mf{g}}\E[\Xi_{i_1i_2}^{(n,g_1)}\Xi_{i_2i_3}^{(n,g_2)}\cdots\Xi_{i_{k-1}i_k}^{(n,g_{k-1})}\Xi_{i_ki_1}^{(n,g_k)}].
    $$
    Here the summation is taken on all $\mf{g}=(g_1,\cdots,g_k)^\tp,\ g_u\in \{0,1\},\ \text{for any }u\in [k]$, and $\Xi_{ij}^{(n,0)}=\Xi_{ij}^{(n)},\ \Xi_{ij}^{(n,1)}=\Xi_{ji}^{(n)}$. From Proposition \ref{Graph-Estimate-EDelta0}, the summand is $O(n^{-\# E(\overline{\Delta}^0)})$.

\subsection{Proof in Section \ref{Section-CLT}}

\subsubsection{Proof of Lemma \ref{cycle-tree independence}}

        Let $\{\sigma_{uv}:\ (u,v)\in E(G^0)\}$ be a set of deterministic permutations. We then have
{\small
\bea\nonumber
    &\P\l(\bigcap_{(u,v)\in E(G^0)}\big\{R_vR_u^{-1}=\sigma_{uv}\big\}\r)\\
    =&\sum_{\uptau_1}\P\l(\bigcap_{(u,v)\in E(G^0)}\big\{R_v=\sigma_{uv}R_u\big\}\bigg| R_{i_1}=\uptau_1\r)\P(R_{i_1}=\uptau_1)\\
    =&\sum_{\uptau_1}\P\l(\bigcap_{(u,v)\in (E(G^0)\setminus E(T_1^0))}\big\{R_v=\sigma_{uv}R_u\big\}\bigg| R_{i_1}=\uptau_1\r)\P\l(\bigcap_{(u,v)\in E(T_1^0)}\big\{R_v=\sigma_{uv}R_u\big\}\bigg| R_{i_1}=\uptau_1\r)\P(R_{i_1}=\uptau_1)\\
    =&\P\l(\bigcap_{(u,v)\in E(T_1^0)}\big\{R_v=\sigma_{uv}R_u\big\}\r)\sum_{\uptau_1}\P\l(\bigcap_{(u,v)\in (E(G^0)\setminus E(T_1^0))}\big\{R_v=\sigma_{uv}R_u\big\}\bigg| R_{i_1}=\uptau_1\r)\P(R_{i_1}=\uptau_1)\\
    =&\P\l(\bigcap_{(u,v)\in E(T_1^0)}\big\{R_vR_u^{-1}=\sigma_{uv}\big\}\r)\P\l(\bigcap_{(u,v)\in (E(G^0)\setminus E(T_1^0))}\big\{R_vR_u^{-1}=\sigma_{uv}\big\}\r).
    \eea
}
    Then $\ml{R}(T_1^0)$ is independent of $\ml{R}(C^0)\cup\ml{R}(T_2^0)\cup\cdots\cup\ml{R}(T_s^0)$. Repeating this process, $$\{\ml{R}(C^0),\ml{R}(T_1^0),\ml{R}(T_2^0),\cdots,\ml{R}(T_s^0)\}$$ are independent. 

\subsection{Proof in Section \ref{section-power-Qxi2}}\label{Section-aux-power-Qxi2}

\subsubsection{Proof of Lemma \ref{lem:chatterjee-one-row-stability}}

\begin{proof}[Proof of Lemma \ref{lem:chatterjee-one-row-stability}]
For a sequence $z=(z_1,\ldots,z_m)$, write
\[
V(z):=\sum_{k=1}^{m-1}|z_{k+1}-z_k|.
\]
For a no-ties bivariate sample, let $R_X$ and $R_Y$ be the two rank permutations. With
\[
\uppi:=R_Y\circ R_X^{-1},
\]
Chatterjee's statistic has the representation
\[
\xi_n(\mX,\mY)=1-\frac{3}{n^2-1}V(\uppi).
\]
Thus it suffices to bound the change in $V$. Suppose that the modified observation has old $X$-rank $a$, old $Y$-rank $c$, new $X$-rank $b$, and new $Y$-rank $d$. The relative $X$-order and the relative $Y$-order of the remaining $n-1$ observations are unchanged. Hence the new permutation $\uppi'$ is obtained from $\uppi$ by deleting the entry $c$ at position $a$, relabeling the remaining ranks, and then inserting the entry $d$ at position $b$.

Deleting one entry from a sequence with values in $[n]$ changes $V$ by at most $2(n-1)$: if the deleted entry is internal, two adjacent edges are replaced by their shortcut, and the assertion follows from the triangle inequality; if it is an endpoint, the bound is immediate. The same bound holds for inserting one entry, by the reverse argument.

It remains only to consider the relabeling step. After deleting rank $c$, suppose that a remaining rank $v\in [n]\setminus \{c\}$ is mapped to the relabeled rank $f(v)$. Then this map satisfies $|f(v)-v|\leq 1$ for all $v\in[n]\setminus\{c\}$. Therefore, for each adjacent pair $(u_k,u_{k+1})$ in the intermediate sequence,
\[
\bigl||f(u_{k+1})-f(u_k)|-|u_{k+1}-u_k|\bigr|\leq 2.
\]
There are $n-2$ such pairs, so relabeling changes $V$ by at most $2(n-2)$. Combining the three bounds gives
\[
|V(\uppi')-V(\uppi)|
\leq 2(n-1)+2(n-2)+2(n-1)
=6n-8.
\]
Consequently,
\[
|\xi_n(\mX,\mY)-\xi_n(\mX',\mY')|
\leq
\frac{3(6n-8)}{n^2-1}
\leq
\frac{18}{n},
\]
as claimed.
\end{proof}

\subsubsection{Proof of Lemma \ref{lem:gaussian-local-xi}}

\begin{proof}[Proof of Lemma \ref{lem:gaussian-local-xi}]
Let $\xi(\rho)$ denote the population Chatterjee correlation of a bivariate
standard Gaussian vector with Pearson correlation $\rho$. Then by Taylor expansion, as $\rho\to0$,
\[
\xi(\rho)=\frac{\sqrt{3}}{\pi}\rho^2+o(\rho^2);
\]
see for example, \citet[Section 3]{auddy2021exact}. By \citet[Theorem 2.1]{auddy2021exact},
\[
\mu_n(\rho_n)=\xi(\rho_n)+o(n^{-1/2})
\]
uniformly for $\rho$ in a neighborhood of zero. Therefore, if
$\lim \rho_n = 0$ and $\liminf n\rho_n^4>0$, then
\[
\frac{\mu_n(\rho_n)}{\rho_n^2}
=
\frac{\sqrt{3}}{\pi}
+o(1)
+o\left(\frac{1}{\sqrt{n}\rho_n^2}\right)
\to
\frac{\sqrt{3}}{\pi}.
\]  
\end{proof}

{
\bibliographystyle{apalike}
\bibliography{reference}
}

\end{document}